\documentclass{amsart}

\usepackage{time}

\usepackage[margin=1.5in]{geometry}

\usepackage{amscd}
\usepackage{amsmath}
\usepackage{amssymb}
\usepackage{amsthm}
\usepackage{arydshln}
\usepackage{bbm}
\usepackage{bm}
\usepackage{colonequals}
\usepackage{color}
\usepackage{enumitem}
\usepackage{latexsym}
\usepackage{mathrsfs}
\usepackage{mathtools}
\usepackage{stmaryrd}
\usepackage{textcomp}
\usepackage{tikz-cd}
\usepackage{url}
\usepackage{varwidth}
\usepackage[all]{xy}
\usepackage{yhmath}

\usepackage[colorlinks=false,hidelinks]{hyperref}
\usetikzlibrary{decorations.pathmorphing}

\newcommand{\ad}{\mathrm{ad}}

\newcommand{\der}{\mathrm{der}}

\newcommand{\pr}{\mathrm{pr}}

\newcommand{\red}{\mathrm{red}}
\newcommand{\reg}{\mathrm{reg}}

\newcommand{\rss}{\mathrm{rss}}

\newcommand{\ur}{\mathrm{ur}}
\newcommand{\vreg}{\mathrm{vreg}}

\newcommand{\cO}{\mathcal{O}}

\newcommand{\mfp}{\mathfrak{p}}

\newcommand{\x}{\mathbf{x}}

\newcommand{\FF}{\mathbb{F}}

\newcommand{\sF}{\mathscr{F}}
\newcommand{\cF}{\mathcal{F}}
\newcommand{\cG}{\mathcal{G}}

\newcommand{\bbA}{\mathbb{A}}

\newcommand{\bbR}{\mathbb{R}}

\newcommand{\QQ}{\mathbb{Q}}

\newcommand{\bbW}{\mathbb{W}}

\newcommand{\bbG}{\mathbb{G}}

\newcommand{\cH}{\mathcal{H}}

\newcommand{\bbZ}{\mathbb{Z}}

\newcommand{\cL}{\mathcal{L}}
\newcommand{\cK}{\mathcal{K}}

\newcommand{\cB}{\mathcal{B}}

\newcommand{\mfb}{\mathfrak{b}}
\newcommand{\mfu}{\mathfrak{u}}
\newcommand{\mfg}{\mathfrak{g}}
\newcommand{\mfh}{\mathfrak{h}}

\newcommand{\mfi}{\mathfrak{i}}

\newcommand{\mfl}{\mathfrak{l}}
\newcommand{\mft}{\mathfrak{t}}
\newcommand{\mfz}{\mathfrak{z}}

\newcommand{\cY}{\mathcal Y}
\newcommand{\G}{\mathsf{G}}
\newcommand{\T}{\mathsf{T}}

\DeclareMathOperator{\tr}{tr}

\DeclareMathOperator{\Av}{Av}

\DeclareMathOperator{\Ad}{Ad}

\DeclareMathOperator{\Hom}{Hom}

\DeclareMathOperator{\Ind}{Ind}

\DeclareMathOperator{\Lie}{Lie}

\DeclareMathOperator{\Stab}{Stab}

\DeclareMathOperator{\Tr}{Tr}

\DeclareMathOperator{\GL}{GL}

\DeclareMathOperator{\SL}{SL}

\DeclareMathOperator{\height}{ht}

\DeclareMathOperator{\pInd}{pInd}
\DeclareMathOperator{\pRes}{pRes}
\DeclareMathOperator{\For}{For}
\DeclareMathOperator{\fpInd}{{\frak{pInd}}}
\DeclareMathOperator{\fpRes}{{\frak{pRes}}}
\DeclareMathOperator{\FT}{FT}
\DeclareMathOperator{\Infl}{Infl}

\newcommand{\sD}{D}

\newcommand{\from}{\colon}

\theoremstyle{plain}
\newtheorem{thm}{Theorem}[section]
\newtheorem{theorem}[thm]{Theorem}
\newtheorem*{thm*}{Theorem}

\newtheorem{proposition}[thm]{Proposition}

\newtheorem{lemma}[thm]{Lemma}

\newtheorem{corollary}[thm]{Corollary}

\theoremstyle{definition}

\newtheorem{definition}[thm]{Definition}

\newtheorem{claim}{Claim}

\theoremstyle{remark}

\newtheorem*{claim*}{Claim}
\newtheorem{remark}[thm]{Remark}

\theoremstyle{theorem}
\newtheorem{displaytheorem}{Theorem}

\newtheorem*{displayconjecture}{Conjecture}
\newtheorem*{displaytheorem*}{Theorem}

\makeatletter
\newcommand{\dashover}[2][\mathop]{#1{\mathpalette\df@over{{\dashfill}{#2}}}}
\newcommand{\fillover}[2][\mathop]{#1{\mathpalette\df@over{{\solidfill}{#2}}}}
\newcommand{\df@over}[2]{\df@@over#1#2}
\newcommand\df@@over[3]{%
  \vbox{
    \offinterlineskip
    \ialign{##\cr
      #2{#1}\cr
      \noalign{\kern1pt}
      $\m@th#1#3$\cr
    }
  }%
}
\newcommand{\dashfill}[1]{%
  \kern-.5pt
  \xleaders\hbox{\kern.5pt\vrule height.4pt width \dash@width{#1}\kern.5pt}\hfill
  \kern-.5pt
}
\newcommand{\dash@width}[1]{%
  \ifx#1\displaystyle
    2pt
  \else
    \ifx#1\textstyle
      1.5pt
    \else
      \ifx#1\scriptstyle
        1.25pt
      \else
        \ifx#1\scriptscriptstyle
          1pt
        \fi
      \fi
    \fi
  \fi
}
\newcommand{\solidfill}[1]{\leaders\hrule\hfill}
\makeatother

\SelectTips{cm}{11}

\title{Generic character sheaves on parahoric subgroups}

\author{Roman Bezrukavnikov}\address{Department of Mathematics, Massachusetts Institute of Technology, Cambridge, Massachusetts}\email{bezrukav@math.mit.edu}

\author{Charlotte Chan}
\address{Department of Mathematics, University of Michigan, Ann Arbor, Michigan}
\email{charchan@umich.edu}

\begin{document}

\maketitle

\begin{abstract}
 We study parabolic induction producing $\ell$-adic sheaves on a parahoric subgroup scheme in the loop group of a reductive group.
Under a genericity assumption on the input data, we 
  prove that it produces conjugation equivariant perverse sheaves on the parahoric subgroup; this is upgraded to 
 a $t$-exact equivalence of categories of $\ell$-adic sheaves.
  An iterative version of the construction produces such a perverse sheaf starting from a geometric analogue of the data considered by J.-K.\ Yu and J.\ Kim.
  We prove, under a mild condition on $q$, that generic parabolic induction from a parahoric torus realizes the character of the representation arising from the associated parahoric Deligne--Lusztig induction, which is known to parametrize the Fintzen--Kaletha--Spice twist of types.
  In the simplest interesting setting, our construction produces a simple perverse sheaf associated to a sufficiently nontrivial multiplicative local system on a torus, resolving a conjecture of Lusztig. 
\end{abstract}

\setcounter{tocdepth}{1}
\tableofcontents

\newpage

\section{Introduction}

Lusztig's character sheaves on reductive groups \cite{MR0792706} is one of the most important discoveries in the last half-century and has incited many breakthroughs marrying perverse sheaves and representation theory. Among the recent developments is an emerging theory of character sheaves on loop groups, with expected applications to the Langlands program. 
First steps in the direction of depth zero character sheaves appear in the literature starting with \cite{Lus15,Lus14}, see also \cite{BKV15,BV21,Che23,NY25};
the available approach here is based on inducing character sheaves on the reductive quotient of a parahoric subgroup.

This emerging theory should not be limited to the depth zero setting; in a higher depth generalization character sheaves on the reductive quotients must be 
replaced by certain  sheaves on the non-reductive algebraic groups $G_r$ coming from quotients in the Moy--Prasad filtration of parahoric group schemes. In general, it is a very difficult problem to study character sheaves for non-reductive algebraic groups which remains widely open outside of some special cases (unipotent groups \cite{MR3068399,MR3147415}, solvable groups \cite{Des17}).

This paper initiates a theory of character sheaves on $G_r$. In the simplest nontrivial case, our construction produces a simple perverse sheaf associated to a sufficiently regular multiplicative local system, resolving an outstanding conjecture of Lusztig \cite{MR2181813}. We establish that these $\ell$-adic sheaves give the sheaf-theoretic counterpart to positive-depth Deligne--Lusztig induction \cite{Lus04,Sta09,CI21-RT}. Since positive-depth Deligne--Lusztig induction is known to realize $L$-packets of supercuspidal representations \cite{CO21,CO25}, the character sheaves on $G_r$ in this paper provide a necessary ingredient for constructing positive-depth character sheaves on loop groups and utilizing them to study endoscopic character relations for positive-depth supercuspidal representations; it is joint work in progress of the authors with Y.\ Varshavsky to generalize the geometric depth-zero approach of the first author and Varhavsky \cite{BV21,BV21b}.

The algebraic groups $G_r$ arise naturally in the representation theory of $p$-adic groups. For example, it is known by work of Kim \cite{Kim07} and Fintzen \cite{Fin21-Ann} that outside a small collection of primes $p$, that Yu's construction \cite{Yu01} is exhaustive: every supercuspidal representation can be obtained as the compact induction of an irreducible representation of the rational points of some $G_r$. Yu (and more generally Kim--Yu for non-supercuspidal types) produces such irreducible representations from \textit{generic datum}.  

In this paper, we produce a class of $G_r$-equivariant perverse sheaves which are constructed from sheaf-theoretic generic datum. We expect that the functions associated to the sheaves we construct should form a basis for the subspace of class functions spanned by all Kim--Yu types associated to unramified tori. By appealing to Kaletha's Howe factorization, our construction associates a simple $G_r$-equivariant perverse sheaf $\cK_\cL$ to any multiplicative local system $\cL$ with trivial Weyl-group stabilizer; we believe these to be the character sheaf incarnation of the character of regular supercuspidal types \cite{Kal19}. We give evidence for these assertions by establishing explicit compatibility between $\cK_\cL$ and positive-depth supercuspidal $L$-packets in the setting that $\cL$ is ``sufficiently generic.'' This crucially uses positive-depth Deligne--Lusztig induction \cite{Lus04,Sta09,CI21-RT} and the results of the second author and Oi \cite{CO21,CO25} establishing comparisons (under a mild condition on the size of the residue field) to the algebraic constructions of Yu and Kim--Yu, especially in revealing that these geometric constructions obtain the \textit{corrected} parametrization \cite{DS18,Kal19,FKS23}. Our construction therefore allows for the possibility of studying positive-depth supercuspidal $L$-packets using the technology of perverse sheaves.

In the following subsections, we describe the main contributions of this paper. In Section \ref{subsec:intro_lusztig}, we describe Lusztig's conjecture in \cite{MR2181813} on the existence of a particular class of character sheaves on $G_r$ (these are exactly the ones associated to ``sufficiently generic'' $\cL$ mentioned in the preceding paragraph). In Section \ref{subsec:intro_construction}, we describe our more general construction of character sheaves on $G_r$ and outline the proof. In Section \ref{subsec:intro_DL}, we discuss our main theorem bridging $\ell$-adic sheaves and positive-depth Deligne--Lusztig induction.

\subsection{Lusztig's conjecture}\label{subsec:intro_lusztig}

Let $G$ be a connected reductive group over the maximal unramified extension $F^{\ur}$ of a non-archimedean local field $F$. Write $\cO^{\ur}$ for the ring of integers of $F^{\ur}$ and choose a uniformizer $\varpi$. Given a split maximal torus $T \hookrightarrow G$, choose a point $\x$ in the apartment of $T$ and fix a non-negative integer $r$. From the Moy--Prasad filtration associated to $\x$, we have then an associated ``truncated'' parahoric group scheme $G_r$ defined over the residue field $k$ of $F^{\ur}$. (When $r=0$, this is a connected reductive group. When $r>0$ and $\x$ is hyperspecial, then $G_r$ is the $r$th jet scheme of a connected reductive group $\bbG$, and in particular, $G_r(k) = \bbG(\cO^{\ur}/\varpi^{r+1})$.) Choose a Borel subgroup $B \subset G$ containing $T$, consider the associated subgroup schemes $T_r \subset B_r \subset G_r$ with natural projection $\beta \from B_r \to T_r$, and consider the diagram
\begin{equation}\tag{$\star$}\label{eq:pull push}
  \begin{tikzcd}
    & \widetilde G_r \ar{dl}[above left]{f} \ar{dr}{\pi} \\
    T_r && G_r
  \end{tikzcd}
\end{equation}
where
\begin{align*}
  \widetilde G_r \colonequals \{(g,hB_r) \in G_r \times G_r/B_r : h^{-1}gh \in B_r\}, \quad f(g,hB_r) = \beta(h^{-1}gh), \quad \pi(g,hB_r) = g.
\end{align*}

\begin{displayconjecture}[Lusztig]
  Let $r > 0$. If $\cL$ is a sufficiently generic multiplicative local system on $T_r$, then $\pInd_{B_r}^{G_r}(\cL) \colonequals \pi_! f^* \cL$ is an intersection cohomology complex on $G_r$.
\end{displayconjecture}

When $r=0$, this statement is true for any multiplicative local system and is one of the first main results in the theory of character sheaves. The proof crucially relies on the fact that $\pi$ is proper and small. When $r > 0$, we strike out twice: $\pi$ is neither proper nor small.

As evidence, Lusztig proved his conjecture for $G_1(k) = \GL_2(k[t]/t^2)$ \cite{MR2181813}. Later, Lusztig considered the case $G_r(k) = G(k[t]/t^{r+1})$ for $G$ connected reductive and proved his conjecture for $r = 1,3$ (and a weak form for $r=2$) \cite{Lus17}, and proposed the hope that this method could be applied for larger $r$. In this paper, we take a completely different approach. The notion of ``sufficiently generic'' with which we work is equivalent to Lusztig's notion under a mild assumption on $p$ ($p$ not a torsion prime for the dual root datum of $G$); for example, $\GL_n$ the notions are the same for all $p$, but for $\SL_n$, our assumption is stronger when $p=2$. (See Remark \ref{rem:small p} for more comments.)

\begin{displaytheorem*}[\ref{thm:IC vreg}]
  Lusztig's conjecture is true and $\pInd_{B_r}^{G_r}$ is an intersection cohomology complex on the very regular locus of $G_r$.
\end{displaytheorem*}

We obtain this as a special case of a much more general construction of simple $G_r$-equivariant perverse sheaves on $G_r$ for $r > 0$.

\subsection{Character sheaves on $G_r$}\label{subsec:intro_construction}

Our construction is based on input datum that can be viewed as a geometric incarnation of the datum used by Yu \cite{Yu01} to construct tame supercuspidal representations. We start with a triple $(\vec G, \cF, \vec \cL)$ where 
\begin{enumerate}[label=\textbullet]
  \item $\vec G = (G^0, \ldots, G^d)$ is a tower of Levi subgroups $G^0 \subsetneq \cdots \subsetneq G^d = G$ containing $T$
  \item $\cF$ is a simple $G_0^0$-equivariant perverse sheaf on $G_0^0$
  \item $\vec \cL = (\cL_0, \ldots, \cL_d)$ is a sequence of multiplicative local systems such that for $0 \leq i \leq d-1$, $\cL_i$ is a $(G^i,G^{i+1})$-generic sheaf on $G_{r_i}^i$, where $0 < r_0 < \cdots < r_d = r$
\end{enumerate}
The notion of genericity here is essentially the same as in \cite{Yu01} (see Definition \ref{def:generic element}).

\begin{displaytheorem}[\ref{thm:generic character sheaf}]\label{thm:intro K_Psi}
  We can construct a simple $G_r$-equivariant perverse sheaf $\cK_\Psi$ on $G_r$.
\end{displaytheorem}

In particular, Theorem \ref{thm:generic character sheaf} shows that associated to a character sheaf on $G_0^0$ and a sequence of easy ``positive-depth'' data, we obtain an object that should be deserving of the terminology \textit{character sheaf on $G_r$}.

We mention a special case of Theorem \ref{thm:intro K_Psi} which may be of particular interest. When $p$ does not divide the order of the Weyl group of $G$, we may appeal to Kaletha's Howe factorization \cite{Kal19} of \textit{any} multiplicative local system $\cL$ on $T_r$. This associates to $\cL$ a (non-unique) sequence $\vec \cL$ of successively generic multiplicative local systems, and construction yielding Theorem \ref{thm:intro K_Psi} then produces for us a semisimple $G_r$-equivariant perverse sheaf on $G_r$. We prove that this is independent of the choice of Howe factorization and that the resulting perverse sheaf $\cK_\cL$ has a simple description as the intermediate extension of an explicit local system on the \textit{very regular locus} of $G_r$ (see Theorems \ref{thm:IC vreg} and \ref{thm:K_L}).

The construction of $\cK_\Psi$ is inductive, with each step being given by parabolic induction. This mimics the inductive nature of Yu's construction \cite{Yu01}. Let $P \subset G$ be a parabolic subgroup whose Levi component $L$ contains $T$. To a(ny) $(L,G)$-generic element $\psi$, we define associated idempotents $e_\psi \in D_{L_r}(L_r)$ and $f_\psi \in D_{G_r}(G_r)$ with respect to convolution $\star$. These in turn define ``generic'' monoidal subcategories $D_{L_r}^\psi(L_r) \colonequals e_\psi \star D_{L_r}(L_r)$ and $D_{G_r}^\psi(G_r) \colonequals f_\psi \star D_{G_r}(G_r)$. These generic subcategories can be described intrinsically: they are the full subcategories of $D_{L_r}(L_r)$ and $D_{G_r}(G_r)$ consisting of objects those whose Fourier--Deligne transform is locally supported ``above'' $\psi$ and the $G_r$-orbit of $\psi$, respectively  (Lemma \ref{lem:local support}). Here, the Fourier--Deligne transform is taken with respect to the vector bundles $L_r \to L_{r-1}$ and $G_r \to G_{r-1}$.

Using a generalization of \eqref{eq:pull push} wherein $T$ is replaced by $L$, we can define a parabolic induction functor $\pInd_{P_r}^{G_r} \from D_{L_r}(L_r) \to D_{G_r}(G_r)$ which restricts to a functor on generic subcategories. We may now state the main theorem of this paper:

\begin{displaytheorem}[\ref{thm:equivalence}]\label{thm:intro pInd}
  $\pInd_{P_r}^{G_r} \from D_{L_r}^\psi(L_r) \to D_{G_r}^\psi(G_r)$ is a $t$-exact equivalence of categories.
\end{displaytheorem}

The proof of Theorem \ref{thm:intro pInd} is based on two main ingredients: the geometric Mackey formula for generic parabolic induction (Proposition \ref{prop:mackey}) and the relationship between generic parabolic restriction and the Harish-Chandra transform (Proposition \ref{prop:p psi support}). In the positive-depth parahoric setting, the Bruhat decomposition pulls back to a decomposition of $G_r$ indexed by the Weyl group elements $w \in W$ and an affine space (depending which depends on $w$) which sits in the kernel of $G_r \to G_0$. To establish the Mackey formula, we use $(L,G)$-genericity to prove a vanishing statement---that this affine space contributes trivially---therefore reducing the argument to a setting similar to the reductive case. This allows us to prove that $\pInd_{P_r}^{G_r}$ has a left inverse given by parabolic restriction $\pRes_{P_r}^{G_r}$.

The relation to the Harish-Chandra transform turns out to give us the remaining steps of the proof of Theorem \ref{thm:intro pInd}. Consider the quotient map $\phi \from G_r \to G_r/U_{P,r}$. In general, $\pRes_{P_r}^{G_r} = i^*\phi_!$ where $i \from L_r \hookrightarrow G_r/U_{P,r}$. However, on the generic subcategory $D_{G_r}^\psi(G_r)$, we get lucky---the convolution $e_\psi \star \phi_!$ is \textit{already} supported on $L_r$. This therefore gives us an alternative formulation of parabolic restriction, allowing us to complete the proof that $\pInd_{P_r}^{G_r}$ and $\pRes_{P_r}^{G_r}$ are inverse equivalences. The above $!$ formulations of parabolic induction and parabolic restriction can be replaced by $*$ formulations, by taking adjunctions, we see that this gives rise to isomorphic functors on the generic subcategories. In particular, we see that $e_\psi \star \phi_! \cong e_\psi \star \phi_*$ on $D_{G_r}^\psi(G_r)$, which by Artin's theorem ($\phi$ being an affine morphism) implies left and right $t$-exactness of generic parabolic restriction, completing the proof.

Finally, we point out that the intrinsic characterization of the generic subcategories implies that $M \in D_{L_r}^\psi(L_r)$ is simple as an object of $D_{L_r}^\psi(L_r)$ if and only if it is simple as an object of $D_{L_r}(L_r)$ (and the same assertion for $D_{G_r}^\psi(G_r)$). In particular, Theorem \ref{thm:intro pInd} shows that $\pInd_{P_r}^{G_r}$ takes $(L,G)$-generic simple perverse sheaves on $L_r$ to simple perverse sheaves on $G_r$.

\subsection{Relation to positive-depth Deligne--Lusztig induction}\label{subsec:intro_DL}

Having now constructed generic character sheaves on $G_r$, a natural question is how these $\cK_\Psi$ are related to the representation theory of $G_r$. We establish in this paper that $\pInd_{B_r}^{G_r}$ is compatible with the parahoric Deligne--Lusztig induction functor $R_{T_r}^{G_r}$ defined in \cite{CI21-RT}. Assume now that $G$ and $T$ each arise as the base-change of a connected reductive group $\G$ and a torus $\T$ defined over $F$. In this way, we have an associated Frobenius morphism $\sigma$ on $G_r$. For any $\sigma$-equivariant sheaf $M$ on $T_r$ or $G_r$, we may consider the associated trace-of-Frobenius function $\chi_M$. We remark that if $\sigma(B_r) = B_r$ (this is the split case), it is easy to see that $\chi_{\pInd_{B_r}^{G_r}(\cL)}$ realizes the parabolic induction of $\chi_\cL$ in the sense of representation theory; the content here is that, like in the classical $r=0$ setting, geometric parabolic induction also realizes Deligne--Lusztig induction in the ``twisted'' (non-split) setting.

\begin{displaytheorem}[\ref{thm:comparison}]\label{thm:intro comparison}
  If the residue field of $F$ is sufficiently large, then for any $\sigma$-equivariant $(T,G)$-generic multiplicative local system $\cL$ on $T_r$,
  \begin{equation*}
    \chi_{\pInd_{B_r}^{G_r}(\cL)} = (-1)^{\dim G_r} \cdot \Theta_{R_{T_r}^{G_r}(\chi_\cL)},
  \end{equation*}
  where $\Theta_{R_{T_r}^{G_r}(\chi_\cL)}$ denotes the character of the parahoric Deligne--Lusztig induction $R_{T_r}^{G_r}(\chi_\cL)$.
\end{displaytheorem}

When $r=0$, this is (a special case of) Lusztig's theorem on Green's functions \cite{Lus90} (see also \cite{MR1040575} for an exposition on the torus case). Lusztig's proof has two main steps: to prove the assertion up to an unknown scalar, and then to pinpoint the scalar by computing both sides at a convenient point. Our proof of Theorem \ref{thm:intro comparison} follows this strategy; let us now explain the additional ingredients the $r>0$ setting requires in order to achieve this.

To obtain the first step, Lusztig gives an alternative description of $\pInd_{B_0}^{G_0}$ in terms of an arbitrary sequence of Borel subgroups $(B^{(0)}, \ldots, B^{(n)})$ of $G$ such that each $B^{(i)}$ contains $T$ and $B^{(0)} = B^{(n)}$. This is established via an inductive argument on the lengths of the Weyl group elements associated to the associated sequence of relative positions of the Borel subgroups. In the $r>0$ setting, we are also able to give an analogous alternative description of $\pInd_{B_r}^{G_r}$ (Theorem \ref{thm:Borel sequence}), but to prove it, we have an additional complication: two Borel subgroups $B,B' \subset G$ may yield $B_r,B_r'$ which are not $G_r$-conjugate. The proof therefore involves a \textit{double} induction, first to reduce to the $G_r$-conjugate setting, and then to induct on length in the Weyl group; each step crucially relies on $(T,G)$-genericity.

As in \cite{Lus90}, specializing to the special case that $B^{(i)} = \sigma^i(B)$ and $n$ is taken to be a(ny) integer such that $\sigma^n(B) = B$, we obtain a formula (Proposition \ref{prop:Ind Y function}) for $\chi_{\pInd_{B_r}^{G_r}(\cL)}$. While this particular formula is essentially not computable, the key significance is that we can show that $\Theta_{R_{T_r}^{G_r}(\chi_\cL)}$ satisfies the \textit{same} formula, up to a scalar multiple (Proposition \ref{prop:DL formula}). This step has some subtleties because our understanding of parahoric Deligne--Lusztig varieties falls short of that of classical Deligne--Lusztig varieties: while it is known that the virtual representation $R_{T_r}^{G_r}(\chi_\cL)$ is irreducible (up to a sign) as a representation of $G_r^\sigma$ \cite[Theorem 1.1]{CI21-RT}, it is not \textit{a priori} clear whether $R_{T_r}^{G_r}(\chi_\cL)$ remains irreducible (up to a sign) as a representation of $G_r^\sigma \times \langle \sigma^n \rangle$. Surprisingly, we are able to establish this using the character sheaf technology developed in this paper (see Theorem \ref{thm:Frob scalar}).

We come now to the second step---we are in a position where to establish Theorem \ref{thm:intro comparison}, we need only compare the two functions $\chi_{\pInd_{B_r}^{G_r}(\cL)}$ and $\Theta_{R_{T_r}^{G_r}(\chi_\cL)}$ at any single convenient nonvanishing value. Let $G_{r,\vreg}$ denote the locus of very regular elements of $G_r$. It has been known for several decades that character values on such elements takes on a particularly simple form, and in a direction pioneered by Henniart \cite{Hen92} and continued in \cite{CO21, CO23}, these character values are often enough to identify the representation itself. From Theorem \ref{thm:intro pInd}, we can deduce a geometric version of this characterization assertion: $\pInd_{B_r}^{G_r}(\cL)$ is given by the intermediate extension of its restriction to $G_{\vreg}$ (Theorem \ref{thm:IC vreg}). Since the restriction of the two morphisms in \eqref{eq:pull push} to $\pi^{-1}(G_{r,\vreg})$ are both \'etale, we see that $\chi_{\pInd_{B_r}^{G_r}(\cL)}$ also takes a simple form on very regular elements. Matching this with $\Theta_{R_{T_r}^{G_r}(\chi_\cL)}$ (see \cite[Theorem 1.2]{CI21-RT}), and utilizing a simple trick (see \cite[Lemma 9.6]{CO21}) to establish the nonvanishing of these values, then gives us the desired comparison of the  natural decategorifications of $\pInd_{B_r}^{G_r}(\cL)$ and $R_{T_r}^{G_r}(\chi_\cL)$. Of course, this comparison only works if there exists a $\sigma$-fixed very regular element in $T_r$! That this is implied by the hypothesis of Theorem \ref{thm:intro comparison} follows from an argument involving transferring $\T$ to a particular elliptic torus in the quasi-split inner form of $\G$ (see \cite[Proposition 5.8]{CO21}).

To finish the introduction, we mention a corollary of Theorem \ref{thm:intro comparison} which may be of particular interest. Assume $\T$ is elliptic. In joint work of the second author with Oi \cite{CO21}, it is shown that the compact induction of $R_{T_r}^{G_r}(\chi_\cL)$ to $\G(F)$ is irreducible and supercuspidal. (In fact, this theorem relies on the study of very regular elements: we prove that these representations are uniquely determined by their character values on such elements. We remark additionally that the proof requires a yet stronger largeness condition on the residue field of $F$.) This gives a geometric incarnation of a subclass of regular supercuspidal representations \cite{Kal19} and moreover defines a natural parametrization compatible with toral $L$-packets \cite{DS18} (this is subtle; see \cite[Theorem B]{CO21} for more details). It follows then that Theorem \ref{thm:intro comparison} allows for the possibility of studying positive-depth supercuspidal $L$-packets  using the character sheaves constructed in this paper. 

\medbreak
\noindent{\bfseries Acknowledgments.}\quad
A major source of motivation for this project for us was the potential to generalize \cite{BV21,BV21b} to positive-depth supercuspidal $L$-packets. We thank Yakov Varshavsky for continued joint discussions in this direction. We also thank George Lusztig, for inspiration, for comments on an earlier draft, and for telling the second author about his conjecture over lunch at Desi Dhaba in Fall 2019.

\section{Notation}

Let $F$ be a non-archimedean local field and let $F^{\ur}$ denote the maximal unramified extension of $F$. We write $\cO_F$ and $\cO^{\ur}$ for the ring of integers of $F$ and $F^{\ur}$. Write $k_F \cong \FF_q$ and $k$ for the residue fields of $F$ and $F^{\ur}$; note that $k$ is an algebraic closure of $k_F$. Choose a uniformizer $\varpi$ of $F$. 
Let $F$ be a non-archimedean local field and let $F^{\ur}$ denote the maximal unramified extension of $F$. We write $\cO_F$ and $\cO^{\ur}$ for the ring of integers of $F$ and $F^{\ur}$. Write $k_F \cong \FF_q$ and $k$ for the residue fields of $F$ and $F^{\ur}$; note that $k$ is an algebraic closure of $k_F$. Choose a uniformizer $\varpi$ of $F$. 

Let $G$ be a connected reductive group over $F^{\ur}$. We denote by $\Lie G$ the Lie algebra of $G$ and $\Lie^* G$ the dual of $\Lie G$. By Bruhat--Tits theory, to every point $\x$ in the (enlarged) Bruhat--Tits building $\cB(G)$, we may associate a smooth affine $\cO^{\ur}$-group scheme $\cG_{\x,0}$ whose generic fiber is $G$ and whose group of rational points is the parahoric subgroup $G_{\x,0}$. By Yu \cite{Y15}, for every $r \in \widetilde \bbR \colonequals \bbR \sqcup \{r+: r \in \bbR\} \sqcup \{\infty\}$, there exists a smooth affine $\cO^{\ur}$-model $\cG_{\x,r}$ of $G$ such that $\cG_{x,r}(\cO^{\ur})$ is the $r$th Moy--Prasad filtration subgroup \cite{MP94,MP96} of the parahoric subgroup $\cG_{x,0}(\cO^{\ur}) \subset G(F^{\ur})$. We have associated filtrations $\Lie(G)_{\x,r}$ and $\Lie^*(G)_{\x,r}$ which are stable under the adjoint and coadjoint action of $\cG_{\x,0}$, respectively.

Given a split maximal torus $T \hookrightarrow G$, choose a point $\x$ in the apartment of $T$. Fix a positive integer $r > 0$. Following \cite[Section 2.5]{CI21-RT}, we consider the perfectly of finite type smooth affine group scheme $G_{s:r+}$ representing the perfection of the functor
\begin{equation}\label{eq:truncated parahoric}
  R \mapsto \cG_{\x,s}(\bbW(R))/\cG_{\x,r+}(\bbW(R)),
\end{equation}
where $R$ is any $k$-algebra. Here, $\bbW$ denotes the Witt ring associated to $F$ if $F$ has characateristic $0$ and $\bbW(R) = R[\![\varpi]\!]$ if $F$ has positive characteristic. The necessity of passing to the perfection comes from the mixed characteristic setting: it is possible to have $\bbW(R)/p\bbW(R) \neq R$ when $R$ is not a perfect $k$-algebra. When $F$ has positive characteristic, \eqref{eq:truncated parahoric} is already representable by a finite-type smooth affine group scheme. Since the operation of taking perfections preserves \'etale sites \cite[Proposition A.5]{Z17}, choosing the perfect framework is innocuous in the characteristic $p$ setting and allows us to work uniformly for any $F$. We refer to \cite[Appendix A]{Z17} for generalities about perfect schemes, including the set-up of constructible $\ell$-adic \'etale sheaves on perfect algebraic spaces, which we will implicitly use throughout this work. As in \cite[Section 2.6]{CI21-RT}, associated to any closed subgroup scheme $H$ of $G$, we have an associated closed subgroup scheme $H_{s:r+}$ of $G_{s:r+}$. Abusing notation, we define
\begin{equation*}
  G_r \colonequals G_{0:r+}.
\end{equation*}

Let $P$ be a parabolic subgroup of $G$ with Levi decomposition $P = LU_P$ such that $L$ contains $T$. We then have corresponding closed subgroup schemes $T_r \subset L_r \subset P_r \subset G_r$ and $U_{P,r} \subset P_r$. We write
\begin{align*}
  \mfl &\colonequals L_{r:r+} \cong \Lie(L)_{\x,r}/\Lie(L)_{\x,r+}, &
  \mfl^* &= \Lie^*(L)_{\x,-r}/\Lie^*(L)_{\x,(-r)+}, \\
  \mfg &\colonequals G_{r:r+} \cong \Lie(G)_{\x,r}/\Lie(G)_{\x,r+}, & 
  \mfg^* &= \Lie^*(G)_{\x,-r}/\Lie^*(G)_{\x,(-r)+}.
\end{align*}
We also write $\mft \colonequals T_{r:r+}$ and $\mfp \colonequals P_{r:r+}$. If $B$ is a Borel subgroup of $G$ containing $T$, we write $\mfb \colonequals B_{r:r+}$.

Since the induced action of $G_{0+:r+}$ on $\mfg,\mfg^*$ is trivial, there is a natural action of the reductive quotient $G_0$ on $\mfg,\mfg^*$. The natural pairing
\begin{equation}\label{eq:h}
  h \from \frak g^* \times \frak g \to \bbG_a, \qquad (X,Y) \mapsto X(Y) \mod \varpi \cO
\end{equation}
is non-degenerate, $G_r$-equivariant, and symmetric bilinear.

Fix once and for all a nontrivial rank-$1$ local system $\cL$ on $\bbG_a$.

If $X$ is a variety over $\FF_q$ with Frobenius map $\sigma \from X \to X$ endowed with an action of an algebraic group $H$ over $\FF_q$, we denote by $D_H(X)$ the associated equivariant derived category of constructible $\ell$-adic sheaves. For a closed subvariety $Z$ of $X$, we denote by $\delta_Z$ the extension-by-zero sheaf of the constant sheaf on $Z$; whenever we use this notation, $X$ should be clear from the context. If $K \in D(X)$ is a complex with a given isomorphism $\varphi \from \sigma^*K \xrightarrow{\sim} K$, we define the associated trace-of-Frobenius function
\begin{equation*}
  \chi_{K,\varphi} \from X(\FF_q) \to \overline \QQ_\ell, \qquad \chi_{K,\varphi}(x) = \sum_i (-1)^i \Tr(\varphi, \cH^i(K)_x),
\end{equation*}
where $\cH^i(K)_x$ denotes the stalk at $x$ of the $i$th cohomology sheaf $\cH^i(K)$ of $K$.

For a morphism $f \from X \to Y$, sheaf functors such as $f_!$ and $f_*$ are always derived.

\subsection{Brief glossary} \mbox{}

\noindent
  \begin{tabular}{l l}
    $X_\psi$ & $(L,G)$-generic element of depth $r$ \\
    $\cL_\psi,\cF_\psi$ & $(L,G)$-generic idempotents in $D_{L_r}(L_r)$ and $D_{G_r}(G_r)$ \\
    $D_{L_r}^\psi(L_r), D_{G_r}^\psi(G_r)$ & $(L,G)$-generic subcategories of $D_{L_r}(L_r)$ and $D_{G_r}(G_r)$ \\
    $\pInd_{P_r}^{G_r} \from D_{L_r}^\psi(L_r) \to D_{G_r}^\psi(G_r)$ & $(L,G)$-generic parabolic induction \\
    $\pRes_{P_r}^{G_r} \from D_{G_r}^\psi(G_r) \to D_{L_r}^\psi(L_r)$ & $(L,G)$-generic parabolic restriction \\
    $\Psi = (T, \vec G, \x, \vec r, \cF_\rho, \vec \cL)$ & generic datum \\
    $K_{\Psi}$ & simple equivariant perverse sheaf associated to $\Psi$
  \end{tabular}

\section{Definitions}

We collect definitions of general constructions in this section.

\subsection{Convolution and Fourier transform}\label{subsec:FT}

\subsubsection{Convolution}

Let $\mu \from H \times X \to X$ be the morphism associated to the action of an algebraic group $H$ on a variety $X$. Consider the diagram
\begin{equation*}
\begin{tikzcd}
& H \times X \ar{r}{\mu} \ar{dl}[above left]{p_1} \ar{dr}{p_2} & X \\
H && X
\end{tikzcd}
\end{equation*}
and define the corresponding convolution functors: 
\begin{align*}
  \sD(H) \times \sD(X) &\to \sD(X), & (M,N) &\mapsto M \star_! N \colonequals \mu_!((p_1^*M) \otimes (p_2^*N)), \\
  \sD(H) \times \sD(X) &\to \sD(X), & (M,N) &\mapsto M \star_* N \colonequals \mu_*((p_1^*M) \otimes (p_2^*N)).
\end{align*}

\subsubsection{Fourier--Deligne transform}

Recall that we have fixed a nontrivial rank-$1$ local system $\cL$ on $\bbG_a$. Assume $X \to S$ is a vector bundle of constant rank $r \geq 1$ and let $X' \to S$ be the dual vector bundle. Let $h \from X \times X' \to \bbG_a$ be the canonical pairing. Consider the diagram
\begin{equation*}
\begin{tikzcd}
& X \times X' \ar{r}{h} \ar{dl}[above left]{\pr} \ar{dr}{\pr'} & \bbG_a \\
X && X'
\end{tikzcd}
\end{equation*}
and define the associated Fourier--Deligne transform:
\begin{align*}
\FT \from \sD(X) &\to \sD(X'), & M &\mapsto \pr_!'(\pr^*(M) \otimes h^*\cL)[r]. 
\end{align*}

\subsubsection{Free actions}

There is a nice relationship between convolution and Fourier--Deligne transforms (see \cite[Proposition 1.2.2.7]{Lau87}); we make use of a slight variation of this (the same proof as in \textit{op.\ cit.} works). Let $\frak h$ be an affine space of dimension $r$ with an algebraic group structure and assume that $\frak h$ acts freely on $X$ via $\mu \from \frak h \times X \to X$. Then $X \to X/\frak h$ is a vector bundle; let $X' \to X/\frak h$ be the dual.

\begin{lemma}\label{lem:convolution}
  For $M \in D(\frak h)$ and $N \in D(X)$, we have
  \begin{align*}
    \FT(M \star_! N) &\cong \mu'{}^* (\FT(M) \boxtimes \FT(N))[-r], \\
    \FT(M \star_* N) &\cong \mu'{}^! (\FT(M) \boxtimes \FT(N)),
  \end{align*}
  where $\mu' \from X' \to \frak h^* \times X'$ is the transpose of the action map $\mu \from \frak h \times X \to X$.
\end{lemma}

A corollary of this is the following lemma:

\begin{lemma}\label{lem:local support}
  Let $U$ be an open subvariety of the quotient $X/\mfh$ over which the vector bundle $X \xrightarrow{\pi} X/\mfh$ trivializes via $\varphi_U \from \mfh \times U \xrightarrow{\cong} \pi^{-1}(U)$. Let $\varphi_U' \from \mfh^* \times U \xrightarrow{\cong} \pi'{}^{-1}(U)$ be the corresponding trivialization map for $X' \xrightarrow{\pi'} X/\mfh$. 
  \begin{enumerate}
    \item For any $M \in D(X)$ and any closed subvariety $\mfz^*$ of $\mfh^*$, the two sheaves $\FT(\FT(\delta_{\mfz^*}) \star_! M)|_{\pi'{}^{-1}(U)}$ and $\FT(\FT(\delta_{\mfz^*}) \star_* M)|_{\pi'{}^{-1}(U)}$ are supported on $\varphi_U'(-\mfz^* \times U)$.
    \item Assume that $M \in D(X)$ is such that $\FT(M)|_{\pi'{}^{-1}(U)}$ is supported on $\varphi_U'(-\mfz^* \times U)$ for any $U$ as above. The $M \cong \FT(\delta_{\mfz^*}) \star_! M[2r]$ and $M \cong \FT(\delta_{\mfz^*}) \star_* M$.
  \end{enumerate}
\end{lemma}

\begin{proposition}\label{prop:free idempotent}
  Let $\mfz^*$ be a closed subvariety of $\mfh^*$. Then $\FT(\delta_{\mfz^*}) \star_! D(X) = \FT(\delta_{\mfz^*}) \star_* D(X)$ and is a full subcategory compatible with the perverse $t$-structure.
\end{proposition}

\begin{proof}
  Both $\FT(\delta_{\mfz^*}) \star_! D(X)$ and $\FT(\delta_{\mfz^*}) \star_* D(X)$ consist of objects $M \in D(X)$ such that---using the same notation as in Lemma \ref{lem:local support}---for any $U$ over which the vector bundle $X \xrightarrow{\pi} X/\mfh$ trivializes, restriction $\FT(M)|_{\varphi^{\prime-1}(U)}$ of the Fourier--Deligne transform is supported on $\varphi_U'(-\mfz^* \times U)$. This shows the equality $\FT(\delta_{\mfz^*}) \star_! D(X) = \FT(\delta_{\mfz^*}) \star_* D(X)$. Furthermore, this description implies that $M \in D(X)$ lies in $\FT(\delta_{\mfz^*}) \star_! D(X)$ if and only if its perverse cohomology sheaves do and this happens if and only if each simple subquotient of each perverse cohomology sheaf does.
\end{proof}

Proposition \ref{prop:free idempotent} implies that the simple (resp.\ perverse) objects in $\FT(\delta_{\mathfrak z^*}) \star_! D(X)$ are exactly the simple (resp.\ perverse) objects in $D(X)$ which lie in $\FT(\delta_{\mathfrak z^*}) \star_! D(X)$.

\subsection{Averaging functors}\label{subsec:Av}

Let $G$ be an algebraic group, let $H \subset G$ be a closed subgroup, and let $X$ be an $H$-variety. We define the induction space $G \times^H X$ to be the geometric quotient of $G \times X$ by the $H$-action $h \cdot (g,x) \mapsto (gh^{-1},h \cdot x)$. We have a forgetful functor $\For_H^G \from D_G(X) \to D_H(X)$ and we denote by $\Av_{H!}^G$ and $\Av_{H*}^G$ the left and right adjoints of $\For_H^G$.  Explicitly,
the functor $\For_H^G$ isomorphic to the composition 
\begin{equation*}
  D_G(X) \xrightarrow[\sigma_X^*]{\sigma_X^!} D_G(G \times^H X) \xrightarrow[\iota_X^* \circ \For_H^G]{\iota_X^! \circ \For_H^G} D_H(X)
\end{equation*}
where $\iota_X \from X \to G \times^H X$ is defined by $x \mapsto (e,x)$ and $\sigma_X \from G \times^H X \to X$ is induced by the action map $G \times X \to X$. We then see that:
\begin{align*}
  &\Av_{H!}^G \from D_H(X) \xrightarrow{(\iota_X^! \circ \For_H^G)^{-1}} D_G(G \times^H X) \xrightarrow{\sigma_{X!}} D_H(X) \\
  &\Av_{H*}^G \from D_H(X) \xrightarrow{(\iota_X^* \circ \For_H^G)^{-1}} D_G(G \times^H X) \xrightarrow{\sigma_{X*}} D_H(X)
\end{align*}
Recall that $\iota_X^! \circ \For_H^G \cong \iota_X^* \circ \For_H^G[-2\dim G/H](-\dim G/H).$

We denote the forgetful functor $D_H(X) \to D(X)$ simply by $\For$, noting that $H$ and $X$ should be clear from the context.

\subsection{Parabolic induction and parabolic restriction}\label{subsec:Ind and Res}

Let $P \hookrightarrow G$ be a parabolic subgroup of $G$ whose Levi component $L$ contains $T$. Consider the associated subgroup schemes $L_r$ and $P_r$ in $G_r$ and consider the inclusion map $i \from P_r \hookrightarrow G_r$ and the natural surjection $p \from P_r \to L_r$. Write $U_{P,r}$ for the subgroup scheme of $G_r$ associated to the unipotent radical $U_P$ of $P$; note $\ker(p) = U_{P,r}$.

\begin{definition}[parabolic induction and parabolic restriction]\label{def:pInd pRes}
  We define parabolic induction and parabolic restriction functors as
  \begin{align*}
    \pInd_{P_r!}^{G_r} &\colonequals \Av_{P_r!}^{G_r}{} \circ i_! \circ \Infl_{L_r}^{P_r} {} \circ p^* \\
    \pInd_{P_r*}^{G_r} &\colonequals \Av_{P_r*}^{G_r}{} \circ i_* \circ \Infl_{L_r}^{P_r} {} \circ p^! \\
    \pRes_{P_r!} &\colonequals p_! \circ i^* \circ \For_{L_r}^{G_r} \\
    \pRes_{P_r*}^{G_r} &\colonequals p_* \circ i^! \circ \For_{L_r}^{G_r}
  \end{align*}
  These define functors
  \begin{align*}
    \pInd_{P_r!}^{G_r},\pInd_{P_r*}^{G_r} \from D_{L_r}(L_r) &\to D_{G_r}(G_r), \\
    \pRes_{P_r!}^{G_r}, \pRes_{P_r*}^{G_r} \from D_{G_r}(G_r) &\to D_{L_r}(L_r).
  \end{align*}
\end{definition}

\begin{remark}
  We could alternatively define these functors in the language of stacks: Consider the correspondence
  \begin{equation*}
    \begin{tikzcd}
      & P_r/P_r \ar{dl}[above left]{p} \ar{dr}{q} \\
      L_r/L_r && G_r/G_r
    \end{tikzcd}
  \end{equation*}
  where each quotient is under the conjugation action. Then
  \begin{align*}
    \pInd_{B_r!}^{G_r} &= q_! \circ p^*, &
    \pInd_{B_r*}^{G_r} &= q_* \circ p^!, \\
    \pRes_{B_r!}^{G_r} &= p_! \circ q^*, &
    \pRes_{B_r*}^{G_r} &= p_* \circ q^!.
  \end{align*}
\end{remark}

\begin{lemma}\label{lem:adjointness}
  $\pRes_{P_r!}^{G_r}$ is left adjoint to $\pInd_{P_r*}^{G_r}$ and $\pRes_{P_r*}^{G_r}$ is right adjoint to $\pInd_{P_r!}^{G_r}$.
\end{lemma}

\begin{proof}
  We have adjointness relations
  \begin{equation*}
    (\Av_{H!}^{G}, \For_{H}^{G}, \Av_{H*}^{G}), \qquad (i^*, i_*=i_!, i^!), \qquad (p^*, p_*), \qquad (p_!, p^!).
  \end{equation*}
  As for the inflation functor $\Infl_{L_r}^{P_r}$, we note that since $P_r = L_r \ltimes U_{P,r}$, this functor is an equivalence of categories with inverse equivalence $\For_{L_r}^{P_r}$. The desired adjointness assertions follow.
\end{proof}

It will be useful to have a non-equivariant description of the parabolic induction functors.

\begin{definition}[$f$,  $\pi$, and $\alpha$]\label{def:f and pi}
  Define
  \begin{align*}
    \widetilde G_r &\colonequals \{(g,hP_r) \in G_r \times G_r/P_r : h^{-1} g h \in P_r\}, \\
    \widehat G_r &\colonequals \{(g, h) \in G_r \times G_r : h^{-1} g h \in P_r\},
  \end{align*}
  and consider the morphisms
  \begin{align*}
    f \from \widehat G_r &\to L_r, & (g, h) &\mapsto p(h^{-1}gh), \\
    \pi \from \widetilde G_r &\to G_r, & (g, hP_r) &\mapsto g, \\
    \alpha \from \widehat G_r &\to \widetilde G_r, & (g,h) &\mapsto (g,hP_r).
  \end{align*}
\end{definition}

\begin{lemma}\label{lem:pInd}
  Let $n = \dim U_{P,r}$. We have commutative diagrams
  \begin{equation*}
    \begin{tikzcd}
      D_{L_r}(L_r) \ar{rrr}{\pInd_{P_r!}^{G_r}} \ar{d}[left]{\For} &&& D_{G_r}(G_r) \ar{d}{\For} \\
      D(L_r) \ar{rrr}{M \mapsto \pi_! \widetilde{f^* M}[2n]} &&& D(G_r)
    \end{tikzcd}
    \qquad
    \begin{tikzcd}
      D_{L_r}(L_r) \ar{rrr}{\pInd_{P_r*}^{G_r}} \ar{d}[left]{\For} &&& D_{G_r}(G_r) \ar{d}[left]{\For} \\
      D(L_r) \ar{rrr}{M \mapsto \pi_* \widetilde{f^! M}[-2n]} &&& D(G_r)
    \end{tikzcd}
  \end{equation*}
  where $\widetilde{f^* M}$ is the unique object in $D(\widetilde G_r)$ satisfying $\alpha^* \widetilde{f^* M} \cong f^*M$.
\end{lemma}

We finish this section by defining analogous functors for sheaves on the $\mfl$ and $\mfg$. Abusing notation, we again write $i \from \mfp \hookrightarrow \mfg$ and $p \from \mfp \to \mfl$. Then the same formulas as in Definition \eqref{def:pInd pRes} define parabolic induction $\fpInd_{P_r!}^{G_r}, \fpInd_{P_r*}^{G_r} \from D_{L_r}(\mft) \to D_{G_r}(\mfg)$ and parabolic restriction functors $\fpRes_{P_r!}^{G_r}, \fpRes_{P_r*}^{G_r}\from D_{G_r}(\mfg) \to D_{L_r}(\mft)$:
  \begin{align*}
    \fpInd_{P_r!}^{G_r} &\colonequals \Av_{P_r!}^{G_r}{} \circ i_! \circ \Infl_{L_r}^{P_r} \circ p^* \\
    \fpInd_{P_r*}^{G_r} &\colonequals \Av_{P_r*}^{G_r}{} \circ i_* \circ \Infl_{L_r}^{P_r} \circ p^! \\
    \fpRes_{P_r!}^{G_r} &\colonequals p_! \circ i^* \circ \For_{P_r}^{G_r} \\
    \fpRes_{P_r*}^{G_r} &\colonequals p_* \circ i^! \circ \For_{P_r}^{G_r}
  \end{align*}
  Define
  \begin{align*}
    \widetilde \mfg &\colonequals \{(g,hP_r) \in \mfg \times G_r/P_r : \Ad(h^{-1})(g) \in \mfp\}, \\
    \widehat \mfg &\colonequals \{(g,h) \in \mfg \times G_r : \Ad(h^{-1})(g) \in \mfp\},
  \end{align*}
  and consider the morphisms
  \begin{align*}
    f \from \widehat \mfg &\to \mft &
    (g,h) &\mapsto p(\Ad(h^{-1})(g)), \\
    \pi \from \widetilde \mfg &\to \mfg &
    (g,hP_r) &\mapsto g, \\
    \alpha \from \widehat \mfg &\to \widetilde \mfg & (g,h) &\mapsto \Ad(h^{-1})(g).
  \end{align*}
  As in Lemma \ref{lem:pInd}, for any $M \in D_{L_r}(\mfl),$
  \begin{align*}
    &\For(\fpInd_{P_r!}^{G_r}(M)) \cong \pi_! \widetilde{f^* \For(M)}[2\dim U_{P,r}], \\
    &\For(\fpInd_{P_r*}^{G_r}(M)) \cong \pi_* \widetilde{f^! \For(M)}[-2\dim U_{P,r}].
  \end{align*}

\begin{remark}\label{rem:fpInd}
  Note that since the $G_r$-action on $\mfg$ factors through $G_0$, we have $D_{G_r}(\mfg) \cong D_{G_0}(\mfg)$. Analogously to above, we may define parabolic induction and parabolic restriction functors with respect to $L_0, P_0, G_0$. By the projection formula, it is not hard to see that
  \begin{equation*}
    \mathfrak{pInd}_{P_r!}^{G_r}(M) = \mathfrak{pInd}_{P_0!}^{G_0}(M).
  \end{equation*}
\end{remark}

\section{Generic idempotents}\label{sec:idempotents}

We will study two closely related notions of genericity.

\begin{definition}[generic elements]\label{def:generic element}
  Let $X \in \mfl^*$ be fixed the coadjoint action of $L$ and consider the following two properties.
  \begin{enumerate}
    \item[$\frak{ge1}$] $X|_{\mft_\alpha} \not\equiv 0$ for all $\alpha \in \Phi(G,T) \smallsetminus \Phi(L,T).$
    \item[$\frak{ge2}$] The stabilizer of $X|_{\mft}$ in the Weyl group of $G$ is the Weyl group of $L$.
  \end{enumerate}   
  We say that $X$ is \textit{$(L,G)$-generic} if $X$ satisfies both $\mathfrak{ge1}$ and $\mathfrak{ge2}$.
\end{definition}

The two conditions $\mathfrak{ge1}$ and $\mathfrak{ge2}$ are the ``at $F^{\ur}$-level'' versions of the genericity notions foundational in Yu's construction of supercuspidal types \cite{Yu01} and of Kim--Yu types \cite{KY17} (see GE1 and GE2 of \cite[Section 8]{Yu01}). Yu proves (see Lemma 8.1 of \textit{op.\ cit.}) that if $p$ is not a torsion prime for the dual root datum of $G$, then $\mathfrak{ge1}$ implies $\mathfrak{ge2}$.

\mbox{}

For the rest of the paper, let $X_\psi$ be a $(L,G)$-generic element of $\frak l^*$. Note that this guarantees that the orbit $\bbG(X_\psi)$ of $X_\psi$ under the coadjoint action of $G_0$ on $\mfg^*$ is closed.

\begin{definition}[generic idempotents]\label{def:generic idempotent}
  Define
  \begin{equation*}
    \cL_\psi \colonequals \FT(\delta_{X_\psi}) \in D(\mfl), \qquad \cF_\psi \colonequals \FT(\delta_{\bbG(X_\psi)}) \in D(\mfg).
  \end{equation*}
  Abusing notation, we write $\mfi$ for both inclusions $\mfl \hookrightarrow L_r$ and $\mfg \hookrightarrow G_r$. Define
  \begin{equation*}
    \cL_{\psi,r} \colonequals \mfi_! \cL_\psi[\dim \mfl]
    \in D_{L_r}(L_r), \qquad 
    \cF_{\psi,r} \colonequals \mfi_! \cF_\psi[\dim \mfg]
    \in D_{G_r}(G_r)
  \end{equation*}
  to be the \textit{$(L,G)$-generic idempotents} associated to $X_\psi$. 
\end{definition}

\begin{lemma}\label{lem:idempotent}
  $\cL_{\psi,r}$ and $\cF_{\psi,r}$ are both idempotents with respect to  $\star_!$ and $\star_*$.
\end{lemma}

\begin{proof}
  Since $\mfi$ is a group homomorphism, it follows from the definitions that
  \begin{equation*}
    \mfi_!M \star_! \mfi_! M \cong \mfi_!(M \star_! M), \qquad \text{for any $M \in D(\mfl)$.}
  \end{equation*}
  Hence to see that $\mfi_! \cL_\psi[\dim \mft](\dim \mft)$ is an idempotent in $D_{L_r}(L_r)$, it is equivalent to show that $\cL_\psi[\dim \mft](\dim \mft)$ is an idempotent in $D(\mfl)$. To this end, we have $\FT(\cL_\psi) \cong \FT(\FT(\delta_{X_\psi})) \cong \delta_{-X_\psi}(-\dim \mft)$. By Lemma \ref{lem:convolution}, we have 
  \begin{align*}
    \FT(\cL_\psi[\dim \mft] \star_! \cL_\psi[\dim \mft])
    &\cong (\FT(\cL_\psi)[\dim \mft]\otimes \FT(\cL_\psi)[\dim \mft])[-\dim \mft] \\
    &\cong (\delta_{-X_\psi} \otimes \delta_{-X_\psi})[\dim \mft] 
    \cong \delta_{-X_\psi}[\dim \mft] \cong \FT(\cL_\psi[\dim \mft]).
  \end{align*}
  Since $\FT$ is an equivalence of categories, it follows that $\cL_{\psi,r} \star_! \cL_{\psi,r} \cong \cL_{\psi,r}$, which proves the lemma. The argument for $\cF_{\psi,r}$ is similar.
\end{proof}

We will also make use of the following orthogonality result.

\begin{lemma}\label{lem:idempotent orthog}
  Let $X_{\psi'}$ be a $(L',G)$-generic element of $(\mfl')^*$ for a Levi subgroup $L'$ of $G$ and write $\cF_{\psi'} \colonequals \FT(\delta_{\bbG(X_{\psi'} + \mfu')}) \in D(\mfg)$. If $X_{\psi'} \notin \bbG(X_{\psi})$, then
  \begin{equation*}
    \mfi_! \cF_{\psi} \star_! \mfi_! \cF_{\psi'} = \mfi_! \cF_{\psi} \star_* \mfi_! \cF_{\psi'} = 0.
  \end{equation*}
\end{lemma}

\begin{proof}
  As in the previous lemma, we see that lemma is equivalent to showing the vanishing of $\cF_{\psi} \star_! \cF_{\psi'}$. By construction, we have $\FT(\cF_\psi) \cong \delta_{-\bbG(X_\psi + \mfu)}$ and $\FT(\cF_{\psi'}) \cong \delta_{-\bbG(X_\psi' + \mfu')}$. Since $X_{\psi'} \notin \bbG(X_\psi)$ by assumption, necessarily the orbits $\bbG(X_\psi + \mfu)$ and $\bbG(X_{\psi'} + \mfu')$ are disjoint, so that $\delta_{-\bbG(X_\psi) + \mfu} \otimes \delta_{-\bbG(X_{\psi'} + \mfu')} = 0$. The conclusion follows by Lemma \ref{lem:convolution}.
\end{proof}

\begin{lemma}\label{lem:FT pInd}
  We have
  \begin{equation*}
    \mathfrak{pInd}_{P_r!}^{G_r}(\cL_\psi[\dim \mfl]) \cong \cF_\psi[\dim \mfg].
  \end{equation*}
\end{lemma}

\begin{proof}
  By Remark \ref{rem:fpInd} we have $\mathfrak{pInd}_{P_r!}^{G_r} \cong \mathfrak{pInd}_{P_0!}^{G_0}$, so this reduces to the classical $r=0$ setting. It is well known 
  that $\FT$ commutes with $\fpInd_{P_0}^{G_0}$ (see for example \cite[Coda 11.3]{B86}, \cite[Proposition 13.6]{KW01}), and so since $\cL_\psi = \FT(\delta_{X_\psi})$, we may conclude that 
  \begin{equation}\label{eq:FT pInd}
    \FT(\fpInd_{P_0}^{G_0}(\cL_\psi)) \cong \fpInd_{P_0}^{G_0}(\delta_{-X_\psi}) \cong \delta_{-\bbG(X_\psi + \mfu)}[2\dim U_{P,0}].  
  \end{equation}
  Therefore  $\fpInd_{P_0}^{G_0}(\cL_\psi) \cong \FT(\delta_{\bbG(X_\psi + \mfu)})[2\dim U_{P,0}]
  \cong \cF_\psi[2\dim U_{P,0}]
  $. Since $\dim U_{P,0} = \dim \mfu$ and $\dim \mfg = \dim \mfl + 2 \dim \mfu$, the desired conclusion now follows.
\end{proof}

\begin{definition}[$(L,G)$-generic subcategories]\label{def:psi subcats}
  Associated to $X_\psi$ we have \textit{$(L,G)$-generic subcategories} of $D_{L_r}(L_r)$ and $D_{G_r}(G_r)$: 
  \begin{align*}
    D_{L_r}^\psi(L_r) &\colonequals \cL_{\psi,r}\star_! D_{L_r}(L_r) = \cL_{\psi,r} \star_* D_{L_r}(L_r), \\
    D_{G_r}^\psi(G_r) &\colonequals \cF_{\psi,r} \star_! D_{G_r}(G_r) = \cF_{\psi,r} \star_* D_{G_r}(G_r),
  \end{align*}
  where the two second equalities hold by Proposition \ref{prop:free idempotent}. 
\end{definition}

The following lemma describes the behavior of objects in the generic subcategory $D_{L_r}^\psi(L_r)$; its proof is standard and we omit it.

\begin{lemma}\label{lem:L psi}\mbox{}
  \begin{enumerate}[label=(\alph*)]
    \item $\cL_\psi$ is a multiplicative local system on $\mfl$.
    \item Any $M \in D_{L_r}^\psi(L_r)$ is $(\mfl,\cL_\psi)$-equivariant.
  \end{enumerate}
\end{lemma}

We next use Proposition \ref{prop:free idempotent} to describe the role of the parabolic induction functors on $(L,G)$-generic subcategories.

\begin{proposition}\label{prop:generic induction}
  Parabolic induction restricts to functors on $(L,G)$-generic subcategories:
  \begin{equation*}
    \pInd_{P_r!}^{G_r}, \pInd_{P_r*}^{G_r} \from D_{L_r}^\psi(L_r) \to D_{G_r}^\psi(G_r).
  \end{equation*}
\end{proposition}

\begin{proof}
  By Proposition \ref{prop:free idempotent} in this setting, we wish to prove that for any $M \in D_{L_r}^\psi(L_r)$, the parabolic induction $\pInd_{P_r!}^{G_r}(M)$ satisfies the following property:
  \begin{quote}
    Let $U$ be any open of $G_r/\mfg$ over which the vector bundle $G_r \xrightarrow{\pi_G} G_r/\mfg$ trivializes, let $G_r' \xrightarrow{\pi_G'} G_r/\mfg$ denote the dual vector bundle to $\pi_G$, and $\varphi_U' \from \mfg^* \times U \xrightarrow{\cong} \pi_G^{\prime-1}(U)$ be a trivialization over $U$. Then $\FT(\pInd_{P_r!}^{G_r}(M))|_{\pi_G^{\prime-1}(U)}$ is supported on $\varphi_U'(-\bbG(X_\psi + \mfu) \times U)$.
  \end{quote}
  To this end, let $U$ be as in the above quote. By definition, our goal is to show that for the maps in the picture
  \begin{equation*}
    \begin{tikzcd}
      & \widehat G_r \ar{dl}[above left]{f} \ar{dr}{\alpha} &&&& \bbG_a \\
      L_r && \widetilde G_r  \ar{dr}{\pi} && G_r \times G_r' \ar{dl}[above left]{\pr} \ar{ur}{h} \ar{dr}{\pr'} \\
      &&& G_r && G_r' & \pi_G^{\prime-1}(U) \ar{l}{i_U'}
    \end{tikzcd}
  \end{equation*}
  we have that 
  \begin{equation}\label{eq:support claim}
    \text{$i_U'{}^* \pr_!' (\pr^* \pi_! \widetilde{f^* M} \otimes h^* \cL)$ is supported on $\varphi_U'(-\bbG(X_\psi + \mfu) \times U)$,}  
  \end{equation}
  where, as usual, $\widetilde{f^* M}$ denotes the unique object on $\widetilde G_r$ such that $\alpha^* \widetilde{f^* M} = f^* M$. Applying base change to the diagram
  \begin{equation*}
    \begin{tikzcd}
      \widetilde G_r \times G_r' \ar{r}{\pi_{12}} \ar{d}[left]{\pi_{13}} & \widetilde G_r \ar{d}{\pi} \\
      G_r \times G_r' \ar{r}{\pr} & G_r
    \end{tikzcd}
  \end{equation*}
  we obtain that
  \begin{align*}
    i_U'{}^* \pr_!' (\pr^* \pi_! \widetilde{f^* M} \otimes h^* \cL)
    &\cong i_U'{}^* \pr_!' (\pi_{13!} \pi_{12}^* \widetilde{f^* M} \otimes h^* \cL) \\
    &\cong i_U'{}^* \pr_!' \pi_{13!} (\pi_{12}^* \widetilde{f^* M} \otimes \pi_{13}^* h^* \cL)
  \end{align*}
  where the second isomorphism holds by the projection formula. Base-changeing once again, we get that the above expression is isomorphic to $\pr_!' \pi_{13!}(\pi_{12}^* \widetilde{f^* M} \otimes \pi_{13}^* h^* \cL)$, where the maps $\pi_{12}, \pi_{13}$ now denote the restrictions to the subvariety $\{(g,hP_r,g') \in \widetilde G_r : g' \in \varphi_U'(\mfg^* \times U)\} \subset \widetilde G_r \times G_r'$.

  Fix a $g' = \varphi_U'(X,\bar g') \in \varphi_U'(\mfg^* \times U)$ and consider the map $\pi_r$ induced by  projection in the first coordinate via $G_r \to G_r/\mfg$:
  \begin{equation*}
    \pi_r \from (\pr' \circ \pi_{13})^{-1}(g') \to \{(\bar g,hP_r) \in G_r/\mfg \times G_r/P_r\}.
  \end{equation*}
  Given $(g,hP_r,g') \in \pi_r^{-1}(\bar g,hP_r)$, we obtain an isomorphism
  \begin{equation*}
    \pi_r^{-1}(\bar g,hP_r) \cong \{Y \in \mfg : \ad(h)(Y) \in \mfp\}
  \end{equation*}
  given by $(\mu(Y,g),hP_r,g') \mapsfrom Y$. Using this isomorphism together with the fact that $M$ is $(\mfl,\cL_\psi)$-equivariant (Lemma \ref{lem:L psi}), we see that the stalk of $\pi_{r!}(\pi_{12}^* \widetilde{f^* M} \otimes \pi_{13}^* h^* \cL)$ over $(\bar g,hP_r, g')$ is zero if and only if the stalk of $\pi_{2!}(\pi_{12} \widetilde{f^* \cL_\psi} \otimes \pi_{13}^* h^* \cL)$ over the chosen $X \in \mfg^*$ is zero. Here, $\pi_2 \from \mfg \times \mfg^* \to \mfg^*$ denotes the second projection. Again by base-change, we have
  \begin{equation*}
    \pi_{2!}(\pi_{12} \widetilde{f^* \cL_\psi} \otimes \pi_{13}^* h^* \cL) \cong \FT(\fpInd_{P_r!}^{G_r}(\cL_\psi)).
  \end{equation*}
  But now by Lemma \ref{lem:FT pInd} (especially \eqref{eq:FT pInd}), this implies that the stalk of $\pi_{r!}(\pi_{12}^* \widetilde{f^* M} \otimes \pi_{13}^* h^* \cL)$ over $(\bar g,hP_r, g')$ vanishes if $X \notin \bbG(-X_\psi + \mfu)$, which establishes \eqref{eq:support claim}. 
\end{proof}

\section{Generic parabolic induction}\label{sec:Ind t-exact equivalence}

In this section we prove the main theorem of this paper: $(L,G)$-generic parabolic induction $\pInd_{P_r!}^{G_r} \from D_{L_r}^\psi(L_r) \to D_{G_r}^\psi(G_r)$ is a $t$-exact equivalence of categories (Theorem \ref{thm:equivalence}). We record some general lemmas in Section \ref{subsec:soft lemmas}. In our $r>0$ setting, a common general strategy to establishing a desired statement is to find some locus on which the assertion is true and prove a vanishing statement outside of this locus. In Section \ref{subsec:hard lemmas} we collect and prove several lemmas we will later use to prove vanishing assertions used to establish the Mackey formula (Proposition \ref{prop:mackey}) in Section \ref{subsec:mackey} and study the Harish-Chandra transform on $D_{G_r}^\psi(G_r)$ (Proposition \ref{prop:p psi support}) in Section \ref{subsec:harishchandra}. 
These are the key propositions to the proof of one of the main theorems of this paper: Theorem \ref{thm:equivalence}.

\subsection{Basic properties}\label{subsec:soft lemmas}

Let $\phi \from G_r \to G_r/U_{P,r}$. We record several general results we will later need. Each of these lemmas is well known in the $r = 0$ case  (for example, see \cite[Theorem 3.6(a)]{MV88} for Lemma \ref{lem:Av phi}, \cite[Lemma 4.3]{Che23} for Lemmas \ref{lem:monoidal p} and \ref{lem:CH HC}) and the proof in our $r>0$ setting is the same. 

\begin{lemma}\label{lem:monoidal p}
  The functor $D_{G_r}(G_r) \to D_{P_r}(G_r/U_{P,r}) \cong D_{L_r}(U_{P,r} \backslash G_r / U_{P,r})$ induced by $\phi_!$ (resp.\ $\phi_*$) is monoidal with respect to $\star_!$ (resp.\ $\star_*$). 
\end{lemma}

\begin{lemma}\label{lem:Av phi}
  For any $N \in D_{G_r}(G_r)$, we have
  \begin{equation*}
    \Av_{P_r!}^{G_r}(\phi^* \phi_! N) \cong N \star_! \Av_{P_r!}^{G_r}(\delta_{U_{P,r}}).
  \end{equation*}
\end{lemma}

\begin{lemma}\label{lem:CH HC}
  For any $M \in D_{L_r}(L_r)$ and  $N \in D_{G_r}(G_r)$ such that $M \star_! \phi_! N$ is supported on $P_r/U_{P,r}$, we have
  \begin{equation*}
    \pInd_{P_r!}^{G_r}(M \star_! \phi_! N) \cong \pInd_{P_r!}^{G_r}(M) \star_! N.
  \end{equation*}
\end{lemma}

\subsection{Vanishing lemmas}\label{subsec:hard lemmas}

\begin{lemma}\label{lem:equiv vanishing}
  Let $X$ be a space equipped with a free action of an algebraic group $H$. If $\cL_H$ is a nontrivial multiplicative local system on $H$, then for any $M \in D(X)$ which is $(H,\cL_H)$-equivariant, then $R\Gamma_c(X,M) = 0$ and $R\Gamma(X,M) = 0$.
\end{lemma}

\begin{proof}
  Consider the quotient map $q \from X \to X/H$. To prove that $R\Gamma_c(X,M) = 0$, it suffices to prove that over any $y \in X/H$, we have $R\Gamma_c(q^{-1}(y),M) = 0$. We have have $q^{-1}(y) \cong \{y\} \times H$. Under this isomorphism, $M|_{q^{-1}(y)} \cong i_y^* M \otimes \cL$. Thus $i_y^* q_! M \cong q_!(M|_{q^{-1}(y)}) \cong R\Gamma_c(q^{-1}(y),\cL_H).$ But since $\cL_H$ is a nontrivial multiplicative local system on $H$ by construction, we have $R\Gamma_c(q^{-1}(y),\cL_H) = 0$.
\end{proof}

\begin{lemma}\label{lem:F psi image}
  Let $\phi \from \mfg \to \mfg/\mfu$. The pushforward $\phi_! \cF_\psi$ is supported on $\mfp/\mfu$.
\end{lemma}

\begin{proof}
  Write $\mfp_-$ to denote the opposite parabolic subalgebra to $\mfp$. By \cite[Theorem 1.2.2.4]{Lau87}, we have that
  \begin{equation*}
    \FT(\phi_! \cF_\psi) \cong \phi'{}^* \FT(\cF_\psi)[-\dim \mfu],
  \end{equation*}
  where $\phi' \from \mfp_-^* \to \mfg^*$. Since $\FT(\cF_\psi) \cong \delta_{-\bbG(X_\psi)}(-\dim \mfg)$, we have 
  \begin{equation*}
    \phi_! \cF_\psi \cong \FT(\delta_{\bbG(X_\psi) \cap \mfp_-^*})[-\dim \mfu](-\dim \mfg +  \dim \mfp) = \FT(\delta_{\bbG(X_\psi) \cap \mfp_-^*})[-\dim \mfu](-\dim \mfu).
  \end{equation*} 
  To complete the proof, we need to show that $\FT(\delta_{\bbG(X_\psi) \cap \mfp_-^*})$ is supported on $\mfp/\mfu \cong \mfl \hookrightarrow \mfp_-$.
  By definition $\FT(\delta_{\bbG(X_\psi) \cap \mfp_-^*}) = \pr_!(\pr'{}^* \delta_{\bbG(X_\psi) \cap \mfp_-^*} \otimes h^* \cL)$. Since $X_\psi$ satisfies the genericity condition $\mathfrak{ge1}$, it follows that for $\bbG(X_\psi) \cap \mfp_-^*$ is closed under translation by $\mfu_-^*$. Since $h^* \cL|_{\{Z\} \times \mfu_-^*}$ is a nontrivial multiplicative local system on $\mfu_-^*$ for any nonzero $Z \in \mfu_-$, it then follows that $\Gamma_c(\{Y\} \times \mfp_-^*, \pr'{}^* \delta_{\bbG(X_\psi) \cap \mfp_-^*} \otimes h^* \cL) = 0$ for any $Y \in \mfp_- \smallsetminus \mfl$.
\end{proof}

\begin{lemma}\label{lem:1}
  Let $P' \subset G$ be another parabolic subgroup whose Levi component $L'$ contains $T$. If $h \notin P_r' N_{G_r}(T_r) P_r$, then for some root $\alpha \in \Phi(G,T) \smallsetminus (\Phi(L,T) \cup \Phi(L',T))$, the image $p(h^{-1}U_{P,r} h \cap L_{r:r+} U_{P,r})$ contains $\mft_\alpha$.
\end{lemma}

\begin{proof}
  Throughout this proof, for any $g \in G_r$, we write $\bar g$ to mean its image in $G_0$. Write $N_r \colonequals h^{-1} U_{P',r} h$ for convenience. By Bruhat decomposition pulled back to $G_r$, we may write $h = p' z \dot w p$ for some $p' \in P_r'$, $p \in P_r$, $\dot w \in N_{G_r}(T_r)$,  and $z \in U_{P',r}^- \cap \dot w U_{P,r}^- \dot w^{-1}$ where $\bar z = 1$. Then we see that $N_r = p^{-1} \dot w^{-1} z^{-1} U_{P',r} z \dot w p$ and that $p(N_r \cap L_{r:r+}U_{P,r}) = p(\dot w^{-1} z^{-1} U_{P',r} z \dot w \cap L_{r:r+}U_{P,r})$ since $L_{r:r+}U_{P,r}$ is normalized by $P_r$. Hence we may assume that $h = z \cdot \dot w$.

  Assume now that $z \neq 1$ so that $h \notin P_r' N_{G_r}(T_r) P_r$. Let $d$ be minimal such that the image $\bar z_d$ of $z$ in $G_d$ is not the identity; such a $d$ exists since $z \neq 1$. Write $\bar z_d = \prod_{\alpha \in (\Phi^-(G,T) \smallsetminus \Phi^-(L',T)) \cap w \cdot (\Phi^-(G,T) \smallsetminus \Phi^-(L,T))} \bar z_d^\alpha$ where $\bar z_d^\alpha \in U_{\alpha,d}$ and the product is taken with respect to an arbitrary fixed order on the roots. Let $\alpha \in (\Phi^-(G,T) \smallsetminus \Phi^-(L',T)) \cap w \cdot (\Phi^-(G,T) \smallsetminus \Phi^-(L,T))$ be maximal in the sense that $\height(\alpha) \geq \height(\beta)$ for all $\beta \in (\Phi^-(G,T) \smallsetminus \Phi^-(L',T)) \cap w \cdot (\Phi^-(G,T) \smallsetminus \Phi^-(L,T))$ with $\bar z_d^\beta \neq 1$. Choose $x \in U_{r,-\alpha} \subset U_{P',r}$ such that $\bar x_{r-d} \neq 1$ and $\bar x_{r-d-1} = 1$. Then $\dot w^{-1} z^{-1} x z \dot w \in T_{r:r+} U_{P,r}$ and $p(\dot w^{-1} z^{-1} x z \dot w)$ contains a nonzero element of $\mft_{w^{-1} \cdot \alpha}$. By rescaling $x$, we see that any element of $\mft_{w^{-1} \cdot \alpha}$ can be written in the form $p(\dot w^{-1} z^{-1} x z \dot w)$ for some $x \in U_{r,-\alpha}$. Noting that $w^{-1} \cdot \alpha \in \Phi^-(G,T) \smallsetminus \Phi^-(L,T)$ by construction, the conclusion of the lemma holds.
\end{proof}

\begin{lemma}\label{lem:2 L}\label{lem:2}
  Let $g \in G_r \smallsetminus P_r$ such that its image $\bar g$ in $G_0$ lies in $P_0$. Then there exists a root $\alpha \in \Phi(G,T) \smallsetminus \Phi(L,T)$ such that
  \begin{equation}
    \label{eq:fiber contains alpha}
    \{u l g l^{-1} u' : u,u' \in U_{P,r}, \, l \in L_r\} \supset g \cdot g'{}^{-1}\mft_\alpha g'
  \end{equation}
  for some $g' \in L_r$.
\end{lemma}

\begin{proof}
  The set of $g \in G_r \smallsetminus P_r$ with $\bar g \in P_0$ forms a subgroup with an Iwahori decomposition with respect to $P = L U_P$, and therefore  we may write $g = up$ for some non-identity $u \in U_{P,r}^-$ with $\bar u = 1$ and some $p \in P_r$; moreover, by the first paragraph of the proof, we may assume $\bar p \in L_0$. Let $d$ be such that $\bar u_d \neq 1$ and $\bar u_{d-1} = 1$. Let $\alpha \in \Phi(G,T)^-$ be a root such that $\bar u_d^\alpha \neq 1$ and such that $\height(\alpha) \geq \height(\beta)$ for all $\beta \in \Phi^-(G,T)$ such that $\bar u_d^\beta \neq 1$. By assumption $\alpha \notin \Phi(L,T)$. Choose $x \in U^{-\alpha} \subset U_{P,r}$ such that $\bar x_{r-d} \neq 1$ and $\bar x_{r-d-1} = 1$. Then $xupx^{-1} \in u \cdot \mft_\alpha \cdot xpx^{-1} \subset u \cdot \mft_\alpha \cdot p \cdot U_{P,r} = up \cdot p^{-1} \mft_\alpha p \cdot U_{P,r}$. 
  Since $\bar p \in L_0$, we see that $p^{-1} \mft_\alpha p = g'{}^{-1} \mft_\alpha g'$ for some $g' \in L_r$. 
  For any $x$ as above, there exists a $u_x' \in U_{P,r}$ such that $xupx^{-1} \in up \cdot g'{}^{-1} \mft_\alpha g' \cdot u_x'$. Moreover, by varying $x$, we can arrange for any element of $up \cdot g'{}^{-1} \mft_\alpha g'$ to be obtained.  
\end{proof}

\subsection{The Mackey formula}\label{subsec:mackey}

Before we prove the Mackey formula, we note that Lemma \ref{lem:FT pInd} can be upgraded to a statement about $\cL_{\psi,r}$ and $\cF_{\psi,r}$.

\begin{proposition}\label{lem:Ind L psi}\label{prop:Ind L psi}
  We have $\pInd_{P_r!}^{G_r}(\cL_{\psi,r}) \cong \cF_{\psi,r}$.
\end{proposition}

\begin{proof}
  By Lemma \ref{lem:FT pInd}, the desired assertion holds once we prove
  \begin{equation}\label{eq:two pInd}
    \pInd_{P_r!}^{G_r}(\mfi_! \cL_\psi) \cong \mfi_! \mathfrak{pInd}_{P_r!}^{G_r}(\cL_\psi).
  \end{equation}
  Observe that by base change with respect to the Cartesian square
  \begin{equation*}
    \begin{tikzcd}
      f^{-1}(\mfi(\mfl)) \ar{r}{f} \ar{d}[left]{\mfi} & \mfl \ar{d}{\mfi} \\
      \widehat G_r \ar{r}{f} & L_r
    \end{tikzcd}
  \end{equation*}
  we have 
  \begin{equation*}
    \pInd_{P_r}^{G_r}(\mfi_! \cL_\psi) \cong \pi_!  \widetilde{{\mfi}_! f^* \cL_\psi}[2n].
  \end{equation*}
  From this we see that to prove \eqref{eq:two pInd}, it suffices to prove that $\pInd_{P_r}^{G_r}(\mfi_! \cL_\psi)$ is supported on $i(\mfg)$. Writing $\tilde i$ for the inclusion in $\widetilde G_r$ induced by $\mfi \from f^{-1}(\mfi(\mfl)) \hookrightarrow \widehat G_r$, we now need to prove:
  \begin{equation*}
    \text{If $g \in G_r \smallsetminus \mfi(\mfg)$, then $R\Gamma_c((\pi \circ \tilde \mfi)^{-1}(g), \widetilde{f^* \cL_\psi}) = 0$.}
  \end{equation*}

  To this end, we first make the following argument: Let $b \in \mfi(\mfl)U_{P,r} \smallsetminus \mfi(\mfp)$. We may write $b = l \cdot \prod_{\alpha \in \Phi^+(G,T) \smallsetminus \Phi^+(L,T)} u_\alpha$ for some $u_\alpha \in U_{\alpha,r}$. By assumption, there exists a $d < r$ such that the image $\bar b_d$ of $b$ in $B_d$ is nontrivial but $\bar b_{d-1} = 1$. Choose an $\alpha \in \Phi^+(G,T)$ such that $\height(-\alpha) \geq \height(-\beta)$ for all $\beta \in \Phi^+(G,T)$ with $\bar u_{d,\beta} \neq 1$. Then $u_{-\alpha} b u_{-\alpha}^{-1} \in i(\mft) U_r$ for any $u_{-\alpha} \in U_{-\alpha,r}$ with $\bar u_{-\alpha,r-d-1} = 1$ and moreover $p(\{u_{-\alpha} b u_{-\alpha}^{-1} : u_{-\alpha} \in U_{-\alpha,r-d:r+}\}) \supset \mft_\alpha$.

  Now consider the following: For $\alpha \in \Phi^+(G,T) \smallsetminus \Phi^+(L,T)$ and $0 \leq d \leq r-1$, set
    \begin{equation*}
      Y_{\alpha,d} \colonequals \left\{(g,hP_r) \in (\pi \circ \tilde \mfi)^{-1}(g) : 
      \begin{gathered}
        \text{$\overline{(hgh^{-1})}_d \neq 1$ but $\overline{(hgh^{-1})}_{d-1} = 1$} \\
        \text{$\alpha \in \Phi^+(G,T)$ is maximal for $(hgh^{-1})_d$}
      \end{gathered}\right\}.
    \end{equation*}
    Since $g \notin \mfi(\mfg)$ by assumption, any $hP_r \in G_r/P_r$ such that $(g,hP_r) \in (\pi \circ \tilde i)^{-1}(hgh^{-1})$ satisfies that $h^{-1} g h \in \mfi(\mfl) U_{P,r} \smallsetminus \mfi(\mfp)$. Hence we have
    \begin{equation*}
      (\pi \circ \tilde \mfi)^{-1}(g) = \bigcup_{\alpha,d} Y_{\alpha,d}.
    \end{equation*}
    By the first paragraph, we know that each $Y_{\alpha,d}$ is equivariant under left-multiplication by $U_{-\alpha,r-d:r+}$ with respect to a \textit{nontrivial} multiplicative local system by the genericity assumption on $\psi$, so the desired assertion holds by Lemma \ref{lem:equiv vanishing}.
\end{proof}

Now we prove the Mackey formula. 

\begin{proposition}\label{prop:mackey}\label{prop:Res Ind}
  Let $P' \subset G$ be a parabolic subgroup with Levi complement $L'$ and assume that $L'$ contains $T$. If $X_\psi$ is $(L,G)$-generic, then for $M \in D_{L_r}^{\psi}(L_r)$, we have
  \begin{equation*}
    \pRes_{P_r'!}^{G_r} \pInd_{P_r!}^{G_r}(M) \cong \bigoplus_{w \in W_{L'} \backslash W / W_L} \pInd_{L_r' \cap \ad(w)P_r!}^{L_r'} \pRes_{P_r' \cap \ad(w)L_r!}^{\ad(w)L_r}(M_w),
  \end{equation*}
  where $M_w$ denotes the pullback of $M$ under $\ad(w^{-1}) \from \ad(w)L_r \to L_r$.
\end{proposition}

\begin{proof}
  For convenience, let $n = \dim U_{P,r}.$ Then by using Lemma \ref{lem:pInd} and base change with respect to the Cartesian square
  \begin{equation*}
    \begin{tikzcd}
      Y \ar{r} \ar{d} & \widetilde G_r \ar{d} \\
      P_r' \ar{r} & G_r
    \end{tikzcd}
  \end{equation*} 
  where $Y \colonequals \{(g,hP_r) \in P_r' \times G_r/P_r : hgh^{-1} \in P_r\}$, we have
  \begin{equation*}
    \pRes_{P_r'!}^{G_r}(\pInd_{P_r!}^{G_r}(M)) \cong \pi_{Y!} \widetilde{f_Y^* M}[2n],
  \end{equation*}
  where $f_Y$, $\pi_Y$, and $\alpha_Y$ are
  \begin{equation*}
    \begin{tikzcd}
      & \widehat Y \ar{dl}[above left]{f_Y} \ar{dr}{\alpha_Y} \\
      L_r & & Y \ar{dr}{\pi_Y}  \\
      &&& L_r'
    \end{tikzcd}
    \qquad
    \begin{tikzcd}
      & (g,h)\ar[mapsto]{dl} \ar[mapsto]{dr} \\
      \beta(h^{-1}gh) & & (g, hP_r) \ar[mapsto]{dr} \\
      &&& \beta'(g)
    \end{tikzcd}
  \end{equation*}
  for the projections $\beta \from P_r \to L_r$ and $\beta' \from P_r' \to L_r'$. Here, $\widehat Y = \{(g,h) \in P_r' \times G_r : h^{-1} g h \in P_r\}$ and again as usual, $\widetilde{f_Y^*M}$ denotes the unique object on $\widetilde G_r$ such that $\alpha_Y^* \widetilde{f_Y^* M} \cong f_Y^* M$.   Consider
  \begin{align*}
    Y' &\colonequals \{(g,hP_r) \in Y : h \notin Q_r N_{G_r}(T_r) P_r\}, &  
    \widehat Y' &\colonequals \{(g,h) \in \widehat Y : h \notin Q_r N_{G_r}(T_r) P_r\},\\
    Y'' &\colonequals \{(g,hP_r) \in Y : h \in Q_r N_{G_r}(T_r)P_r\} & 
    \widehat Y'' &\colonequals \{(g,h) \in \widehat Y : h \in Q_r N_{G_r}(T_r)P_r\}.
  \end{align*}
  Write $f_Y' = f_Y|_{\widehat Y'}, f_Y'' = f_Y|_{\widehat Y''}$ and $\pi_Y' = \pi_Y|_{Y'}, \pi_Y'' = \pi_Y|_{Y''}$.

  \begin{claim}\label{claim:Y'}
    We have
    \begin{equation*}
      \pi_{Y!}' \widetilde{f_Y^{\prime *} M} = 0.
    \end{equation*}
  \end{claim}

  \begin{proof}
    Choose any $l' \in L_r'$ and set
    \begin{equation*}
      Y_{l'} \colonequals \pi_Y^{-1}(l') = \{(l'u',hP_r) \in l'U_{P',r} \times G_r/P_r : h^{-1}l'u'h \in P_r\}
    \end{equation*}
    and consider the projection to the second coordinate
    \begin{equation*}
      Y_{l'} \xrightarrow{\pi_2} G_r/P_r, \qquad (l'u',hP_r) \mapsto hP_r.
    \end{equation*}
    Write $Y_{l'}' = Y_{l'} \cap Y'$. By definition, the stalk of $\pi_{Y!}' \widetilde{f_Y^{\prime *} M}$ at $l'$ vanishes exactly when the cohomology of $\widetilde{f_Y^{\prime*} M}$ on $Y_{l'}$ is zero. We will prove that this cohomology is zero by showing that the stalks of $\pi_{2!}(\widetilde{f_Y^* M}|_{Y_{l'}})$ at $hP_r$ for any $h \in P_r' N_{G_r}(T_r) P_r$ \textit{already} vanish.

    If $(l'u_0', hP_r) \in \pi_2^{-1}(hP_r)$, then $(l'u_0'u',hP_r) \in \pi_2^{-1}(hP_r)$ if and only if $h^{-1}u'h \in P_r$. Hence we have an isomorphism $\pi_2^{-1}(hP_r) \cong U_{P',r} \cap hP_rh^{-1}$. Under this isomorphism, the map $f_Y$ is transported to the map $U_{P',r} \cap hP_r h^{-1} \to L_r$ given by $u' \mapsto p(h^{-1}l'u_0'h) p(h^{-1}u'h)$.
    By Lemma \ref{lem:1}, if $h \notin P_r' N_{G_r}(T_r) P_r$, then there exists a root $\alpha \in \Phi(G,T) \smallsetminus \Phi(L,T)$ such that $U_{P',r} \cap hP_rh^{-1}$ contains $U_{P',r} \cap h\mfi(\mft_\alpha)U_{P,r} h^{-1}$. It follows then, using that $U_{P,r} \cap hP_rh^{-1}$ is invariant under right-multiplication by $h\mfi(\mft_\alpha)h^{-1}$, that $f_Y^*M|_{\pi_2^{-1}(hP_r)}$ is $(\mft_\alpha,\cL_\psi|_{\mft_\alpha})$-equivariant. By the genericity assumption $\mathfrak{ge1}$ on $X_\psi$, we know that $\cL_\psi|_{\mft_\alpha}$ is a nontrivial multiplicative local system on $\mft_\alpha$. Therefore, by Lemma \ref{lem:equiv vanishing}, the stalk of $\pi_{2!}(f_Y^*(M))$ at $hP_r$ for $h \notin P_r'N_{G_r}(T_r)P_r$ vanishes. The claim now follows.
  \end{proof}

  Claim \ref{claim:Y'} shows that
  \begin{equation*}
    \pi_{Y!} \widetilde{f_Y^* M} \cong \pi_{Y!}'' \widetilde{f_Y^{\prime\prime*} M},
  \end{equation*}
  so for the rest of the proof, we calculate on $Y''$. We have a disjoint union decomposition
  \begin{equation*}
    Y'' = \bigsqcup_{W_{L'} \backslash W / W_L} Y_w'', \qquad \text{where $Y_w'' = \{(g,hP_r) \in Y : h \in P_r' \dot w P_r\}$,}
  \end{equation*}
  where by $\dot w$ we mean a lift of any representative of the double coset $w$. Of course this lifts to an analogous decomposition for $\widehat Y''$. Write $f_w'' = f_Y|_{\widehat Y_w''}$ and $\pi_w'' = \pi_Y|_{Y_w''}.$ 

  The upshot of Claim 1 is that it proves that the functor $\pRes_{P_r'!}^{G_r} \pInd_{P_r!}^{G_r}$ on $D_{L_r}^\psi(L_r)$ is really controlled by its behavior on the ``expected'' part of $G_r/P_r$---namely, the part that has a recognizable generalized Bruhat decomposition. From here, using standard methods, one should expect $\pRes_{P_r'!}^{G_r} \pInd_{P_r!}^{G_r}$ to have the typical shape of a Mackey formula. For only this next claim, we abuse notation and write everything in terms of ``$\pi_! f^* M$'' when we really mean ``$\pi_! \widetilde{f^* M}$.'' 

  \begin{claim}
    We have
    \begin{equation*}
      \pi_{w!}'' f_w^{\prime\prime*} \cong \pInd_{L_r' \cap \ad(w)P_r!}^{L_r'} \pRes_{P_r' \cap \ad(w)L_r}^{\ad(w)L_r} {} \circ {} \ad(w^{-1})^* [-2n].
    \end{equation*}
  \end{claim}

  \begin{proof}
    Using Lemma \ref{lem:pInd}, we have
    \begin{equation*}
      \pInd_{L_r' \cap \ad(w)P_r!}^{L_r'} \pRes_{P_r' \cap \ad(w)L_r}^{\ad(w)L_r} {} \circ {} \ad(w^{-1})^* \cong 
      \pi_!' f'{}^* \beta_{!}' q_w^* \ad(w^{-1})^*[2n_w],
    \end{equation*}
    where
    \begin{enumerate}[label=\textbullet]
      \item $n_w = \dim(L_r' \cap \ad(w)U_{P,r})$
      \item $q_w$ denotes the inclusion $P_r' \cap \ad(w)L_r \to \ad(w)L_r$
      \item $\beta'$ denotes the natural projection $P_r' \to L_r'$
      \item $\beta_w$ denotes the natural projection $\ad(w)P_r \to \ad(w)L_r$
      \item $\widetilde L_r' = \{(g',h'(L_r' \cap \ad(w)P_r)) \in L_r' \times L_r'/(L_r' \cap \ad(w)P_r) : h'{}^{-1}g'h' \in L_r' \cap \ad(w)P_r\}$
      \item $f' \from \widetilde L_r' \to L_r' \cap \ad(w)L_r$ is $(g',h'(L_r' \cap \ad(w)P_r)) \mapsto \beta_w(h'{}^{-1}g'h')$
      \item $\pi' \from \widetilde L_r' \to L_r'$ is the projection to the first coordinate
    \end{enumerate}
    Pictorially:
    \begin{equation*}
      \begin{tikzcd}
        && P_r' \cap \ad(w)L_r \ar{dl}[above left]{q_w} \ar{dr}{\beta'} && \widetilde L_r' \ar{dl}[above left]{f'} \ar{dr}{\pi'} \\
        L_r & \ar{l}{\ad(w^{-1})} \ad(w)L_r && L_r' \cap \ad(w)L_r && L_r'
      \end{tikzcd}
    \end{equation*}
    Let 
    \begin{equation*}
      \widetilde L_r^{\prime(P)} \colonequals \{(g',h'(L_r' \cap \ad(w)P_r), p') \in \widetilde L_r' \times (P_r' \cap \ad(w)L_r) : \beta_w(h'{}^{-1}g'h') = \beta'(p')\}
    \end{equation*}
    and consider the Cartesian square
    \begin{equation*}
      \begin{tikzcd}
        \widetilde L_r^{\prime (P)} \ar{r}{p_3} \ar{d}[left]{p_{12}} & P_r' \cap \ad(w)L_r \ar{d}{\beta'} \\
        \widetilde L_r' \ar{r}{f'} & L_r' \cap \ad(w)L_r
      \end{tikzcd}
    \end{equation*}
    By base change, we have
    \begin{equation*}
      \pi_!' f^{\prime*} \beta_!' q_w^* \ad(w^{-1})^* \cong \pi_!' p_{12!} p_3^* q_w^* \ad(w^{-1})^* = \tilde \pi_{w!} \tilde f_w^*,
    \end{equation*}
    where
    \begin{align*}
      \tilde \pi_w &= \pi' \circ p_{12} \from (g',h'(L_r' \cap \ad(w)P_r),p') \mapsto g', \\
      \tilde f_w &= \ad(w^{-1}) \circ q_w \circ p_3 \from (g',h'(L_r' \cap \ad(w)P_r),p') \mapsto \ad(w^{-1})(p').
    \end{align*}

    We now relate this to $Y_w''$. We have an isomorphism
    \begin{equation*}
      P_r' \dot w P_r/P_r \to P_r'/(P_r' \cap \ad(w)P_r), \qquad p'' \dot w P_r \mapsto p''(P_r' \cap \ad(w)P_r),
    \end{equation*}
    and we may consider the map
    \begin{align*}
      \varphi \from Y_w'' &\to \widetilde L_r^{\prime (P)}, &
      (p',p'' \dot w P_r) &\mapsto (\beta'(p'),\beta'(p'')(L_r' \cap \ad(w)P_r),\beta_w(p''{}^{-1}p'p'')). 
    \end{align*}
    Then we have a commutative diagram
    \begin{equation*}
      \begin{tikzcd}
        & Y_w'' \ar[bend right=15]{ddl}[above left]{f_w''} \ar[bend left=15]{ddr}{\pi_w''} \ar{d}{\varphi} \\
        & \widetilde L_r^{\prime(P)} \ar{dl}[above left]{\tilde f_w} \ar{dr}{\tilde \pi_w} \\
        L_r && L_r'
      \end{tikzcd}
    \end{equation*}
    where we note that left triangle holds since 
    \begin{equation*}
      f_w''(p',p'' \dot w P_r) = \beta(\dot w^{-1} p''{}^{-1}p'p'' \dot w) = \dot w^{-1} \beta_w(p''{}^{-1}p'p'') \dot w.
    \end{equation*}
    The map $\varphi$ is an affine fibration with fibers of dimension
    \begin{align*}
      \dim{}&{}(P_r' \cap \ad(w)U_{P,r}) + \dim(P_r' \dot w P_r/P_r) \\
      &= \dim(P_r' \cap \ad(w)U_{P,r}) + \dim(P_r') - \dim(P_r' \cap \ad(w)P_r) \\
      &= \dim(P_r') - \dim(P_r' \cap \ad(w)L_r) \\
      &= \dim (L_r') - \dim(L_r' \cap \ad(w)L_r) + \dim(U_{P',r}) - \dim(U_{P',r} \cap \ad(w)L_r) \\
      &= \dim(L_r' \cap \ad(w)U_{P,r}) + \dim(L_r' \cap \ad(w)U_{P,r}^-) \\
      &\qquad\qquad\qquad+ \dim(U_{P',r} \cap \ad(w)U_{P,r}) + \dim(U_{P',r} \cap \ad(w)U_{P,r}^-) \\
      &= \dim(L_r' \cap \ad(w)U_{P,r}) + \dim(\ad(w)U_{P,r}) = n_w + n.
    \end{align*}
    Therefore we have
    \begin{equation*}
      \tilde \pi_{w!} \tilde f_w^* 
      \cong \tilde \pi_{w!} \varphi_! \varphi^* \tilde \pi_w[-2n-2n_w] 
      = \pi_{w!}'' f_w^{\prime\prime*}[-2n-2n_w].
    \end{equation*}
    Altogether, 
    \begin{equation*}
      \pInd_{L_r' \cap \ad(w)P_r!}^{L_r!} \pRes_{P_r' \cap \ad(w)L_r}^{\ad(w)L_r} \circ \ad(w^{-1})^* \cong \tilde \pi_{w!} \tilde f_w^*[2n_w] = \pi_{w!}'' f_w^{\prime\prime*}[-2n].\qedhere
    \end{equation*}
  \end{proof}

  Let us now see the conclusion of the proposition from the two claims. For $w \in W$, we may view $\ad(w^{-1})^* X_\psi$ as an element of $(\mfl' \cap \ad(w)\mfl)^*$ so that $\ad(w^{-1})X_\psi$ is $(L' \cap ad(w)L,L')$-generic. We write $D_{L_r' \cap \ad(w)L_r}^{\psi_w}(L_r' \cap \ad(w) L_r)$ and $D_{L_r'}^{\psi_w}(L_r')$ for the corresponding generic subcategories as in Definition \ref{def:psi subcats}. Then by Proposition \ref{prop:generic induction}, we have $\pInd_{L_r' \cap \ad(w)P_r!}^{L_r'} \pRes_{P_r' \cap \ad(w)L_r}^{\ad(w)L_r}(\ad(w^{-1})^* M) \in D_{L_r'}^{\psi_w}(L_r')$. By assumption, $X_\psi$ is $(L,G)$-generic, which means it satisfies the condition $\mathfrak{ge2}$. Hence for $w,w'$ representing two different double cosets in $W_{L'} \backslash W/W_L$, the corresponding $(L' \cap \ad(w)L,L')$-generic elements $\ad(w^{-1})^* X_\psi$, $\ad(w'{}^{-1})^* X_\psi$ satisfy the hypothesis of Lemma \ref{lem:idempotent orthog}. Therefore, this lemma implies that the categories $D_{L_r'}^{\psi_w}(L_r'), D_{L_r'}^{\psi_{w'}}(L_r')$ are orthogonal subcategories. The direct summand decomposition now follows.
\end{proof}

\subsection{Harish-Chandra transform}\label{subsec:harishchandra}

Let $\phi \from G_r \to G_r/U_{P,r}$ as in Section \ref{subsec:soft lemmas}.

\begin{proposition}
  \label{prop:p psi support}
  For any $N \in D_{G_r}^\psi(G_r)$, the convolution $\cL_{\psi,r} \star_! \phi_! N$ is supported on $P_r/U_{P,r} \cong L_r$ so that $\cL_{\psi,r} \star_! \phi_! N \cong \cL_{\psi,r} \star_! \pRes_{P_r!}^{G_r}(N).$
\end{proposition}

\begin{proof}
  Since $N$ is $G_r$-equivariant under conjugation, the pushforward $\phi_! N$ is $P_r$-equivariant under conjugation. On the other hand, $\cL_{\psi,r}$ is also $P_r$-equivariant under conjugation, and therefore so must be $\cL_{\psi,r} \star_! \phi_! N$. Explicitly, this means that for the conjugation action $c \from P_r \times G_r/U_{P,r} \to G_r/U_{P,r}$, we have $c^* (\cL_{\psi,r} \star_! \phi_! N) \cong (\overline \QQ_\ell)_{P_r} \boxtimes (\cL_{\psi,r} \star_! \phi_! N).$ In particular, for every $g \in G_r$, we have
  \begin{equation}\label{eq:constancy}
    c^*(\cL_{\psi,r} \star_! \phi_! N)_{P_r \times \{g U_{P,r}\}} \cong (\overline \QQ_\ell)_{P_r} \boxtimes (\cL_{\psi,r} \star_! \phi_! N)_{gU_{P,r}}.
  \end{equation}

  For any element of $G_r$, we use $\bar{\phantom{g}}$ to denote its image in $G_0$. 
  First consider the setting that the image $\bar g \notin P_0$. Then we have $g = x \dot w$ for some $x \in G_r$ with $\bar x \in L_0$ and some $\dot w \in N_{G_0}(T_0) \smallsetminus N_{L_0}(T_0)$. Since $X_\psi$ is fixed by the coadjoint action of $L_r$ and this action factors through $L_0$, we see that while the restriction $(\cL_{\psi,r} \star_! \phi_! N)|_{g P_r/U_{P,r}}$ is $(\mfl, \cL_\psi)$-equivariant by right multiplication (Lemma \ref{lem:L psi}), it is $(\ad^*(\dot w)(\mfl), \ad^* \cL_\psi)$-equivariant by left multiplication. By \eqref{eq:constancy} together with the argument as the end of Proposition \ref{prop:mackey}, we see that condition $\mathfrak{ge2}$ on $X_\psi$ then forces $(\cL_{\psi,r} \star_! \phi_! N)|_{g P_r/U_{P,r}} = 0$.

  It remains to show then that $(\cL_{\psi,r} \star_! \phi_! N)_{g U_{P,r}} = 0$ for $g \notin P_r$ with $\bar g \in P_0$. For this, we will invoke Lemma \ref{lem:2 L}, which tells us that $c(P_r \times g U_{P,r})$ contains $g \cdot g'{}^{-1} \mft_\alpha g'$ for some $\alpha \in \Phi(G,T) \smallsetminus \Phi(L,T)$ and some $g' \in L_r.$ Since $\cL_{\psi,r} \star_! \phi_! N$ is $(\mfl,\cL_\psi)$-equivariant, the genericity assumption $\mathfrak{ge1}$ guarantees that the restriction $\cL_\psi|_{g'{}^{-1} \mft_\alpha g'}$ is a nontrivial multiplicative local system. But now the constancy required in \eqref{eq:constancy} then forces $(\cL_{\psi,r} \star_! \phi_! N)_{gU_{P,r}} = 0$.

  We have now shown that $\cL_{\psi,r} \star_! \phi_! N$ is supported on $P_r/U_{P,r} \cong L_r$. By base change, 
  this implies that $\cL_{\psi,r} \star_! \phi_! N \cong \cL_{\psi,r} \star_! i_L^* \phi_! N$, where $i_L \from L_r \hookrightarrow G_r/U_{P,r}$. On the other hand, by base change again, 
  we have $i_L^* \phi_! N \cong \pRes_{P_r!}^{G_r}(N)$. Hence we have the isomorphism $\cL_{\psi,r} \star_! \phi_! N \cong \cL_{\psi,r} \star_! \pRes_{P_r!}^{G_r}(N)$.
\end{proof}

\begin{corollary}\label{cor:monoidal}
  The functor $\cL_{\psi,r} \star_! \pRes_{P_r!}^{G_r}$ is compatible with $\star_!$ and $\cL_{\psi,r} \star_* \pRes_{P_r*}^{G_r}$ is compatible with $\star_*$.
\end{corollary}

\begin{proof}
  By Lemma \ref{lem:monoidal p}, we have that $\phi_! \from D_{G_r}(G_r) \to D_{L_r}(G_r/U_{P,r})$ is monoidal with respect to $\star_!$. Since $\cL_{\psi,r} \star_! \pRes_{P_r!}^{G_r}(N) \cong \cL_{\psi,r} \star_! \phi_! N$ for any $N \in D_{G_r}^\psi(G_r)$ (Proposition \ref{prop:p psi support}), the corollary now holds using the fact that $\cL_{\psi,r}$ is idempotent (Lemma \ref{lem:idempotent}). 
\end{proof}

\begin{lemma}\label{lem:pRes summands}
  For any $N \in D_{G_r}^{\psi}(G_r)$, we have
  \begin{equation*}
    \phi_! N  \cong \bigoplus_{w \in W_L \backslash W/W_L} \phi_! N \star_! \pInd_{L_r \cap \ad(w)P_r!}^{L_r} \pRes_{\ad(w)L_r \cap P_r!}^{\ad(w)L_r}(\ad(w^{-1})^* \cL_{\psi,r}).
  \end{equation*}
\end{lemma}

\begin{proof}
  The end of the proof of Proposition \ref{prop:p psi support} shows that if $N \in D_{G_r}(G_r)$ has the property that $\phi_!N$ is supported on $P_r/U_{P,r} \cong L_r$, then $\phi_! N \cong \pRes_{P_r!}^{G_r}(N)$. Combining this with Lemma \ref{lem:F psi image}, we therefore have $\phi_! \cF_{\psi,r} \cong \pRes_{P_r!}^{G_r}(\cF_{\psi,r})$, which by Proposition \ref{prop:Ind L psi} and Proposition \ref{prop:mackey} yields
  \begin{equation}\label{eq:phi F psi}
    \phi_! \cF_{\psi,r} 
    \cong \bigoplus_{w \in W_L \backslash W / W_L} \pInd_{L_r \cap \ad(w)P_r!}^{L_r} \pRes_{\ad(w)L_r \cap P_r!}^{\ad(w)L_r}(\ad(w^{-1})^* \cL_{\psi,r}).
  \end{equation}
  Of course, for any $N \in D_{G_r}^\psi(G_r)$, we have $N \cong N \star_! \cF_{\psi,r}$, and therefore Lemma \ref{lem:monoidal p} implies $\phi_! (N \star_! \cF_{\psi,r}) \cong \phi_! N \star_! \phi_! \cF_{\psi,r}$, and so the conclusion follows from \eqref{eq:phi F psi}.
\end{proof}

\subsection{Harish-Chandra transform in the torus case}

In this subsection, we work in the special case that $L = T$ and $P = B$, $U_P = U$. 
Here, the behavior of the Harish-Chandra transform $\phi_! N$ of any $N \in D_{G_r}^\psi(G_r)$ (i.e.\ without needing to additionally project by convolving with $\cL_{\psi,r}$) already satisfies the nice properties established in Section \ref{subsec:harishchandra}.

\begin{proposition}\label{prop:torus p support}
  Let $N \in D_{G_r}^\psi(G_r)$.
  \begin{enumerate}[label=(\alph*)]
    \item $\phi_!N$ is supported on $B_r/U_r \cong T_r$.
    \item  $\pRes_{B_r!}^{G_r}(N) \cong \phi_!N$ and $\pInd_{B_r!}^{G_r}(\pRes_{B_r!}^{G_r}(N)) \cong \Av_{B_r!}^{G_r} \phi^* \phi_! N$.
  \end{enumerate}
\end{proposition}

\begin{proof}
  We know from Proposition \ref{prop:p psi support} that $\cL_{{}^w \psi,r} \star_! \phi_! N$ is supported on $T_r$, for every $w \in W$. To show that $\phi_! N$ is already supported on $T_r$, it is enough to show that $\cL_{\psi',r} \star_! \phi_! N = 0$ for any $X_{\psi'} \notin \{w^* X_\psi : w \in W\}$. By construction, we have $N \cong N \star_! \cF_{\psi,r}$, and by Lemma \ref{lem:monoidal p} and Proposition \ref{prop:Ind L psi}, we see that
  \begin{equation*}
    \cL_{\psi',r} \star_! \phi_! N \cong \cL_{\psi',r} \star_! \phi_! N \star_! \phi_! \pInd_{B_r}^{G_r}(\cL_{\psi,r}).
  \end{equation*}
  so that we see that the desired vanishing follows if we can show
  \begin{equation*}\label{eq:vanishing psi'}
    \cL_{\psi'} \star_! \phi_! \fpInd_{\mft}^{\mfg}(\cL_{\psi}) = 0,
  \end{equation*}
  where now $\phi \from \mfg \to \mfg/\mfu$. 
  By the same methodology as in the proof of Proposition \ref{prop:mackey}, we see that the above vanishing holds if 
  \begin{equation}\label{eq:vanishing psi' 2}
    \cL_{\psi'} \star_! \phi_! \pi_{w!} f_w^* \cL_\psi = 0,
  \end{equation}
  where $f_w, \pi_w$ denote restrictions of the usual $f,\pi$ to $\{(X,hB) \in \mfg \times BwB/B : \ad(h^{-1})(X) \in \mfb\}$. (Note here that there is no need to work with $\alpha$ since $f$ descends to a map on $\tilde \mfg$ due to the commutativity of $T$.)
  A simplification of the proof of Claim 2 in the proof of Proposition \ref{prop:mackey} shows that $\phi_! \pi_{w!} f_w^* \cL_\psi = \cL_{{}^w \psi}$ up to a shift. This implies \eqref{eq:vanishing psi' 2}, and so (a) is proved.

  The first assertion in (b) follows from (a) using the Cartesian square at the end of the proof of Proposition \ref{prop:p psi support}. From this, the second assertion holds via
  \begin{equation*}
    \pInd_{B_r!}^{G_r}(\pRes_{B_r!}^{G_r}(N)) \cong \Av_{B_r!}^{G_r} i_! p^* i_T^* \phi_! N \cong \Av_{B_r!}^{G_r} \phi^* \phi_! N. \qedhere
  \end{equation*}
\end{proof}

\begin{proposition}\label{lem:easy Res}
  \mbox{}
  \begin{enumerate}[label=(\alph*)]
    \item $\pRes_{B_r}^{G_r} \from D_{G_r}^\psi(G_r) \to D_{T_r}^\psi(T_r)$ is compatible with $\star_!$.
    \item The composition $\pInd_{B_r!}^{G_r} \pRes_{B_r!}^{G_r} \from D_{G_r}^\psi(G_r) \to D_{G_r}^\psi(G_r)$ is isomorphic to $\star_!$ convolution with $\pInd_{B_r!}^{G_r}(\pRes_{B_r!}^{G_r}(\cF_{\psi,r}))$.
  \end{enumerate}
\end{proposition}

\begin{proof}
  By Proposition \ref{prop:torus p support}, (a) now follows from Lemma \ref{lem:monoidal p}.

  We next prove (b). By Proposition \ref{prop:torus p support}, for any $N \in D_{G_r}^\psi(G_r)$, we have $\pRes_{B_r}^{G_r}(N) = \phi_! N$. Therefore, using Lemma \ref{lem:Av phi} and Lemma \ref{lem:idempotent}, we have
  \begin{align*}
    \pInd_{B_r!}^{G_r}(\pRes_{B_r!}^{G_r}(N))
    &\cong \Av_{B_r!}^{G_r} i_! \Infl_{T_r}^{B_r} \phi^* \phi_!(N) \\
    &\cong \Av_{B_r!}^{G_r} \phi^* \phi_! N \\
    &\cong \Av_{B_r!}^{G_r} \delta_{U_r} \star_! N \\
    &\cong \Av_{B_r!}^{G_r} \delta_{U_r} \star_! \cF_{\psi,r} \star_! N \\
    &\cong \pInd_{B_r!}^{G_r}(\pRes_{B_r!}^{G_r}(\cF_{\psi,r})) \star_! N. \qedhere
  \end{align*}
\end{proof}

\subsection{Generic parabolic induction is a $t$-exact equivalence}\label{subsec:Ind t-exact equivalence}

We come now come to the main theorem of this section.

\begin{theorem}\label{thm:equivalence}
  \mbox{}
  \begin{enumerate}[label=(\alph*)]
    \item $\cL_{\psi,r} \star_! \pRes_{P_r!}^{G_r}$ and $\pInd_{P_r!}^{G_r}$ are $t$-exact monoidal inverse equivalences between $D_{L_r}^\psi(L_r)$ and $D_{G_r}^\psi(G_r)$, and similarly for the functors $\cL_{\psi,r} \star_* \pRes_{P_r*}^{G_r}$ and $\pInd_{P_r*}^{G_r}$.
    \item $\cL_{\psi,r} \star_! \pRes_{P_r!}^{G_r} \cong \cL_{\psi,r} \star_* \pRes_{P_r*}^{G_r}$ on $D_{G_r}^\psi(G_r)$ and $\pInd_{P_r!}^{G_r} \cong \pInd_{P_r*}^{G_r}$ on $D_{L_r}^\psi(L_r)$.
    \item $\pInd_{P_r}^{G_r} \colonequals \pInd_{P_r!}^{G_r} \from D_{L_r}^\psi(L_r) \to D_{G_r}^\psi(G_r)$ is $t$-exact.
  \end{enumerate}
  In particular, $\pInd_{P_r}^{G_r}$ sends simple perverse sheaves lying in $D_{L_r}^\psi(L_r)$ to simple perverse sheaves lying in $D_{G_r}^\psi(G_r)$.
\end{theorem}

\begin{proof}
  We prove (a) first. We first check that $\pInd_{B_r!}^{G_r}\circ (\cL_\psi \star_! \pRes_{B_r!}^{G_r}) \cong \mathrm{Id}_{D_{G_r}^\psi(G_r)}$. Recall from Proposition \ref{prop:p psi support} that $\cL_{\psi,r} \star_! \pRes_{P_r!}^{G_r}(M) \cong \cL_{\psi,r} \star_! \phi_! M$. Hence by Lemma \ref{lem:CH HC}, we have
  \begin{equation*}
    \pInd_{P_r!}^{G_r}(\cL_{\psi,r} \star_! \pRes_{P_r!}^{G_r}(M)) 
    \cong \pInd_{P_r!}^{G_r}(\cL_{\psi,r}) \star_! M.
  \end{equation*}
  By Lemma \ref{lem:Ind L psi}, $\pInd_{P_r!}^{G_r}(\cL_{\psi,r}) = \cF_{\psi,r}$, and so in fact
  \begin{equation*}
    \pInd_{P_r!}^{G_r}(\cL_{\psi,r} \star_! \pRes_{P_r!}^{G_r}(M)) \cong \cF_{\psi,r} \star_! M \cong M,
  \end{equation*}
  where the last equality holds by Lemma \ref{lem:idempotent}. We next show the other composition---that $\cL_{\psi,r} \star_! \pRes_{B_r!}^{G_r} \circ \pInd_{B_r!}^{G_r} \cong \mathrm{Id}_{D_{L_r}^\psi(L_r)}$. 
  By the idempotency of $\cL_{\psi,r}$ (Lemma \ref{lem:idempotent}) and the orthogonality $\cL_{\psi,r} \star_! D_{L_r \cap \ad(w)L_r}^{\psi_w}(L_r \cap \ad(w)L_r) = 0$ for $w \notin W_L$ ($\mathfrak{ge2}$ and Lemma \ref{lem:idempotent orthog}), it follows from Proposition \ref{prop:Res Ind} that
  \begin{equation*}
    \cL_{\psi,r} \star_! \pRes_{P_r!}^{G_r}(\pInd_{P_r!}^{G_r}(N)) 
    \cong N.
  \end{equation*} 
  We may now conclude that $\cL_{\psi,r} \star_! \pRes_{P_r!}^{G_r}$ and $\pInd_{P_r!}^{G_r}$ are inverse equivalences of each other. The monoidality of $\pInd_{P_r!}^{G_r} \from D_{L_r}^\psi(L_r) \to D_{G_r}^\psi(G_r)$ now follows from the monoidality of $\cL_{\psi,r} \star_! \pRes_{P_r!}^{G_r} \from D_{G_r}^\psi(G_r) \to D_{L_r}^\psi(L_r)$ (Corollary \ref{cor:monoidal}). This proves (a) for the $!$ functors, and the proof is completely parallel for the $*$ functors.
  
  Next let us establish (b). By Lemma \ref{lem:adjointness}, we know that $\pRes_{P_r!}^{G_r}$ is left adjoint to $\pInd_{P_r*}^{G_r}$, which means that for any $N \in D_{G_r}(G_r)$ and any $M \in D_{L_r}(L_r)$, we have
  \begin{equation*}
    \Hom_{D_{G_r}(G_r)}(N, \pInd_{P_r*}^{G_r}(M)) \cong \Hom_{D_{L_r}(L_r)}(\pRes_{P_r!}^{G_r}(N), M).
  \end{equation*}
  We showed in Proposition \ref{prop:generic induction} that $\pInd_{P_r!}^{G_r}(D_{L_r}^\psi(L_r)) \subset D_{G_r}^\psi(G_r)$ and therefore we see that for any $N \in D_{G_r}^\psi(G_r)$ and $M \in D_{L_r}^\psi(L_r)$,
  \begin{align*}
    \Hom_{D_{G_r}^\psi(G_r)}{}&{}(N, \pInd_{P_r*}^{G_r}(M)) \\
    &\cong \Hom_{D_{L_r}^\psi(L_r)}(\cL_{\psi,r} \star_! \pRes_{P_r!}^{G_r}(N), M).
  \end{align*}
  On the other hand, the right adjoint of an equivalence must be its inverse equivalence, and so from (a) it follows that in fact $\pInd_{P_r!}^{G_r} \cong \pInd_{P_r*}^{G_r}$. Furthermore, analogously to above, $\pInd_{P_r!}^{G_r}$ is left adjoint to $\cL_{\psi,r} \star_* \pRes_{P_r*}^{G_r}$, and since we know from (a) that $\cL_{\psi,r} \star_! \pRes_{P_r!}^{G_r}$ is an inverse equivalence to $\pInd_{P_r!}^{G_r}$, we now have $\cL_{\psi,r} \star_* \pRes_{P_r*}^{G_r} \cong \cL_{\psi,r} \star_! \pRes_{P_r!}^{G_r}$. This proves (b). 

  For (c), we first recall that by Proposition \ref{prop:p psi support}, we have $\cL_{\psi,r} \star_! \pRes_{P_r!}^{G_r}(N) \cong \cL_{\psi,r} \star_! \phi_!N$, and by Lemma \ref{lem:pRes summands}, we have that $\cL_{\psi,r} \star_! \phi_!N$ is a direct summand of $\phi_! N$ for any $N \in D^\psi_{G_r}(G_r)$.  Therefore by Artin's theorem applied to $\phi$, we have that $\cL_{\psi,r} \star_! \pRes_{P_r!}^{G_r}$ is left $t$-exact for all parabolic subgroups $P$ which have Levi component $L$. By the analogous argument for $*$, we have that $\cL_{\psi,r} \star_* \pRes_{P_r*}^{G_r}$ is right $t$-exact for all such $P$. By hyperbolic localization, we have $\pRes_{P_r*}^{G_r} \cong \pRes_{P_r^-!}^{G_r}$, allowing us to conclude that $\cL_{\psi,r} \star_! \pRes_{P_r!}^{G_r}$ and $\cL_{\psi,r} \star_* \pRes_{P_r*}^{G_r}$ are both $t$-exact. By (a), it follows that $\pInd_{P_r*}^{G_r} \cong \pInd_{P_r!}^{G_r}$ is $t$-exact.
  
  The final assertion follows from Proposition \ref{prop:free idempotent}.
\end{proof}

\section{Construction of character sheaves on $G_r$}\label{sec:generic char sheaves}

We have shown that $(L,G)$-generic parabolic induction preserves perversity. In this section, we demonstrate that this construction can be iterated to produce many conjugation-equivariant perverse sheaves on $G_r$ starting from the data of any conjugation-equivariant perverse sheaf on a connected reductive group and a compatible sequence of generic rank-1 multiplicative local systems on Levi subgroups. 

This is decidedly inspired from the structure of Yu's construction \cite{Yu01} of supercuspidal representations of $p$-adic groups and Kim--Yu's generalization \cite{KY17} to the construction of types. The constructions in \textit{op. cit.\ }are $F$-rational, but here we work geometrically, at the level of $F^{\ur}$. In the context of Yu and Kim--Yu, our constructions can be viewed as sheaf-theoretic incarnations of their constructions in the case that everything splits over $F^{\ur}$. As such, our ``start'' torus $T$ can always be taken to be a \textit{split} maximal torus. Of course, as in the classical setting of Lusztig's character sheaves, these objects still know about the representation theory even when $T$ descends to a non-split torus rationally. In Section \ref{sec:comparison}, we will see the first instance of a highly nontrivial relationship between our construction of conjugation-equivariant perverse sheaves and Yu's construction of supercuspidal types.

\subsection{Clipped generic data}

We consider the notion of a \textit{clipped generic datum} (see \cite[Definition 5.3]{CO23}).

\begin{definition}
\label{def:generic datum}
  A \textit{clipped generic datum} is a 5-tuple $\dashover\Psi \colonequals (T, \vec G, \x, \vec r, \vec \cL)$ satisfying:
  \begin{enumerate}
    \item[\textbf{D0}] $T$ is a split maximal torus of $G$
    \item[\textbf{D1}] $\vec G = (G^0, G^1, \ldots, G^d)$ is a strictly increasing sequence of Levi subgroups of $G$ which contain $T$; we assume $G^d = G$
    \item[\textbf{D2}] $\x$ is a point in the apartment of $T$ in $G$ 
    \item[\textbf{D3}] $\vec r = (r_0, r_1, \ldots, r_d)$ is a sequence of integers satisfying $0 < r_0 < r_1 < \cdots < r_{d-1} \leq r_d$ if $d > 0$ and $0 \leq r_0$ if $d = 0$.
    \item[\textbf{D5}] $\vec \cL = (\cL_0, \cL_1, \ldots, \cL_d)$ is a sequence where for $0 \leq i \leq d$, $\cL_i$ is a rank-1 multiplicative local system on $G_{r_i}^i$ which is $(\mfg^i,\cL_{\psi_i})$-equivariant for a $(G^i,G^{i+1})$-generic element $X_{\psi_i} \in (\mfg^i)^*$; if $r_{d-1} = r_d$, we assume $\cL_d$ is the constant local system
  \end{enumerate}
\end{definition}

To handle edge cases, we put $r_{-1} = 0$. 

\begin{definition}
  Let $\dashover \Psi$ be any clipped generic datum. For $0 \leq i \leq d-1$, choose a parabolic subgroup $P^i$ of $G^{i+1}$ whose Levi subgroup is $G^i$; write $\vec P = (P^0, P^1, \ldots, P^{d-1})$. Define the functor
  \begin{equation*}
    \Ind_{\dashover \Psi, \vec P} \from D_{G_0^0}(G_0^0) \to D_{G_r}(G_r), \qquad \cK_{-1} \mapsto \cK_d,
  \end{equation*}
  where
  \begin{equation*}
    \cK_i \colonequals (\pi^i)^\dagger \pInd_{P_{r_{i-1}}^{i-1}}^{G_{r_{i-1}}^i}(\cK_{i-1}) \otimes \cL_i, \qquad \text{for $0 \leq i \leq d$},
  \end{equation*}
  where we write $\pi^i$ for the corresponding quotient maps $G_{r_i}^i \to G_{r_{i-1}}^i$ for $0 \leq i \leq d$ and $(\pi^i)^\dagger = (\pi^i)^*[\dim G_{r_{i-1}+:r_i+}^i]$ for smooth pullback.
\end{definition}

From Theorem \ref{thm:equivalence}, we obtain the following corollary:

\begin{theorem}\label{thm:clipped datum}
  To any clipped  generic datum $\dashover \Psi$ and any associated $\vec P$,
  \begin{equation*}
    \Ind_{\dashover \Psi, \vec P} \from D_{G_0^0}(G_0^0) \to D_{G_r}(G_r)
  \end{equation*}
  is $t$-exact, monoidal, fully faithful, and its image is closed under taking subquotients. 
\end{theorem}

\begin{proof}
  Since $\Ind_{\dashover \Psi, \vec P}$ is constructed inductively, it suffices to show that for each $i$, the functor
  \begin{equation*}
    D_{G_{r_{i-1}}^{i-1}}^{\psi_{i-1}}(G_{r_{i-1}}^{i-1}) \to D_{G_{r_i}^i}^{\psi_i}(G_{r_i}^i), \qquad \cK_{i-1} \mapsto (\pi^i)^\dagger \pInd_{P_{r_{i-1}}^{i-1}}^{G_{r_{i-1}}^i}(\cK_{i-1}) \otimes \cL_i
  \end{equation*}
  satisfies the desired adjectives. The above functor is given by a composition of the functors 
  \begin{align*}
    \pInd_{P_{r_{i-1}}^{i-1}}^{G_{r_{i-1}}^{i}} \from D_{G_{r_{i-1}}^{i-1}}^{\psi_{i-1}}(G_{r_{i-1}}^{i-1}) &\to D_{G_{r_{i-1}}^{i}}^{\psi_{i-1}}(G_{r_{i-1}}^{i}) \\
    (\pi^i)^\dagger - \otimes \cL_i \from D_{G_{r_{i-1}}^{i}}^{\psi_{i-1}}(G_{r_{i-1}}^{i}) &\to D_{G_{r_{i}}^{i}}^{\psi_i}(G_{r_{i}}^{i})
  \end{align*}
  where we note that the second functor has image contained in $D_{G_{r_i}^i}^{\psi_i}(G_{r_i}^i)$ exactly because of the assumption that $r_{i-1} < r_i$. The second functor is obviously monoidal, $t$-exact, and fully faithful. For the first functor, the required properties are the content of Theorem \ref{thm:equivalence}.
\end{proof}

In particular, Theorem \ref{thm:clipped datum} allows us to build an irreducible \textit{character sheaf} on $G_r$ from the datum of:
\begin{enumerate}
  \item a clipped generic datum $\dashover \Psi$ and an associated $\vec P$
  \item any irreducible character sheaf $\cK_{-1}$ of the connected reductive group $G_0^0$ over $k$.
\end{enumerate}
That is to say, by completing a clipped generic datum $\dashover \Psi = (T, \vec G, \x, \vec r, \vec \cL)$ to a \textit{generic datum} $\Psi \colonequals (T, \vec G, \x, \vec r, \vec \cL, \cK_{-1})$, Theorem \ref{thm:clipped datum} constructs from $\cK_{-1}$ a positive-depth character sheaf on $G_r$. (These constructions work equally well for weak clipped generic data, but for this next part, we choose to work with clipped generic data only so that we may make use of Theorem \ref{thm:equivalence}.) We summarize this for easy reference:

\begin{theorem}[sheaf associated to a datum]\label{thm:generic character sheaf}\label{thm:datum}
  Let $\Psi \colonequals (T,\vec G,\x,\vec r, \cF_\rho, \vec \cL)$ be a generic datum and $\vec P$ an associated sequence of parabolic subgroups of $G$. Then
  \begin{equation*}
    \cK_{\Psi} \colonequals \Ind_{\dashover \Psi, \vec P}(\cK_{-1})
  \end{equation*}
  is a simple conjugation-equivariant perverse sheaf on $G_r$.
\end{theorem}

\subsection{Howe factorization}\label{subsec:howe}

We describe a special case of Theorem \ref{thm:datum} which may be of particular interest. This section depends on the existence of \textit{Howe factorizations} of local systems of $T_r$ established by Kaletha \cite{Kal19}. As such, let us assume as in \textit{op.\ cit.\ }that $p$ is odd, not bad for $G$, and that $p \nmid |\pi_1(G_{\der})|$ and $p \nmid |\pi_1(\widehat G)|$ (note that the latter two conditions are implied by the non-badness of $p$ unless a component
of type $A_n$ is present).

\begin{definition}
  A \textit{Howe factorization} for a multiplicative local system $\cL$ on $T_r$ is a sequence  $\vec \cL = (\cL_{-1}, \cL_0, \cL_1, \ldots, \cL_d)$ where:
  \begin{enumerate}[label=\textbullet]
    \item there is a strictly increasing sequence $\vec G = (G^0, G^1, \ldots, G^d)$ of Levi subgroups of $G$ which contain $T \equalscolon G^{-1}$,
    \item for each $-1 \leq i \leq d$, $\cL_i$ is a multiplicative local system on $G_{r_i}^i$ which is $(G^i,G^{i+1})$-generic if $i < d$,
    \item for each $-1 \leq i \leq d-2$, we have $0 < r_i < r_{i+1}$ and we write $\vec r = (r_0, r_1, \ldots, r_d = 0)$,
    \item $\cL = \cL_{-1} \otimes \cL_0|_T \otimes \cdots \otimes \cL_d|_T$, where $\cL_i|_T$ is interpreted to mean the restriction of $\cL_i$ to $T_{r_i}$ viewed as a sheaf on $T_r$.
  \end{enumerate}
\end{definition}

Although Howe factorizations are not uniquely determined by $\cL$, both the sequence of Levi subgroups $\vec G$ and the sequence of depths $\vec r$ \textit{are}. In \textit{op. cit.}, Kaletha proves that any character has a Howe factorization.

\begin{proposition}[existence of Howe factorization \cite{Kal19}]\label{prop:howe}
  Any multiplicative local system $\cL$ on $T_r$ has a Howe factorization. Furthermore, if $\sigma$ is the Frobenius associated to an $F$-rational structure on $T \hookrightarrow G$ and $n$ is a positive integer such that $\sigma^n{}^* \cL \cong \cL$, then its Howe factorization can also be chosen to be $\sigma^n$-equivariant. 
\end{proposition}

\begin{proof}
  We give a sketch of Kaletha's construction. Associated to $\cL$ is a family of subsets 
  \begin{equation*}
    \Phi_s \colonequals \{\alpha \in \Phi(G,T) : ((\alpha^\vee)^* \cL)|_{\bbW_s} \cong \overline \QQ_\ell\} \subset \Phi(G,T), \qquad s \in \bbR_{\geq 0},
  \end{equation*}
  where $\bbW_s$ denotes the $s$th filtration of the Witt ring $\bbW$. The assumption on $p$ implies that each $\Phi_s$ is a Levi subsystem of $\Phi(G,T)$ (Lemma 3.6.1 of \textit{op. cit.}); these determine $\vec G$ and $\vec r$. One then constructs $\cL_{-1}, \cL_0, \ldots, \cL_d$ inductively. If $\Phi_r \subsetneq \Phi(G,T)$, then we take $\cL_d = \overline \QQ_\ell$. If $\Phi_r = \Phi(G,T)$, then $\cL$ restricts to the trivial local system on $\mfi(\mft) \cap G_{\der}$. By assumption, $\sigma^n{}^* \cL \cong \cL$ and therefore $\cL$ corresponds to a character on the $F_n$-points of the $r$th filtration of the torus $G/G_{\der}$ which is a quotient of $\mfi(\mft)(\FF_{q^n})$ by Lemma 3.1.3 of \textit{op. cit.}. (Here, $G_{\der}$ denotes the derived subgroup of $G$ and  $F_n$ denotes the degree-$n$ unramified extension of $F$.) Therefore $\cL|_{\mfi(\mft)}$ can be viewed as a $\sigma^n$-equivariant multiplicative local system on a subgroup scheme of the torus $G/G_{\der}$. Choose any extension of this to a $\sigma^n$-equivariant multiplicative local system on $G/G_{\der}$ and define $\cL_d$ to be its pullback. Then $\cL \otimes \cL_d^{-1}|_{T_r}$ can be viewed as a local system on $T_{r_{d-1}}$ is $(G^{d-1},G^d)$-generic using Lemma 3.6.8 and Corollary 3.6.10 as in Proposition 3.6.7, all in in \textit{op. cit.}.
\end{proof}

The upshot of Proposition \ref{prop:howe} is that from a Howe factorization $\vec \cL$ of $\cL$, we may extract a clipped generic datum $\dashover \Psi_{\vec \cL}$ by forgetting $\cL_{-1}$. We may then apply Theorem \ref{thm:clipped datum}:

\begin{corollary}
  Let $\cL$ be any multiplicative local system on $T_r$. For any Howe factorization $\vec \cL$ and any associated sequence of parabolic subgroups $\vec P$, we may assign a semisimple conjugation-equivariant perverse sheaf 
  \begin{equation*}
    \cK_{\vec \cL, \vec P} \colonequals \Ind_{\dashover \Psi, \vec P}(\pInd_{B_0^0}^{G_0^0}(\cL_{-1})). 
  \end{equation*}  
  In particular, if $\cL$ has trivial $W$-stabilizer, then $\cK_{\vec \cL, \vec P}$ is simple.
\end{corollary}

\begin{proof}
  Theorem \ref{thm:clipped datum} implies that $\Ind_{\dashover \Psi, \vec P}$ maps semisimple perverse sheaves to semisimple perverse sheaves. The semisimplicity and perversity of $\cK_{\vec \cL, \vec P}$ therefore follows from the semisimplicity and perversity of $\pInd_{B_0^0}^{G_0^0}(\cL_{-1})$, which holds by the decomposition theorem together with the fact that $\pi$ is proper and small when $r=0$ \cite[Proposition 1.2]{Lus84}. 
  The last sentence holds since $\cL$ having trivial stabilizer in the Weyl group $W$ of $G$ is equivalent to $\cL_{-1}$ having trivial stabilizer in the Weyl group $W_0$ of $G^0$, and then $\pInd_{B_0^0}^{G_0^0}(\cL_{-1})$ is simple by Proposition 4.5 of \textit{op. cit}.
\end{proof}

We shall see later (see Theorem \ref{thm:IC vreg}) that in fact $\cK_{\vec \cL, \vec P}$ is independent of the choice of Howe factorization $\vec \cL$ and independent of the choice of parabolic subgroups $\vec P$.

\section{Intermediate extension from the very regular locus}\label{sec:IC T vreg}

For the rest of the paper, we focus on studying $(T,G)$-generic parabolic induction. In this section and the next, we give alternative descriptions of $\pInd_{B_r}^{G_r}(\cL)$ for $(T,G)$-generic multiplicative local systems $\cL$ on $T_r$. The first (Theorem \ref{thm:IC vreg}) implies in particular that if $\cL$ is Frobenius-equivariant, then so is $\pInd_{B_r}^{G_r}(\cL)$, and therefore it makes sense to consider its trace-of-Frobenius function. The second description (Theorem \ref{thm:Borel sequence}) will provide for us a framework wherein we can establish that the trace-of-Frobenius functions associated to $(T,G)$-generic parabolic induction in fact coincides, up to a sign, with the character of the representation obtained by the corresponding parahoric Deligne--Lusztig induction (Theorem \ref{thm:comparison}).

\subsection{Very regular elements}

We first recall the notion of \textit{very regularity}, following \cite[Definition 5.1]{CI21-RT} \cite[Definition 4.2]{CO21}:

\begin{definition} 
  We say $\gamma \in \cG_{\x,0}$ is \textit{very regular} if:
  \begin{enumerate}
    \item the connected centralizer $T_\gamma$ of $\gamma$ in $G$ is a maximal torus,
    \item the apartment of $T_{\gamma}$ contains $\x$,
    \item $\alpha(\gamma) \not\equiv 1$ modulo $\mfp$ for all roots $\alpha$ of $T_\gamma$ in $G$.
  \end{enumerate}
  We say an element in $G_r$ is very regular if it is the image of a very regular element of $\cG_{\x,0}$.
\end{definition}

We write $G_{r,\vreg}$ to denote the locus of very regular elements in $G_r$ and let $j_{\vreg} \from G_{r,\vreg} \hookrightarrow G_r$ denote the inclusion.   Note that $G_{r,\vreg}$ is a subvariety of the preimage, under the natural map $G_r \to G_0$, of the regular semisimple locus $G_{0,\rss}$ of the reductive quotient $G_0$ of $\cG_{\x,0}$. If $\x$ is hyperspecial, then in fact $G_{r,\vreg}$ is the entire preimage of $G_{0,\rss}$.

Let $T_{r,\vreg} \colonequals T_r \cap G_{r,\vreg}$. Set
\begin{equation*}
  \widetilde G_{r,\vreg} \colonequals \{(g,hT_r) \in G_r \times G_r/T_r : h^{-1} g h \in T_r\}
\end{equation*} 
and consider the maps
\begin{equation*}
  \begin{tikzcd}
    & \widetilde G_{r,\vreg} \ar{dl}[above left]{f_{\vreg}} \ar{dr}{\pi_{\vreg}} \\
    T_{r,\vreg} && G_{r,\vreg}
  \end{tikzcd}
\end{equation*}
given by:
\begin{equation*}
  f_{\vreg}(g,hT_r) = h^{-1} g h, \qquad \pi_{\vreg}(g,hT_r) = g.
\end{equation*}

\begin{lemma}\label{lem:tilde Gvreg}
  The map $(g,hT_r) \mapsto (g,hB_r)$ defines an isomorphism
  \begin{equation*}
    \widetilde G_{r,\vreg} \cong \{(g,hB_r) \in G_{r,\vreg} \times G_r/B_r : h^{-1}gh \in B_r\} \subset \widetilde G_r.
  \end{equation*}
  Moreover, under this isomorphism, $f_{\vreg}$ and $\pi_{\vreg}$ correspond to restrictions of the maps $f$ and $\pi$ defined in Definition \ref{def:f and pi} and are both $W$-torsors.
\end{lemma}

\begin{proof}
  Note that in this setting, since $T_r$ is commutative, Definition \ref{def:f and pi} can be simplified: we may define $f \from \widetilde G_r \to T_r$ via $f(g, hB_r) = p(h^{-1} g h)$ and then the $\widetilde{f^* M}$ for $f \from \widehat G_r \to T_r$ in Lemma \ref{lem:pInd} is simply $f^*M$. 

  By \cite[Proposition 5.5]{CI21-RT}, if $g \in G_{r,\vreg}$ and $h \in G_r$ is such that $h^{-1} g h \in B_r$, then there exists a unique $w \in W$ such that for any lift $\dot w$, we have $h \in \dot w B_r$. Furthermore, it follows from this that we have $\dot w^{-1} g \dot w \in T_{r,\vreg} U_r$. For any $u \in U_r$ and any $t \in T_{r,\vreg}$, the very regularity of $t$ implies that there exists a unique $v \in U_r$ such that $t^{-1} v t \cdot v^{-1} = u$. Applying this to $\dot w^{-1} g \dot w = tu$, we see that $\dot w v T_r$ is the unique element of $G_r/T_r$ such that $(\dot w v)^{-1} g (\dot w v) \in T_r$ and such that $\dot w v B_r = hB_r$. 
\end{proof}

\subsection{Intersection cohomology complexes}\label{subsec:IC generic}

Lusztig's conjecture \cite[8(a)]{MR2181813} (see Section \ref{subsec:intro_lusztig}) predicts that if $X_\psi$ is $(T,G)$-generic, then for any multiplicative local system $\cL \in D_{T_r}^\psi(T_r)$, the positive-depth parabolic induction $\pInd_{B_r}^{G_r}(\cL)$ is an IC sheaf.  We can now establish this:

\begin{theorem}\label{thm:IC vreg}
  For any multiplicative local system $\cL \in D_{T_r}^\psi(T_r)$, we have
  \begin{equation*}
    \pInd_{B_r}^{G_r}(\cL[\dim T_r]) \cong (j_{\vreg})_{!*} ((\pi_{\vreg})_! f_{\vreg}^* \cL_{\vreg}[\dim G_r]),
  \end{equation*}
  where $\cL_{\vreg}$ denotes the restriction of $\cL$ to $T_r \cap G_{r,\vreg}$.
\end{theorem}

\begin{proof}
  By Theorem \ref{thm:equivalence}, we know that $\pInd_{B_r}^{G_r}(\cL[\dim T_r])$ is a simple perverse sheaf. By Lemma \ref{lem:tilde Gvreg}, we see that 
  \begin{equation*}
    j_{\vreg}^* \pInd_{B_r}^{G_r}(\cL[\dim T_r]) \cong (\pi_{\vreg})_! f_{\vreg}^* \cL_{\vreg}[\dim G_r].
  \end{equation*}
  The desired conclusion then follows by properties of intermediate extension.
\end{proof}

\begin{corollary}\label{cor:independence of B}
  Let $\cL \in D_{T_r}^\psi(T_r)$ be a multiplicative local system. 
  \begin{enumerate}[label=(\alph*)]
    \item $\pInd_{B_r}^{G_r}(\cL)$ is independent of the choice of $B$ containing $T$.
    \item Let $\sigma \from G_r \to G_r$ be an automorphism such that $\sigma(T_r) = T_r$. If $\cL$ is $\sigma$-equivariant, then so is $\pInd_{B_r}^{G_r}(\cL)$.
  \end{enumerate}
\end{corollary}

\begin{proof}
  By definition, the right-hand side of the displayed equation in Theorem \ref{thm:IC vreg} is independent of the choice of $B$ containing $T$, so (a) holds. To see (b), we note that
  \begin{equation*}
    \sigma^* \pInd_{B_r}^{G_r}(\cL) \cong \pInd_{\sigma^{-1}(B_r)}^{G_r}(\sigma^* \cL) \cong \pInd_{\sigma^{-1}(B_r)}^{G_r}(\cL) \cong \pInd_{B_r}^{G_r}(\cL),
  \end{equation*}
  where the first isomorphism holds by base-change, the second isomorphism holds by hypothesis, and the third isomorphism holds by (a).
\end{proof}

\begin{remark}\label{rem:small p}
  Technically, Lusztig's conjecture \cite[8(a)]{MR2181813} predicts that if $X_\psi$ satisfies $\mathfrak{ge1}$, then for any multiplicative local system $\cL \in D_{T_r}^\psi(T_r)$, the sheaf $\pInd_{B_r}^{G_r}(\cL)$ is an IC sheaf. As noted after Definition \ref{def:generic element}, under mild conditions on $p$, condition $\mathfrak{ge1}$ implies $\mathfrak{ge2}$, so the setting left unresolved is a small-prime phenomenon. We predict the statement of Theorem \ref{thm:IC vreg} should also hold in this slightly more general setting of requiring only $\frak{ge1}$. We have verified this in the setting that $G = \SL_2$ and $p=2$: there are $W$-invariant multiplicative local systems $\cL$ satisfying $\frak{ge1}$, and yet $\pInd_{B_r}^{G_r}(\cL)$ is still an IC sheaf.
\end{remark}

\subsection{Character sheaves associated to multiplicative local systems on $T_r$}

In this subsection, we resume the assumptions in place in Section \ref{subsec:howe}. We make these assumptions so that we may state our result for arbitrary multiplicative local systems on $T_r$. (Alternatively, one could choose to forgo the assumption on $p$ and instead only allow multiplicative local systems which have a Howe factorization: that is, any multiplicative local system $\cL \in D_{T_r}(T_r)$ for which there exists a clipped generic datum $(T, \vec G, \x, \vec r, \vec \cL)$ such that $\cL_{-1} \colonequals \cL \otimes (\cL_0|_{T_r} \otimes \cL_1|_{T_r} \otimes \cdots \otimes \cL_d|_{T_r})$ factors through $T_r \to T_0$. Here, each $\cL_i|_{T_r}$ is given by pullback along $T_r \to T_{r_i} \hookrightarrow G_{r_i}^i$.) We prove the ``multi-step'' version of Theorem \ref{thm:IC vreg}.

\begin{theorem}\label{thm:K_L}
  Let $\cL \in D_{T_r}(T_r)$ be a multiplicative local system. For any Howe factorization $\vec \cL$ and any associated sequence of parabolic subgroups $\vec P$, 
  \begin{equation*}
    \cK_{\vec \cL, \vec P}[\dim T_r] \cong (j_{\vreg})_{!*} ((\pi_{\vreg})_! f_{\vreg}^* \cL_{\vreg}[\dim G_r]).
  \end{equation*}
  In particular, $\cK_{\vec \cL, \vec P}$ depends only on $\cL$.
\end{theorem}

\begin{proof}
  We know from Theorem \ref{thm:clipped datum} that $\cK_{\vec \cL, \vec P}$ is semisimple and the endomorphism algebra has dimension equal to $|\Stab_{W^0}(\cL_{-1})| = |\Stab_W(\cL)|$. The same is true of the intermediate extension $(j_{\vreg})_{!*} ((\pi_{\vreg})_! f_{\vreg}^* \cL_{\vreg}[\dim G_r])$. Therefore to see that they are isomorphic, it is enough to show that the restriction of $\cK_{\vec \cL, \vec P}$ to $G_{r,\vreg}$ coincides with $(\pi_{\vreg})_! f_{\vreg}^* \cL_{\vreg}$ (up to the specified shift). To do this, we proceed by induction. 
  
  We first observe that by construction, the very regular locus $L_{r,\vreg}$ in $L_r$ contains the intersection $L_r \cap G_{r,\vreg}$. Since $L_r \cap G_{r,\vreg}$ is still dense in $L_r$, the intermediate extension from $L_{r,\vreg}$ agrees with the intermediate extension from $L_r \cap G_{r,\vreg}$. Our task is then to prove that for any Levi $L \subset G$ containing $T$, any $s < r$, any multiplicative local system $\cL$ on $T_s$, any $(L,G)$-generic multiplicative local system $\cL'$ on $L_r$, we have
  \begin{align}\nonumber
    \pInd_{P_r}^{G_r}(\varphi^* (j_{s,\vreg}^L)_{!*}&((\pi_{s,\vreg}^L)_! (f_{s,\vreg}^L)^* \cL_{\vreg}) \otimes \cL'[\dim L_r]) \\ \label{eq:vreg induction}
    &\cong (j_{\vreg})_{!*} ((\pi_\vreg)_! f_{\vreg}^* (\cL \otimes \cL'|_{T_r})_{\vreg}[\dim G_r]),
  \end{align}
  where:
  \begin{enumerate}[label=\textbullet]
    \item $\varphi \from L_r \to L_s$
    \item $j_{s,\vreg}^L \from L_s \cap G_{s,\vreg} \hookrightarrow L_s$
    \item $\pi_{s,\vreg}^L = \pi^L|_{\widetilde L_{s,\vreg}}$ and $f_{s,\vreg}^L = f^L|_{\widetilde L_{s,\vreg}}$, where $\widetilde L_{s,\vreg} = \pi^{-1}(L_s \cap G_{s,\vreg})$ and $\widetilde L_s$ is relative to $T \hookrightarrow L$,
    \item $\cL_\vreg = \cL|_{T_s \cap L_{s,\vreg}}$
    \item $j_{\vreg} \from G_{r,\vreg} \hookrightarrow G_r$
    \item $\pi_{\vreg} = \pi|_{\widetilde G_{r,\vreg}}$ and $f_{\vreg} = f|_{\widetilde G_{r,\vreg}}$, where again $\widetilde G_r$ is relative to $T$
    \item $(\cL \otimes \cL'|_{T_r})_{\vreg} = (\cL \otimes \cL'|_{T_r})_{T_r \cap G_{r,\vreg}}$.
  \end{enumerate}
  It is straightforward to see that the argument of $\pInd_{P_r}^{G_r}$ simplifies to
  \begin{align*}
    &(j_{\vreg}^L)_{!*} (\varphi_{\vreg}^* (\pi_{s,\vreg}^L)_! (f_{s,\vreg}^L)^* \cL_\vreg) \otimes \cL' \\
    \cong {}&{} (j_{\vreg}^L)_{!*} ((\pi_{r,\vreg}^L)_! (f_{r,\vreg}^L)^* (\varphi_{\vreg}^T)^* \cL_\vreg) \otimes \cL' \\
    \cong {}&{} (j_{\vreg}^L)_{!*} ((\pi_{r,\vreg}^L)_! (f_{r,\vreg}^L)^* ((\varphi_{\vreg}^T)^* \cL \otimes \cL'|_{T_r})_{\vreg}),
  \end{align*}
  where $\varphi_\vreg \from L_r \cap G_{r,\vreg} \to L_s \cap G_{s,\vreg}$ and $\varphi_{\vreg}^T = \varphi_{\vreg}|_{T_r \cap G_{r,\vreg}}$ and the first isomorphism holds by base-change and the second isomorphism holds by properties of intermediate extension. So now, to show \eqref{eq:vreg induction}, it remains to prove that
  \begin{align}\nonumber
    \pInd_{P_r}^{G_r}((j_{\vreg}^L)_{!*} &((\pi_{r,\vreg}^L)_! (f_{r,\vreg}^L)^* ((\varphi_{\vreg}^T)^* \cL \otimes \cL'|_{T_r})_{\vreg}))|_{G_{r,\vreg}} \\ \label{eq:vreg induction 2}
    &\cong (\pi_\vreg)_! f_{\vreg}^* (\cL \otimes \cL'|_{T_r})_{\vreg}[\dim G_r - L_r].
  \end{align}
  By base-change, the left-hand side becomes
  \begin{equation}\label{eq:vreg pInd of vreg}
    (\pi_{\vreg}')_! ((f_{\vreg}')^* (\pi_{r,\vreg}^L)_! (f_{r,\vreg}^L)^* ((\varphi_{\vreg}^T)^* \cL \otimes \cL'|_{T_r})_{\vreg})^{\sim}[\dim G_r - \dim L_r],
  \end{equation}
  where $\pi_{\vreg}' = \pi'|_{\widetilde G_{r,\vreg}'}$ and $f_{\vreg}' = f'|_{\widehat G_{r,\vreg}'}$ for $\widetilde G_{r,\vreg}' = \pi'{}^{-1}(G_{r,\vreg}) \subset G_{r,\vreg} \times G_r/P_r$ and $\widehat G_{r,\vreg}' = f'{}^{-1}(L_r \cap G_{r,\vreg}) \subset G_{r,\vreg} \times G_r$ (i.e., these spaces are taken to be relative to $L \hookrightarrow G$). Here, for $M$ on $\widehat G_{r,\vreg}'$, we write $M^\sim = \widetilde M$ to mean the unique object on $\widetilde G_{r,\vreg}'$ such that $\alpha_{\vreg}^* M^\sim \cong M$, where $\alpha_{\vreg} \from \widehat G_{r,\vreg}' \to \widetilde G_{r,\vreg}'$. The fibered product $\widehat G_{r,\vreg}' \times_{L_r \cap G_{r,\vreg}} \widetilde L_{r,\vreg}$ is a $L_r$-fibration over $\widehat G_r$, and so it follows that \eqref{eq:vreg pInd of vreg} agrees with the right-hand side of \eqref{eq:vreg induction 2}.
\end{proof}

\section{Sequences of Borel subgroups}\label{sec:Borel sequence}

In this section, we give a description of $(T,G)$-generic parabolic induction in terms of sequences 
\begin{equation*}
  \underline B = (B^{(1)}, B^{(2)}, \ldots, B^{(n+1)})
\end{equation*}
of Borel subgroups of $G$, each of which contains $T$. In Section \ref{sec:comparison}, we will use Theorem \ref{thm:Borel sequence} in the special case $B^{(i)} = \sigma^{i-1}(B)$, where $\sigma$ is a Frobenius morphism on $G$ associated to an $F$-rational structure, and $n$ is a positive integer such that $\sigma^n(B) = B$.

Fix a sequence $\underline B$ as above and assume $B^{(1)} = B^{(n+1)}$.  Let $\underline B_r = (B_r^{(1)}, B_r^{(2)}, \ldots, B_r^{(n+1)})$ denote the associated sequence of subgroups of $G_r$ and write $\underline \Omega = (\Omega_1, \Omega_2, \ldots, \Omega_n)$ where $\Omega_i = B_r^{(i)}B_r^{(i+1)}$. Consider the following subvariety of $G_r \times G_r/B_r^{(1)} \times G_r/B_r^{(2)} \times \cdots \times G_r/B_r^{(n+1)}$:
\begin{equation*}
  Y_{\underline \Omega}
  \colonequals \left\{
  (g,h_1B_r^{(1)}, h_2B_r^{(2)}, \ldots, h_{n+1}B_r^{(n+1)}) :  
  \begin{gathered}
    \text{$h_i^{-1}h_{i+1} \in \Omega_i$ for $i = 1, \ldots, n$} \\
    h_{n+1}^{-1} g h_1 \in B_r^{(1)}
  \end{gathered}
  \right\}
\end{equation*}
together with the maps
\begin{align*}
  \pi_{\underline \Omega}(g,h_1B_r^{(1)}, \ldots, h_{n+1}B_r^{(n+1)}) &= g \\
  f_{\underline \Omega}(g,h_1B_r^{(1)}, \ldots, h_{n+1}B_r^{(n+1)}) &= \beta_{\Omega_1}(h_1^{-1}h_2) \cdot \beta_{\Omega_2}(h_2^{-1}h_3) \cdots \beta_{\Omega_n}(h_n^{-1}h_{n+1}) \cdot \beta(h_{n+1}^{-1}gh_1)
\end{align*}
where $\beta \from B_r^{(1)} \to T_r$ and $\beta_{\Omega_i} \from \Omega_i \to T_r$ for $1 \leq i \leq n$. (Note here that the last two factors in $f_{\underline \Omega}$ can be combined: $\beta_{\Omega_n}(h_n^{-1}h_{n+1})\beta(h_{n+1}^{-1} g h_1) = \beta_{\Omega_n}(h_n^{-1}gh_1)$.)

\begin{theorem}\label{thm:Borel sequence}
  Let $\underline \Omega$ be as above. For any $M \in D_{T_r}^\psi(T_r)$,
  \begin{equation*}
    \pInd_{B_r}^{G_r}(M) \cong \pi_{\underline \Omega!} f_{\underline \Omega}^* M[2\dim U_r+\ell(\underline \Omega)],
  \end{equation*}
  where $\ell(\underline \Omega) \colonequals \sum_{i=1}^n \dim(\Omega_i/B_r^{(i+1)})$.
\end{theorem}

Our proof of Theorem \ref{thm:Borel sequence} is  is modeled on Lusztig's methods \cite{Lus90} (see \cite{MR1040575} for an exposition on Lusztig's methods in the Borel case). In the classical $r=0$ setting, any two Borel subgroups which contain the same maximal torus are conjugate by an element of the Weyl group. Thus, one can associate to $\underline B_0$ a sequence of elements $\underline w$ in the Weyl group and prove the $r=0$ case of Theorem \ref{thm:Borel sequence} by working with $\underline w$. 

The naive generalization of this strategy immediately fails in the $r>0$ setting: not only can one not associate such a $\underline w$ to $\underline B_r$, but the elements in $\underline B_r$ may not even be conjugate in $G_r$! The simplest example of this is to take an Iwahori subgroup and consider two opposing Borel subgroups. In Section \ref{subsec:Borel sequence proof} we prove Theorem \ref{thm:Borel sequence} via a double induction. The idea is to modify $\underline B$, keeping track of the resulting impact on $\pi_{\underline \Omega!} \circ f_{\underline \Omega}^*$. We first establish that we may reduce the calculation to the setting where each consecutive pair $B^{(i)}, B^{(i+1)}$ of Borel subgroups $\underline B$ is one of the following two types:
\begin{enumerate}
  \item $B^{(i)}$ and $B^{(i+1)}$ are related by a simple reflection in $W_{G_r}(T_r)$ (``elementary'')
  \item $B_0^{(i)} = B_0^{(i+1)}$ (``0-trivial'')
\end{enumerate}
Establishing this first reduction follows from Section \ref{subsec:elementary}. We then proceed by inducting on both $\ell(\underline \Omega)$ and the number of times Case(2) appears; this involves the two lemmas in Section \ref{subsec:two vanishing}. If Case(2) does not appear in $\underline B$, then we use Lemma \ref{lem:vanishing elementary}. If Case(2) appears, the number of such instances can be reduced using Lemma \ref{lem:vanishing 0-trivial}; this induction crucially uses the $(T,G)$-genericity of $X_\psi$.

\subsection{Elementary double cosets}\label{subsec:elementary}

\begin{definition}[elementary double coset]\label{def:elementary}
  Let $B,B'$ be two Borel subgroups of $G$ which contain $T$. We say that $B_rB_r'$ is \textit{elementary} if there exists a simple reflection $s \in W_{G_r}(T_r)$ relative to $B$ such that $B' = sBs^{-1}$.
\end{definition}

Let $\ell_B \from W_{G_r}(T_r) \to \bbZ_{\geq 0}$ denote the length function relative to $B$.

\begin{lemma}\label{lem:bruhat multiplication}
  Let $\alpha$ be a simple root and $s \in W_{G_r}(T_r)$ the corresponding simple reflection. Then for $w \in W_{G_r}(T_r)$,
  \begin{align*}
    B_r s B_r \cdot B_r w B_r =
    \begin{cases}
      B_r sw B_r & \text{if $\ell_B(sw) = \ell_B(w) + 1$} \\
      B_r U_{-\alpha,0+:r+} sw B_r \sqcup B_r w B_r & \text{if $\ell_B(sw) = \ell_B(w) - 1$.}
    \end{cases}
  \end{align*}
\end{lemma}

\begin{proof}
We have 
\begin{equation*}
  B_r w B_r = \prod_{\substack{\beta \in \Phi^+(G,T) \\ \text{s.t.\ } w^{-1}\beta < 0}} U_{\beta,r} w B_r
\end{equation*}
and therefore
\begin{equation}\label{eq:bruhat product}
  B_r s B_r \cdot B_r w B_r = B_r s \prod_{\substack{\beta \in \Phi^+(G,T) \\ \text{s.t.\ } w^{-1}\beta < 0}} U_{\beta,r} w B_r.
\end{equation}

Assume $\ell_B(sw) = \ell_B(w) + 1$. Then $w^{-1}(\alpha) > 0$, and in particular, if $\beta \in \Phi^+(G,T)$ is such that $w^{-1}\beta < 0$, then $\beta \neq \alpha$. Since $s$ permutes all positive roots other than $\alpha$, we have that
\begin{equation*}
  s \prod_{\substack{\beta \in \Phi^+(G,T) \\ \text{s.t.\ } w^{-1}\beta < 0}} U_{\beta,r} s \subset B_r.
\end{equation*}
Therefore using \eqref{eq:bruhat product}, we have $B_r s B_r \cdot B_r w B_r = B_r sw B_r$, as desired.

Assume $\ell(sw) = \ell(w) + 1$. By \eqref{eq:bruhat product}, we have
\begin{equation*}
  B_r s B_r \cdot B_r s B_r = U_{\alpha,r} s U_{\alpha,r} s B_r = U_{\alpha,r} U_{-\alpha,r} B_r.
\end{equation*}
By doing an $\SL_2$ calculation, we see that
\begin{equation*}
  U_{\alpha,r} (U_{-\alpha,r} \smallsetminus U_{-\alpha,0+:r+}) B_r = B_r s B_r.
\end{equation*}
For the complement of this piece, by the Iwahori decomposition we have
\begin{equation*}
  U_{\alpha,r}U_{-\alpha,0+:r+}B_r = U_{-\alpha,0+:r+} B_r.
\end{equation*}
This therefore gives
\begin{equation*}
  B_r s B_r \cdot B_r s B_r = U_{-\alpha,0+:r+}B_r \sqcup B_r s B_r.
\end{equation*}
Using this together with the length-additive case, we get
\begin{align*}
  B_r s B_r \cdot B_r w B_r 
  &= B_r s B_r \cdot B_r s B_r \cdot B_r sw B_r \\
  &= (U_{-\alpha,0+:r+}B_r \sqcup B_r s B_r) \cdot B_r s w B_r \\
  &= B_r U_{-\alpha,0+:r+} sw B_r \sqcup B_r w B_r. \qedhere
\end{align*}
\end{proof}

\begin{lemma}[length-additive composition]\label{lem:elementary composition}
  Let $(B,B',B'')$ be a subsequence of $\underline B$ such that $B_r' = s B_r s^{-1}$ and $B_r'' = swB_rw^{-1}s^{-1}$ for some $w,s \in W_{G_r}(T_r)$, where $s$ is a simple reflection relative to $B.$ If $\ell_B(sw) = \ell_B(w)+1$, then 
  \begin{equation*}
    \pi_{\underline \Omega!} f_{\underline \Omega}^* \cL \cong (\pi_{\underline \Omega'})f_{\underline \Omega'}^* \cL,
  \end{equation*}
  where $\underline \Omega'$ is obtained from $\underline \Omega$ by replacing $(B_rB_r', B_r'B_r'')$ with $B_rB_r''$ (i.e.\ deleting $B'$ from $\underline B$).
\end{lemma}

\begin{proof}
  The proof relies on establishing the isomorphism \eqref{eq:add length} below. This is well known and classical when $r=0$ and the proof when $r>0$ is the same; we present it here for the convenience of the reader. Since $\ell_B(sw) = \ell_B(w) + 1$, then by Lemma \ref{lem:bruhat multiplication} we have
  \begin{equation*}
    B_r s B_r \cdot B_r w B_r = B_r sw B_r.
  \end{equation*}
  This implies that
  \begin{align*}
    B_r(sB_rs^{-1}) \cdot (sB_rs^{-1})(swB_rw^{-1}s^{-1}) 
    &= B_r s B_r \cdot B_r w B_r w^{-1} s^{-1}
    &= B_r (swB_r w^{-1}s^{-1}).
  \end{align*}
  Moreover, we have an isomorphism
  \begin{equation}\label{eq:add length}
    (B_r (sB_rs^{-1})) \cdot (sB_rs^{-1})(swB_rw^{-1}s^{-1})/sB_rs^{-1} \to B_r(swB_rw^{-1}s^{-1}),
  \end{equation}
  where the quotient is taken under the action $b' \cdot (g,g'') = (gb',b'{}^{-1}g'')$ for $b' \in sB_rs^{-1}$ and $(g,g'') \in B_r(sB_rs^{-1}) \times (sB_rs^{-1})(swB_rw^{-1}s^{-1})$. Hence the map $Y_{\underline \Omega} \to Y_{\underline \Omega'}$ obtained by deleting the $h'B_r'$ term is an isomorphism, and the desired conclusion follows.
\end{proof}

We immediately obtain the following two corollaries.

\begin{corollary}\label{cor:composition}
  Let $(B,B',B'')$ be a subsequence of $\underline B$ such that $B_r' = w B_r w^{-1}$ and $B_r'' = ww' B_r w^{\prime-1}w^{-1}$ for some $w,w' \in W_{G_r}(T_r)$ such that $\ell_B(ww') = \ell_B(w) + \ell_B(w')$. Then
  \begin{equation*}
    \pi_{\underline \Omega!} f_{\underline \Omega}^* \cL \cong \pi_{\underline \Omega'!} f_{\underline \Omega'}^* \cL,
  \end{equation*}
  where $\underline \Omega'$ is obtained from $\underline \Omega$ by replacing $(B_rB_r',B_r'B_r'')$ by $B_rB_r''$ (i.e.\ deleting $B'$ from $\underline B$).
\end{corollary}

\begin{corollary}\label{cor:elementary expansion}
  Let $(B,B')$ be a subsequence of $\underline B$ and let $B_rB_r'$ be the corresponding element of $\underline \Omega$. Assume that $B' = s_1s_2 \cdots s_m B s_m \cdots s_2s_1$ for some reduced word $s_1 \cdots s_m$ relative to $B$. Then
  \begin{equation*}
    \pi_{\underline \Omega!} f_{\underline \Omega}^* \cL \cong \pi_{\underline \Omega'!} f_{\underline \Omega'}^* \cL
  \end{equation*}
  where $\underline \Omega'$ is obtained from $\underline \Omega$ by replacing $B_rB_r'$ with $(B_r(s_1B_rs_1), (s_1B_rs_1)(s_1s_2B_rs_2s_1), \ldots, \\(s_1s_2\cdots s_{m-1}B_rs_{m-1} \cdots s_1s_2)(s_1s_2 \cdots s_{m-1}s_mB_rs_ms_{m-1} \cdots s_2s_1))$ (i.e.\ inserting the Borel subgroups $s_1Bs_1, s_1s_2Bs_2s_1, \ldots, s_1s_2 \cdots s_{m-1}Bs_{m-1}\cdots s_2s_1$ between $B$ and $B'$ in $\underline B$).
\end{corollary}

\subsection{Two vanishing lemmas}\label{subsec:two vanishing}

\begin{lemma}[$(s,s)$ composition]\label{lem:vanishing elementary}
  If $(B,B',B)$ is a subsequence of $\underline B$ such that $B_rB_r'$ is elementary, then for any $M \in D_{T_r}^\psi(T_r)$,
  \begin{equation*}
    \pi_{\underline \Omega!} f_{\underline \Omega}^* M \cong \pi_{\underline \Omega'!} f_{\underline \Omega'}^* M[2 \dim(B_rB_r'/B_r')],
  \end{equation*}
  where $\underline \Omega'$ is obtained from $\underline \Omega$ by deleting $(B_rB_r', B_r'B_r)$ (i.e.\ by deleting $(B,B')$ in $\underline B$).
\end{lemma}

\begin{proof}
  Assume that $B,B',B$ occur as $B^{(i)}, B^{(i+1)}, B^{(i+2)}$ in $\underline B$ and let $s$ be the simple reflection such that $B' = sBs$. Recall that
  \begin{equation*}
    (B_r B_r')(B_r'B_r) = (B_r sB_rs)(sB_rs B_r) = U_{-\alpha,0+:r+}B_r \sqcup B_r s B_r.
  \end{equation*}
  It is an easy calculation to see that 
  \begin{equation*}
    (B_r B_r')(B_r'B_r) = B_r \sqcup B_r u_{-\alpha}(\varpi) B_r \sqcup B_r u_{-\alpha}(\varpi^2) B_r \sqcup \cdots \sqcup B_r u_{-\alpha}(\varpi^r)B_r \sqcup B_r s B_r.
  \end{equation*}
  Define
  \begin{align*}
    Y_{\underline \Omega}' &= \{(g, \underline{hB_r}) \in Y_{\underline \Omega} : h_i^{-1} h_{i+2} \in B_r\}, \\
    Y_{\underline \Omega,k}'' &= \{(g, \underline{hB_r}) \in Y_{\underline \Omega} : h_i^{-1} h_{i+2} \in B_r u_{-\alpha(\varpi^k)}B_r\}, \\
    Y_{\underline \Omega,s}'' &= \{(g, \underline{hB_r}) \in Y_{\underline \Omega} : h_i^{-1} h_{i+2} \in B_r s B_r\}.
  \end{align*}
  Define $\underline \Omega', \underline \Omega_k'', \underline \Omega_s''$ to be the sequences of double cosets associated to
  \begin{align*}
    \underline B' &= (B^{(1)}, \ldots, B^{(i-1)}, B^{(i)}, B^{(i+3)}, \ldots, B^{(n+1)}), \\
    \underline B_k'' &= (B^{(1)}, \ldots, B^{(i-1)}, B^{(i)}, u_{-\alpha}(\varpi^k)B^{(i+2)}u_{-\alpha}(-\varpi^k), 
    \ldots, u_{-\alpha}(\varpi^k)B^{(n+1)}u_{-\alpha}(-\varpi^k)), \\
    \underline B_s'' &= (B^{(1)}, \ldots, B^{(i-1)}, B^{(i)}, sB^{(i+2)}s, 
    \ldots, sB^{(n+1)}s).
  \end{align*}
  For shorthand, let us write $x$ to either mean $k$ (and $u_{-\alpha}(\varpi^k)$) or $s$.
  We then have projection maps $p' \from Y_{\underline \Omega' \to Y_{\underline \Omega'}}$, $p_x'' \from Y_{\underline \Omega,x}'' \to Y_{\underline \Omega_x''}$ defined by 
  \begin{align*}
    p'(g, \underline{hB_r}) &= (g, h_1B_r^{(1)}, \ldots, h_{i-1}B_r^{(i-1)}, h_i B_r^{(i)}, h_{i+3}B_r^{(i+3)}, \ldots, h_{n+1}B_r^{(n+1)}), \\
    p_x''(g, \underline hB_r)
    &= (g, h_1B_r^{(1)}, \ldots, h_{i-1}B_r^{(i-1)}, h_i B_r^{(i)}, h_{i+2} x^{-1} \cdot x B_r^{(i+2)} x^{-1}, \ldots, h_{n+1} x^{-1} \cdot x B_r^{(n+1)} x^{-1}).
  \end{align*}
  It is a straightforward check to see that the condition on $(h_iB_r, h_{i+2}B_r)$ defining $Y_{\underline \Omega}', Y_{\underline \Omega,x}''$ guarantees that $p', p_x''$ are well defined and surjective. 

  We may organize all these maps into the commutative diagram
  \begin{equation*}
    \begin{tikzcd}
      (\sqcup_{k=1}^r Y_{\underline \Omega,k}'') \sqcup Y_{\underline \Omega,s}'' \ar{d}[left]{(\sqcup_{k=1}^r p_k'') \sqcup p_s''} \ar[hookrightarrow]{rrr}{(\sqcup_{k=1}^r j_k'') \sqcup j_s''} &&& Y_{\underline \Omega} \ar[hookleftarrow]{r}{j'} \ar{dd}{\pi_{\underline \Omega}} & Y_{\underline \Omega}' \ar{d}{p'} \\
      (\sqcup_{k=1}^r Y_{\underline \Omega_k''}) \sqcup Y_{\underline \Omega_s''} \ar{drrr}[below left]{(\sqcup_{k=1}^r \pi_{\underline \Omega_k''}) \sqcup \pi_{\underline \Omega_s''}} & & & & Y_{\underline \Omega'} \ar{dl}{\pi_{\underline \Omega'}} \\
      &&& G_r
    \end{tikzcd}
  \end{equation*}
  Roughly speaking, our aim is to show that $p'$ is an affine fibration (Claim \ref{claim:p'}) and that the entire left side of this diagram does not contribute (Claims \ref{claim:k vanishing} and \ref{claim:s vanishing}), therefore reducing the information on $Y_{\underline \Omega}$ to information on $Y_{\underline \Omega'}$.

  \begin{claim}\label{claim:p'}
    We have $\pi_{\underline \Omega'!} p_!' j^{\prime*} f_{\underline \Omega}^* M \cong \pi_{\underline \Omega'!} f_{\underline \Omega'}^* M[2 \dim (B_rB_r'/B_r')]$. 
  \end{claim}

  \begin{proof}
    It is enough to show that for any $y' \in Y_{\underline \Omega'}$, we have $p'{}^{-1}(y') \cong \bbA^N$ for $N = \dim(B_rB_r'/B_r')$ and $(j'{}^* f_{\underline \Omega}^*\cL)|_{p'{}^{-1}(y')}$ is the constant sheaf. By definition, 
    \begin{equation*}
      p'{}^{-1}(y') \cong \{h_{i+1}B_r' \in G_r/B_r' : \text{$h_i^{-1}h_{i+1} \in B_rB_r'$ and $h_{i+1}^{-1}h_{i+2} \in B_r'B_r$}\},
    \end{equation*}
    where we recall that $h_i,h_{i+2}$ are determined by $y'$ and $h_i^{-1}h_{i+2} \in B_r$. Hence we have
    \begin{equation*}
      f'{}^{-1}(y') 
      = h_i B_r B_r'/B_r'.
    \end{equation*}
    Moreover, this is an affine space; explicitly, $B_rB_r'/B_r' \cong (U_r \cap U_r^{\prime-})$.

    Next we show constancy of $(j^{\prime *} f_{\underline \Omega}^* M)|_{p^{\prime -1}(y')}$. Let $y = (g, \underline{hB_r}) \in p'{}^{-1}(y')$ be the lift of $y'$ with entry $h_iB_r^{(i)} = h_iB_r$. Then by the above, any other lift of $y'$ is obtained by modifying $h_iB_r$ to $h_iu$ where $u \in U_r \cap U_r^{\prime-}$; for convenience, we denote this element by $yu$. We have
    \begin{align*}
      \frac{f_{\underline \Omega}(yu)}{f_{\underline \Omega}(y)}
      &= \beta_{B_rB_r'}(h_i^{-1}h_iu) \beta_{B_r'B_r}(u^{-1}h_i^{-1}h_i) = \beta_{B_rB_r'}(u) \cdot \beta_{B_r'B_r}(u^{-1}) = 1,
    \end{align*}
    which shows constancy.
  \end{proof}

  For the rest of the proof, we focus on showing
  \begin{equation}\label{eq:Y'' vanishing 2}
    \pi_{\underline \Omega_x''!} p_{x!}'' j_x''{}^* f_{\underline \Omega}^* M = 0
  \end{equation}
  for all possibilities of $x$. We will make crucial use of the assumption $M \in D_{T_r}^\psi(T_r)$.

  Without loss of generality, assume that $h_{i+2} = h_i x$. Choose any $y'' \in Y_{\underline \Omega_x''}$.

  \begin{claim}\label{claim:k vanishing}
    For $1 \leq k \leq r$, the vanishing statement \eqref{eq:Y'' vanishing 2} holds.
  \end{claim}

  \begin{proof}
    We have 
    \begin{align*}
      p_k^{\prime \prime -1}(y'') 
      &= \{h_{i+1} B_r' : h_i^{-1}h_{i+1} \in B_r B_r', \, h_{i+1}^{-1}h_{i+2} \in B_r'B_r\} \\
      &\cong h_i \{u B_r' \in U_{\alpha,r}B_r' : u \in u_{-\alpha}(\varpi^k) U_{\alpha,r} B_r'\}.
    \end{align*}
    We have
    \begin{equation*}
      u_{-\alpha}(\varpi^k) u_{\alpha}(y) \left(\begin{smallmatrix}
        a & 0 \\ b & c
      \end{smallmatrix}\right) \in U_{\alpha,r}
    \end{equation*}
    if and only if
    \begin{equation*}
      a = 1 + \varpi^k y, \quad b = -\varpi^k, \quad c = (1+\varpi^k y)^{-1},
    \end{equation*}
    and so we see that we have an isomorphism $\varphi \from U_{\alpha,r} \xrightarrow{\sim} p_k^{\prime\prime-1}(y'')$. Moreover, this calculation shows that 
    \begin{align*}
      \frac{f_{\underline \Omega}(\varphi(u_\alpha(y)))}{f_{\underline \Omega}(\varphi(u_\alpha(0)))} 
      &= \beta_{B_rB_r'}(h_i^{-1}h_iu_{-\alpha}(\varpi^k)u_\alpha(y)) \cdot \beta_{B_r'B_r}((h_i u_{-\alpha}(\varpi^k)u_\alpha(y))^{-1} h_i u_{-\alpha}(\varpi^k)) \\
      &= \beta_{B_rB_r'}(h_i^{-1}h_iu_{-\alpha}(\varpi^k)u_\alpha(y)) \cdot \beta_{B_r'B_r}(u_\alpha(y)^{-1}) = \alpha^\vee(1+\varpi^k y).
    \end{align*}
    In particular, this holds for $u_\alpha(y) \in U_{\alpha,r:r+}$.
    Since $M$ is $(\mft,\cL_\psi)$-equivariant, the $(T,G)$-genericity of $\cL_\psi$ implies that we must have
    \begin{equation*}
      p_{k!}'' j_k^{\prime\prime*} f_{\underline \Omega}^* M = 0.
    \end{equation*}
    Indeed, we can consider an intermediate variety $\cY''$ fitting in $Y_{\underline \Omega,k}'' \to \cY'' \to Y_{\underline \Omega_k''}$ where the first map is given by replacing $h_{i+1}B_r'$ with its image in $G_{r-1}/B_{r-1}$. Then the pushforward along $p_k$ would be given by a composition of pushforwards, the first of which is fiber-wise a pushforward of a nontrivial multiplicative local system on $U_{\alpha,r:r+}$.
  \end{proof}

  \begin{claim}\label{claim:s vanishing}
    When $x = s$, the vanishing statement \eqref{eq:Y'' vanishing 2} holds.
  \end{claim}

  \begin{proof}
    We have
    \begin{equation*}
      p_s^{\prime\prime-1}(y'') = h_i\{u B_r' \in U_{\alpha,r}B_r' : u \in sU_{\alpha,r}B_r'\}.
    \end{equation*}
    If 
    \begin{equation*}
      s u_\alpha(y)\left(\begin{smallmatrix}
        a & 0 \\ b & c
      \end{smallmatrix}\right) \in U_{\alpha,r},
    \end{equation*}
    then this forces
    \begin{equation*}
      a = y, \quad b = 1, \quad cy = 1,
    \end{equation*}
    which forces $y$ to be a unit. Therefore we have an isomorphism $\varphi \from (U_{\alpha,r} \smallsetminus U_{\alpha,0+:r+}) \xrightarrow{\sim} p_s^{\prime\prime-1}(y'')$ and for any $y,y_r$ with $u_\alpha(y) \in U_{\alpha,r} \smallsetminus U_{\alpha,0+:r+}$ and $u_\alpha(y_r) \in U_{\alpha,r:r+}$,
    \begin{align*}
      \frac{f_{\underline \Omega}(\varphi(u_\alpha(yy_r)))}{f_{\underline \Omega}(\varphi(u_\alpha(y)))} 
      &= \beta_{B_rB_r'}(h_i^{-1}h_i s u_\alpha(y_r)) \beta_{B_r'B_r}((h_i s u_\alpha(y_r))^{-1} h_i s) \\
      &= \beta_{B_rB_r'}(su_\alpha(y_r)) \beta_{B_r'B_r}(u_\alpha(y_r)) = \alpha^\vee(y_r).
    \end{align*}
    Since $M$ is $(\mft,\cL_\psi)$-equivariant, then the $(T,G)$-genericity of $\cL_\psi$ implies that 
    \begin{equation*}
      p_{s!}'' j_s^{\prime\prime *} f_{\underline \Omega}^* M = 0
    \end{equation*}
    by the reasoning at the end of the proof of Claim \ref{claim:k vanishing}.
  \end{proof}

The conclusion of the lemma now holds by combining \eqref{eq:Y'' vanishing 2} with Claim \ref{claim:p'}.
\end{proof}

\begin{definition}[0-trivial double coset]
  Let $B,B'$ be two Borel subgroups of $G$ which contain $T$. We say $B_rB_r'$ is \textit{0-trivial} if $B_0 = B_0'$.
\end{definition}

\begin{lemma}[$0$-trivial composition]\label{lem:vanishing 0-trivial}
  Let $(B, B', B'')$ be a subsequence of $\underline B$ and let $(B_rB_r', B_r'B_r'')$ be the corresponding subsequence in $\underline \Omega$. If $B_rB_r'$ is $0$-trivial, then for any $M \in D_{T_r}^\psi(T_r)$,
  \begin{equation*}
    \pi_{\underline \Omega!} f_{\underline \Omega}^* M \cong \pi_{\underline \Omega'!} f_{\underline \Omega'}^* M[\dim(B_rB_r'/B_r') + \dim (B_r'B_r''/B_r'') - \dim(B_rB_r''/B_r'')],
  \end{equation*} 
  where $\underline \Omega'$ is obtained from $\underline \Omega$ by replacing the subsequence $(B_rB_r', B_r'B_r'')$ by $B_rB_r''$ (i.e.\ by deleting $B_r'$ in $\underline B$).
\end{lemma}

\begin{proof}
  For convenience, let $N = \dim(B_rB_r'/B_r') + \dim (B_r'B_r''/B_r'') - \dim(B_rB_r''/B_r'')$.
  Assume that $B,B',B''$ occur as $B^{(i)}, B^{(i+1)}, B^{(i+2)}$ in $\underline B$. Consider
  \begin{align*}
    Y_{\underline \Omega}' &= \{(g,\underline h \underline B) \in Y_{\underline \Omega} : h_i^{-1}h_{i+2} \in B_rB_r''\}, \\
    Y_{\underline \Omega}'' &= \{(g,\underline h \underline B) \in Y_{\underline \Omega} : h_i^{-1}h_{i+2} \notin B_rB_r''\},.
  \end{align*}
  Let $\underline \Omega'$ be the sequence obtained from $\underline \Omega$ by replacing the subsequence $(B_rB_r', B_r'B_r'')$ by $B_rB_r''$. We then have the commutative diagram
  \begin{equation*}
    \begin{tikzcd}
      Y_{\underline \Omega}'' \ar{ddr}[below left]{\pi_{Y_{\underline \Omega}''}} \ar[hookrightarrow]{r}{j''} & Y_{\underline \Omega} \ar[hookleftarrow]{r}{j'} \ar{dd}{\pi_{\underline \Omega}} & Y_{\underline \Omega}' \ar{d}{p'} \\
      & & Y_{\underline \Omega'} \ar{dl}{\pi_{\underline \Omega'}} \\
      & G_r
    \end{tikzcd}
  \end{equation*}
  where $p' \from Y_{\underline \Omega}' \to Y_{\underline \Omega'}$ is obtained by forgetting the term $h_{i+1}B_r'$.

  \begin{claim}\label{claim:p' 0-trivial}
    We have $(\pi_{\underline \Omega'})_! (p')_! (j')^* f_{\underline \Omega} M \cong (\pi_{\underline \Omega'})_! f_{\underline \Omega'}^* M[N].$
  \end{claim}

  \begin{proof}
    It is enough to show that for any $y' \in Y_{\underline \Omega'}$, we have $p'{}^{-1}(y') \cong \bbA^{N/2}$ and $(j'{}^* f_{\underline \Omega}^*\cL)|_{p'{}^{-1}(y')}$ is the constant sheaf. By definition, 
    \begin{equation*}
      p'{}^{-1}(y') \cong \{h_{i+1}B_r' \in G_r/B_r' : \text{$h_i^{-1}h_{i+1} \in B_rB_r'$ and $h_{i+1}^{-1}h_{i+2} \in B_r'B_r''$}\},
    \end{equation*}
    where we recall that $h_i,h_{i+2}$ are determined by $y'$ and $h_i^{-1}h_{i+2} = bb'' \in B_rB_r''$. Hence we have
    \begin{align*}
      f'{}^{-1}(y') 
      &\cong h_i(B_rB_r'/B_r' \cap bb''B_r''B_r'/B_r') \\
      &= h_ib(B_r B_r'/B_r' \cap B_r''B_r'/B_r') \\
      &\cong (B_r B_r'/B_r' \cap B_r''B_r'/B_r').
    \end{align*}
    We see that $B_r B_r'/B_r' \cap B_r''B_r'/B_r'$ is an affine space. Since $B_{0:0+} = B_r'$ by assumption, explicitly, we have
    \begin{equation*}
      B_rB_r'/B_r' \cap B_r''B_r'/B_r' \cong U_{0+:r+} \cap U_r^{\prime -} \cap U_{0+:r+}''.  
    \end{equation*}
    We have
    \begin{align*}
      N &= \dim(B_rB_r'/B_r') + \dim(B_r'B_r''/B_r'') - \dim(B_rB_r''/B_r'') \\
      &= \dim(B_rB_r'/B_r' \cap B_r''B_r'/B_r') + \dim(B_rB_r'/B_r' \cap B_r^{\prime\prime-}B_r'/B_r') \\
      &\qquad + \dim(B_r'B_r''/B_r'' \cap B_rB_r''/B_r'') + \dim(B_r'B_r''/B_r'' \cap B_r^-B_r''/B_r'') \\
      &\qquad - \dim(B_rB_r''/B_r'' \cap B_r'B_r''/B_r'') - \dim(B_rB_r''/B_r'' \cap B_r^{\prime -}B_r''/B_r'') \\
      &= 2\dim(B_rB_r'/B_r' \cap B_r''B_r'/B_r'),
    \end{align*}
    where in the last equality we use that $\dim(B_rB_r'/B_r' \cap B_r''B_r'/B_r') = \dim(B_r'B_r''/B_r'' \cap B_r^-B_r''/B_r'').$
      
    Let $y \in p'{}^{-1}(y')$ be the lift of $y'$ whose $(i+1)$th entry is $h_ibB_r'$ for some $b \in B_rB_r' \cap B_r''B_r'$. Let $yu_{i+1}$ denote the lift of $y'$ whose $(i+1)$th entry is $h_ibuB_r'$ for $u \in U_{0+:r+} \cap U_{r}^{\prime -} \cap U_{0+:r+}''$. We have
    \begin{align*}
      \frac{f_{\underline \Omega}(yu)}{f_{\underline \Omega}(y)}
      &= \beta_{B_rB_r'}(h_i^{-1}h_ibu_{i+1}) \beta_{B_r'B_r''}(u_{i+1}^{-1}b^{-1}h_i^{-1}h_{i+2}) \\
      &= \beta_{B_rB_r'}(u_{i+1}) \cdot \beta_{B_r'B_r''}(u_{i+1}^{-1}b^{-1}bb'') \\
      &= \beta_{B_rB_r'}(u_{i+1}) \cdot \beta_{B_r'B_r''}(u_{i+1}^{-1}) = 1.
    \end{align*}
    This shows that $(j'{}^* f_{\underline \Omega}^*\cL)|_{p'{}^{-1}(y')}$ is constant.
  \end{proof}

  For the rest of the proof, we focus on showing
  \begin{equation}\label{eq:Y'' vanishing}
    (\pi_{Y_{\underline \Omega}''})_! j''{}^* f_{\underline \Omega}^* M = 0.
  \end{equation}
  Set $Z_{\underline \Omega}$ to be the image of $Y_{\underline \Omega}''$ under the map $p''$ given by forgetting the $G_r/B_r^{(i+1)}$ coordinate. Then $\pi_{Y_{\underline \Omega}''}$ factors through $p''$ and to prove \eqref{eq:Y'' vanishing}, it is enough to show
  \begin{equation*}
    (p'')_! j^{\prime\prime*} f_{\underline \Omega}^* M = 0.
  \end{equation*}
  To this end, we will prove that for every $\bar{\bar y}'' \in Z_{\underline \Omega}$, there exists an intermediate space $\cY_{\bar{\bar y}''}$ satisfying
  \begin{equation*}
    p^{\prime\prime-1}(z) \stackrel{p_1''}{\longrightarrow} \cY_{\bar{\bar y}''} \stackrel{p_2''}{\longrightarrow} \{\bar{\bar y}''\}
  \end{equation*}
  such that for every $\bar y'' \in \cY_{\bar{\bar y}''}$, 
  \begin{enumerate}[label=(\roman*)]
    \item we have an isomorphism $\varphi \from A \to p^{\prime\prime-1}(\bar y'')$ for some connected algebraic group $A$, and
    \item $\varphi^*(j^{\prime\prime*} f_{Y_{\underline \Omega}^*} M)|_{p^{\prime\prime-1}(\bar y'')}$ is a nontrivial multiplicative local system on $A$.
  \end{enumerate}
  Indeed, if we can prove (i) and (ii), then we have that $(p_1'')_!(j^{\prime\prime *} f_{\underline \Omega}^* M)|_{p_1^{\prime\prime-1}(\bar y'')} = 0$ for all $\bar y'' \in p_2^{\prime\prime-1}(\bar{\bar y}'')$, which implies that the stalk of $(p'')_! j''{}^* f_{\underline \Omega}^* M$ over $\bar{\bar y}''$ is zero. 
  
  It now remains to establish (i) and (ii). Let $h_iB_r, h_{i+2}B_r''$ denote the relevant coordinates of $\bar{\bar y}''$. By assumption $h_i^{-1} h_{i+2} \in B_rB_r'B_r'' \smallsetminus B_rB_r''$. Since $B_0 = B_0'$ by assumption, we may choose coset representatives of such that $h_i^{-1} h_{i+2} = z \in U_{0+:r+}' \cap U_{0+:r+}^- \cap U_{0+:r+}^{\prime\prime-}$. By definition,
  \begin{equation*}
    h_{i+1} \in (h_iB_rB_r'/B_r' \cap h_iz B_r''B_r'/B_r').
  \end{equation*}
  Write $z = \prod_{\alpha \in \Phi^{\prime -} \cap \Phi^+ \cap \Phi^{\prime\prime+}} z_{-\alpha} \in \prod_\alpha U_{\alpha,r}$ in a(ny) fixed order. Let $a$ be minimal such that $z \in G_{a:r+} \smallsetminus G_{a+:r+}$ (note that $a > 0$ since $B_{0:0+} = B_{0:0+}'$) and let $\alpha \in \Phi^{\prime-} \cap \Phi^+ \cap \Phi^{\prime\prime+}$ have maximal height (with respect to $\Phi^+$) amongst all $\alpha$ with $z_{-\alpha} \in U_{-\alpha,a:r+} \smallsetminus U_{-\alpha,a+:r+}$. Then $[z,U_{\alpha,r-a:r+}] = T_{\alpha,r:r+}$ and we may define
  \begin{equation*}
    \cY_{\bar{\bar y}''} \colonequals h_i z U_{\alpha,r-a:r+} z^{-1} h_i^{-1} \backslash (h_iB_rB_r'/B_r' \cap h_iz B_r''B_r'/B_r').
  \end{equation*}

  \begin{claim}\label{claim:p1''}
    For any $\bar y'' \in \cY_{\bar{\bar y}''}$, we have $p_1^{\prime\prime-1}(\bar y'') \cong U_{\alpha,r-a:r+}$ (non-canonically).
  \end{claim}

  \begin{proof}
    Choose any $y'' = h_{i+1}B_r' \in p_1^{\prime\prime-1}(\bar y'')$. We may choose a representative of this coset to be $h_{i+1} = z u''$ for some $u'' \in B_r''$. To prove the claim, we need only show: if $u_\alpha \in U_{\alpha,r-a:r+}$, then $zu_\alpha u'' \in B_rB_r'/B_r' \cap B_r''B_r'/B_r'$. Indeed, we have $z u_\alpha u'' = zu_\alpha z^{-1} \cdot zu'' \in zu_\alpha z^{-1}(B_rB_r' \cap B_r''B_r')$. Since $zu_\alpha z^{-1} = [z,u_\alpha] \cdot u_\alpha \in B_r \cap B_r''$, the conclusion now follows.
  \end{proof}

  \begin{claim}\label{claim:bar y'' 0 trivial}
    Under the isomorphism $\varphi \from U_{\alpha,r-a:r+} \to p_1^{\prime\prime-1}(\bar y'')$ constructed in Claim \ref{claim:p1''}, the restriction $\varphi^*(j^{\prime\prime*}f_{\underline \Omega}^*M)|_{p^{\prime\prime-1}(\bar y'')}$ is transported to a nontrivial multiplicative local system on $U_{\alpha,ra:r+}$.
  \end{claim}

  \begin{proof}
    We retain the notation as in the proof of Claim \ref{claim:p1''}. For any $u_\alpha \in U_{\alpha,r-a:r+}$, we have 
    \begin{align*}
      \frac{f_{\underline \Omega}(y''u_\alpha)}{f_{\underline \Omega}(y'')}
      &= \frac{\beta_{B_rB_r'}(h_i^{-1} h_i z u_\alpha u'') \cdot \beta_{B_r'B_r''}(h_{i+2}^{-1}h_izu_\alpha u'')}{\beta_{B_rB_r'}(h_i^{-1} h_i z u'') \cdot \beta_{B_r'B_r''}(h_{i+2}^{-1}h_iz u'')} \\
      &= \frac{\beta_{B_rB_r'}(z u_\alpha u'') \cdot \beta_{B_r'B_r''}(u_\alpha u'')}{\beta_{B_rB_r'}(z u'') \cdot \beta_{B_r'B_r''}(u'')} \\
      &= \frac{\beta_{B_rB_r'}([z, u_\alpha] u_\alpha \cdot z u'')}{\beta_{B_rB_r'}(z u'')} = [z,u_\alpha]. \qedhere
    \end{align*}
    Since $[z,U_{\alpha,r-a:r+}] = T_{\alpha,r:r+}$, then the nontriviality of $\varphi^*(j^{\prime\prime*}f_{\underline \Omega}^*M)|_{p^{\prime\prime-1}(\bar y'')}$ holds by the $(\mft,\cL_\psi)$-equivariance of $M$ and the $(T,G)$-genericity of $\cL_\psi$.
  \end{proof}
  
  Claims \ref{claim:p1''} and \ref{claim:bar y'' 0 trivial} now establish (i) and (ii), and so the proof is now done.
\end{proof}

\subsection{Proof of Theorem \ref{thm:Borel sequence}}\label{subsec:Borel sequence proof}

We proceed by induction on $\ell(\underline \Omega) = \sum_{i=1}^n \dim(\Omega_i/B_r^{(i+1)})$. If $\ell(\underline \Omega) = 0$, then this means that $B^{(i)} = B^{(1)}$ for all $i$ and we then have an isomorphism
\begin{equation*}
  Y_{\underline \Omega} \to \widetilde G_r, \qquad (g,h_1B_r^{(1)}, h_2B_r^{(2)}, \ldots, h_{n+1}B_r^{(n+1)}) \mapsto (g,h_1B_r^{(1)}).
\end{equation*}
This establishes the base case of the induction.

Now assume $\ell(\underline \Omega) > 0$ and assume that the theorem holds for all $\underline \Omega'$ such that $\ell(\underline \Omega') < \ell(\underline \Omega)$. Note that the assumption $\ell(\underline \Omega) > 0$ immediately requires $n \geq 2$. By the same argument as in the base case of our induction, we may assume that $\dim(\Omega_i/B_r^{(i+1)}) > 0$ for all $i = 1, \ldots, n$---equivalently, we may assume that we do not have $B^{(i)} = B^{(i+1)}$ for any $i = 1, \ldots, n$. By Corollary \ref{cor:elementary expansion}, we may assume that each $\Omega_i$ is either elementary or $0$-trivial. This implies that to $\underline \Omega$ we may associate a sequence $\underline s = (s_1, s_2, \ldots, s_n)$ of elements of $W_{G_r}(T_r)$ wherein either $s_i$ is a simple reflection relative to $B^{(1)}$ or $s_i$ is the identity element.

\begin{lemma}[elementary composition]\label{lem:all simple}
  Suppose that $\underline \Omega$ is such that each $\Omega_i$ is an elementary double coset. Then for any $M \in D_{T_r}^\psi(T_r)$,
  \begin{equation*}
    \pInd_{T_r,B_r^{(1)}}^{G_r}(M) \cong \pi_{\underline \Omega!} f_{\underline \Omega}^* M [\ell(\underline \Omega)].
  \end{equation*}
\end{lemma}

\begin{proof}
  As noted above, if $\ell(\underline \Omega) = 0$, then the statement holds. We proceed by induction on $\ell(\underline \Omega)$ and assume that the statement is true for all elementary double coset sequences of length strictly less than $\ell(\underline \Omega)$. We may assume that $B^{(i)} \neq B^{(i+1)}$ for all $i = 1, \ldots, n$ so that the associated sequence $\underline s = (s_1, \ldots, s_n)$ of elements of $W_{G_r}(T_r)$ consists only of simple reflections.
  
  Since $B^{(1)} = B^{(n+1)}$ by assumption, we know that $s_1s_2\cdots s_n = 1$. In particular, this means there exists an integer $1 \leq i \leq n-1$ such that $\ell(s_1 \cdots s_i) > \ell(s_1 \cdots s_i s_{i+1})$; let $i$ be the smallest such integer. Let $\Omega' = B_r^{(1)}B_r^{(i+1)}$ and set
  \begin{equation*}
    \underline \Omega' \colonequals (\Omega', \Omega_{i+1}, \Omega_{i+2}, \ldots, \Omega_n).
  \end{equation*}
  By Corollary \ref{cor:elementary expansion}, we have
  \begin{equation*}
    \pi_{\underline \Omega!} f_{\underline \Omega}^* \cL \cong \pi_{\underline \Omega'!} f_{\underline \Omega'}^* \cL.
  \end{equation*}
  Set $\Omega'' = B_r^{(1)}B_r^{(i+2)}$. Recall that $\Omega_{i+1} = B_r^{(i+1)} B_r^{(i+2)}$ where explicitly, we have $B_r^{(i+1)} = s_1 \cdots s_i B_r^{(1)} s_i \cdots s_1$ and $B_r^{(i+2)} = s_1 \cdots s_{i+1} B_r^{(1)} s_{i+1} \cdots s_1$. Since $\ell(s_1 \cdots s_{i+1}) = \ell(s_1 \cdots s_i) - 1$ by assumption, then of course $\ell(s_1 \cdots s_i) = \ell(s_1 \cdots s_i s_{i+1}) + \ell(s_{i_1})$, and so we have by Corollary \ref{cor:composition} that
  \begin{equation*}
    \pi_{\underline \Omega'!} f_{\underline \Omega'}^* \cL \cong \pi_{\underline \Omega''!} f_{\underline \Omega''}^* \cL,
  \end{equation*}
  where 
  \begin{equation*}
    \underline \Omega'' \colonequals (\Omega'', \Omega_{i+1}^{-1}, \Omega_{i+1}, \Omega_{i+2}, \ldots, \Omega_n).
  \end{equation*}
  Now $\underline \Omega''$ has a subsequence $(\Omega_{i+1}^{-1}, \Omega_{i+1})$, where $\Omega_{i+1}$ is elementary. By Lemma \ref{lem:vanishing elementary}, we have
  \begin{equation*}
    \pi_{\underline \Omega''!} f_{\underline \Omega''}^* \cL \cong \pi_{\underline \Omega'''!} f_{\underline \Omega'''}^* \cL[2 \dim(\Omega_{i+1}/B_r^{(i+2)})],
  \end{equation*}
  where $\underline \Omega'''$ is the sequence obtained from $\underline \Omega''$ by deleting $\Omega_{i+1}^{-1}, \Omega_{i+1}$. By construction, $\ell(\underline \Omega) = \ell(\underline \Omega') = \ell(\underline \Omega'')$ and so therefore 
  \begin{align*}
    \ell(\underline \Omega''') 
    = \ell(\underline \Omega'') - 2\dim(\Omega_{i+1}/B_r^{(i+2)}) 
    &= \ell(\underline \Omega) - 2\dim(\Omega_{i+1}/B_r^{(i+2)}) < \ell(\underline \Omega).
  \end{align*}
  Therefore, by the induction hypothesis, we obtain
  \begin{align*}
    \pi_{\underline \Omega!} f_{\underline \Omega}^* \cL 
    &\cong \pInd_{T_r, B_r^{(1)}}^{G_r}(\cL)[\ell(\underline \Omega''') + 2 \dim(\Omega_{i+1}/B_r^{(i+2)})]
    = \pInd_{T_r, B_r^{(1)}}^{G_r}(\cL)[\ell(\underline \Omega)].\qedhere
  \end{align*}
\end{proof}

To complete the theorem, we induct on the number of trivial elements in $\underline s$. The base case (when there are no trivial elements in $\underline s$) is done by Lemma \ref{lem:all simple}. Now let $i$ be such that $s_i = 1$; recall that this means that $\Omega_i$ is $0$-trivial. Then by Lemma \ref{lem:vanishing 0-trivial},
\begin{equation*}
  \pi_{\underline \Omega!} f_{\underline \Omega}^* \cL \cong \pi_{\underline \Omega'!} f_{\underline \Omega'}^* \cL[N]
\end{equation*}
where $N = \dim(B_rB_r'/B_r') + \dim (B_r'B_r''/B_r'') - \dim (B_rB_r''/B_r'')$ and 
where $\underline \Omega'$ is obtained from $\underline \Omega$ by replacing the subsequence $(B_rB_r', B_r'B_r'')$ with $B_rB_r''$. The sequence $\underline s'$ corresponding to $\underline \Omega'$ is obtained from $\underline s'$ by deleting $s_i = 1$ so that $\underline s'$ now has one fewer trivial element than $\underline s$. By our induction hypothesis, we have
\begin{align*}
  \pi_{\underline \Omega'!} f_{\underline \Omega'}^* \cL[N]
  &\cong \pInd_{T_r, B_r^{(1)}}^{G_r}(\cL)[\ell(\underline \Omega') + N]
  = \pInd_{T_r,B_r^{(1)}}^{G_r}(\cL)[\ell(\underline \Omega)],
\end{align*}
which proves the theorem.

\section{Trace of Frobenius}\label{sec:sheaf function}

Let $\G$ be a connected reductive group over $F$ and assume $G = \G \otimes_F F^{\ur}$, equipped with the Frobenius endomorphism $\sigma \from G \to G$. Assume that our chosen point $\x$ in the apartment of $T$ is in the rational building $\cB(\G,F)$. Then $\sigma$ induces an endomorphism $G_r \to G_r$, which we also denote by $\sigma$; moreover, $\sigma(T_r) = T_r$. This endows $T_r$ and $G_r$ with $\FF_q$-rational structures. 

We warn the reader that it may not be the case that $\sigma(B_r) = B_r$. This is exactly the point of establishing the description in Section \ref{sec:Borel sequence} of $(T,G)$-parabolic induction in terms of sequences of Borels! Set
\begin{equation*}
  \underline B = (B, \sigma(B), \sigma^2(B), \ldots, \sigma^{n-1}(B), B),
\end{equation*}
where $n$ is a positive integer such that $\sigma^n(B) = B$. As in Section \ref{sec:Borel sequence}, let $\underline B_r$ denote the corresponding sequence of subgroups of $G_r$ and $\underline \Omega$ denote the corresponding sequence of double cosets. Recall from Section \ref{sec:Borel sequence} the associated variety $Y_{\underline \Omega}$; in this section, we write $Y \colonequals Y_{\underline \Omega}$ and $\pi_Y \colonequals \pi_{\underline \Omega}, f_Y \colonequals f_{\underline \Omega}$.

The arguments in this section follow those of Lusztig's quite closely (see \cite[Section 5]{Lus90} and \cite[Section 2.6]{MR1040575} for an exposition).

\subsection{Frobenius maps}\label{subsec:Borel sequence Frob}

Consider the map $\sigma_Y \from Y \to Y$ defined by
\begin{align*}
  \sigma_Y(g, h_1B_r, {}&{} h_2\sigma(B_r), \ldots, h_n\sigma^{n-1}(B_r), h_{n+1}B_r) \\
  &= (\sigma(g), \sigma(g^{-1}h_n)B_r, \sigma(h_1)\sigma(B_r), \sigma(h_2)\sigma^2(B_r), \ldots, \sigma(h_n)B_r).
\end{align*}
Recall that $\pi_Y \from Y \to G_r$ is the projection map onto the first coordinate; let $Y_g \colonequals \pi_Y^{-1}(g)$ for $g \in G_r$.

\begin{lemma}\label{lem:F and bar F}
  We have a commutative diagram
  \begin{equation*}
    \begin{tikzcd}
      T_r \ar{d}[left]{\sigma} & \ar{l}[above]{f_Y} Y \ar{r}{\pi} \ar{d}{\sigma_Y} & G_r \ar{d}{\sigma} \\
      T_r & \ar{l}[above]{f_Y} Y \ar{r}{\pi} & G_r
    \end{tikzcd}
  \end{equation*}
\end{lemma}

\begin{proof}
  The commutativity of the right square is clear from the definitions. We check the left square:
  \begin{align*}
    \sigma{}&{}(f_Y(g,h_1B_r, h_2\sigma(B_r), \ldots, h_n\sigma^{n-1}(B_r), h_{n+1}B_r)) \\
    &= \sigma(\beta_{\Omega_1}(h_1^{-1}h_2) \beta_{\Omega_2}(h_2^{-1} h_3) \cdots \beta_{\Omega_n}(h_n^{-1}gh_1)) \\
    f_Y{}&{}(\sigma(g,h_1B_r, h_2\sigma(B_r), \ldots, h_n\sigma^{n-1}(B_r), h_{n+1}B_r)) \\ 
    &= f_Y(\sigma(g), \sigma(g^{-1}h_n)B_r, \sigma(h_1)\sigma(B_r), \sigma(h_2)\sigma^2(B_r), \ldots, \sigma(h_n)B_r) \\
    &= \beta_{\Omega_1}(\sigma(h_n^{-1} g h_1)) \cdot \beta_{\Omega_2}(\sigma(h_1)^{-1}\sigma(h_2)) \cdots \beta_{\Omega_n}(\sigma(h_{n-1})^{-1} \sigma(h_n))
  \end{align*}
  Noting that $T_r$ is commutative, and that $\sigma \circ \Omega_i = \Omega_{i+1} \circ \sigma$ for $i = 1, \ldots, n-1$ and $\sigma \circ \Omega_n = \Omega_1 \circ \sigma$, we now see that the commutativity of the left square holds.
\end{proof}

The preceding lemma implies that if $M \in D_{T_r}(T_r)$ is such that $\sigma^* M \cong M$, then we have
\begin{equation}\label{eq:sigma pullback}
  \sigma^*(\pi_Y)_! f_Y^* M \cong (\pi_Y)_! \sigma_Y^* f_Y^* M \cong (\pi_Y)_! f_Y^* \sigma^* M
\end{equation}

\begin{lemma}
  If $g \in G_r^\sigma$, then $\sigma_Y \from Y_g \to Y_g$ is the Frobenius map for an $\FF_q$-rational structure on $Y_g$.
\end{lemma}

\begin{proof}
  For convenience, write $\underline \sF \colonequals G_r/B_r \times G_r/\sigma(B_r) \times \cdots \times G_r/\sigma^{n-1}(B_r).$ We have an embedding $j \from Y_g \hookrightarrow \underline \sF$ given by forgetting the first and last entries ($g$ and $h_{n+1}B_r$). 
  We also have that $\sigma_Y \from Y_g \to Y_g$ is the restriction to $Y_g$ of the morphism $\tilde \sigma \from \underline \sF \to \underline \sF$ given by
  \begin{align*}
    \tilde \sigma(h_1 B_r, h_2\sigma(B_r),{}&{} \ldots, h_n\sigma^{n-1}(B_r)) \\
    &= (g^{-1}\sigma(h_n)B_r, \sigma(h_1)\sigma(B_r), \ldots, \sigma(h_{n-1})\sigma^{n-1}(B_r)).
  \end{align*}
  Hence to show that $\sigma_Y$ is the Frobenius map for an $\FF_q$-rational structure on $Y_g$, it suffices to show that $\tilde \sigma$ is the Frobenius map for an $\FF_q$-rational structure on $\underline \sF$. Choose $x \in G_r$ such that $\sigma^n(x)x^{-1} = g$ (such an $x$ exists by Lang's theorem). Then consider the map $\delta \from \underline \sF \to \underline \sF$ defined as
  \begin{align*}
    \delta(h_1 B_r, h_2\sigma(B_r), {}&{} \ldots, h_n\sigma^{n-1}(B_r)) \\
    &= (x h_1 B_r, \sigma(x) h_2 \sigma(B_r), \ldots, \sigma^{n-1}(x) h_n \sigma^{n-1}(B_r)).
  \end{align*}
  The map 
  \begin{align*}
    \sigma(h_1B_r, h_2\sigma(B_r), {}&{}\ldots, h_n\sigma^{n-1}(B_r)) \\
    &= (\sigma(h_n)B_r, \sigma(h_1)\sigma(B_r), \ldots, \sigma(h_{n-1})\sigma^{n-1}(B_r))
  \end{align*} 
  is the Frobenius map for an $\FF_q$-rational structure on $\underline \sF$. We can check that we have a commutative diagram
  \begin{equation*}
    \begin{tikzcd}
      \underline \sF \ar{r}{\sigma} \ar{d}[left]{\delta} & \underline \sF \ar{d}{\delta} \\
      \underline \sF \ar{r}{\tilde \sigma} & \underline \sF
    \end{tikzcd}
  \end{equation*}
  Indeed,
  \begin{align*}
    &\delta(\sigma(h_1B_r, h_2\sigma(B_r), \ldots, h_n\sigma^{n-1}(B_r))) \\
    &\qquad\qquad = \delta(\sigma(h_n)B_r, \sigma(h_1)\sigma(B_r), \ldots, \sigma(h_{n-1})\sigma^{n-1}(B_r)) \\
    &\qquad\qquad = 
    (x\sigma(h_n)B_r, \sigma(x)\sigma(h_1) B_r, \ldots, \sigma^{n-1}(x)\sigma(h_{n-1})\sigma^{n-1}(B_r)) \\
    &\tilde \sigma(\delta(h_1B_r, h_2\sigma(B_r), \ldots, h_n\sigma^{n-1}(B_r))) \\
    &\qquad\qquad = 
    \tilde \sigma(xh_1B_r, \sigma(x)h_2\sigma(B_r), \ldots, \sigma^{n-1}(x)h_n\sigma^{n-1}(B_r)) \\
    &\qquad\qquad = (g^{-1}\sigma^n(x) \sigma(h_n) B_r, \sigma(x) \sigma(h_1) \sigma(B_r), \ldots, \sigma^{n-1}(x)\sigma(h_{n-1} \sigma^{n-1}(B_r))) \\
    &\qquad\qquad (x \sigma(h_n) B_r, \sigma(x) \sigma(h_1) \sigma(B_r), \ldots, \sigma^{n-1}(x)\sigma(h_{n-1} \sigma^{n-1}(B_r))),
  \end{align*}
  where the final equality holds since $\sigma^n(x)x^{-1} = g$ by construction. Since $\delta$ is an isomorphism, the conclusion of the lemma now follows.
\end{proof}

\subsection{An explicit formula}\label{subsec:sheaf function formula}

\begin{proposition}\label{prop:Ind Y function}
  Let $M \in D_{T_r}(T_r)$ be such that $\sigma^*M \cong M$. For any $g \in G_r^\sigma$, we have
  \begin{equation*}
    \chi_{(\pi_Y)_!f_Y^*M}(g) = \sum_{hT_r^\sigma(U_r \cap \sigma(U_r)) \in Z_g} \chi_M(\pr_{T_r}((\sigma^nh)^{-1} g h)),
  \end{equation*}
  where 
  \begin{equation}\label{eq:Z}
    Z_g \colonequals \left\{hT_r^\sigma(U_r \cap \sigma(U_r)) : 
    \begin{gathered}
      \text{$h^{-1}\sigma(h) \in \sigma(U_r)$ and} \\
      \sigma^n(h)^{-1}gh \in T_r^\sigma(U_r \cap \sigma(U_r))
    \end{gathered}\right\}.
  \end{equation}
\end{proposition}

  \begin{proof}
    Recall that by definition we have
    \begin{equation*}
      \chi_{(\pi_Y)_! f_Y^* M(g)} = \sum_i (-1)^i \Tr(\sigma_Y; H_c^i(Y_g, (f_Y^* M)_{Y_g})).
    \end{equation*}
    By the previous lemma, $\sigma_Y$ is the Frobenius of an $\FF_q$-rational structure on $Y_g$, which means we may apply the Grothendieck trace formula. We then get:
    \begin{equation*}
      \chi_{(\pi_Y)_! f_Y^* M}(g) = \sum_{y \in (Y_g)^{\sigma_Y}} \chi_M(f_Y(y)).
    \end{equation*}
    By definition, $(Y_g)^{\sigma_Y}$ consists of tuples $(h_1B_r, h_2\sigma(B_r), \ldots, h_n \sigma^{n-1}(B_r), h_{n+1}B_r)$ satisfying the following conditions:
    \begin{enumerate}[label=(\roman*)]
      \item $h_i^{-1} h_{i+1} \in \sigma^{i-1}(B_r) \sigma^i(B_r)$ for $i = 1, \ldots, n$ 
      \item $h_{n+1}^{-1} g h_1 \in B_r$
      \item $h_1 B_r = g^{-1} \sigma(h_n) B_r$
      \item $h_{i} \sigma^{i-1}(B_r) = \sigma(h_{i-1})\sigma^{i-1}(B_r)$ for $i = 2, \ldots, n+1$
    \end{enumerate}
    Suppose we are given $h_1B_r$. Then condition (iv) determines $h_i \sigma^{i-1}(B_r)$ for $i = 2, \ldots, n+1$; explicitly, we get $h_i \sigma^{i-1}(B_r) = \sigma^{i-1}(h_1) \sigma^{i-1}(B_r)$. Once this is done, then $h_1^{-1} h_2 = h_1^{-1} \sigma(h_1) \in B_r \sigma(B_r)$ automatically implies the remaining relations $h_i^{-1} h_{i+1} = \sigma^{i-1}(h_1) \sigma^i(h_1) \in \sigma^{i-1}(B_r) \sigma^i(B_r)$ for $i = 2, \ldots, n$ in condition (i). Furthermore, (iii) now is equivalent to (ii). Hence we see that $(Y_g)^{\sigma_Y}$ is isomorphic to
    \begin{equation*}
      Y_g' \colonequals \left\{h_1B_r : \begin{gathered}
        h_1^{-1} \sigma(h_1) \in B_r\sigma(B_r) \\
        \sigma^n(h_1)^{-1} g h_1 \in B_r        
      \end{gathered}
      \right\}.
    \end{equation*}
    Observe that under this isomorphism, the morphism $f_Y \from (Y_g)^{\sigma_Y} \to T_r$ becomes $f_Y' \from Y_g' \to T_r$ where
    \begin{align*}
      f_Y'(h_1B_r)
      &= f_Y(g,h_1B_r,\sigma(h_1)\sigma(B_r), \ldots, \sigma^{n-1}(h_1)\sigma^{n-1}(B_r), \sigma^n(h_1)B_r) \\
      &= \beta_{\Omega_1}(h_1^{-1}\sigma(h_1)) \cdot \beta_{\Omega_2}(\sigma(h_1)^{-1}\sigma^2(h_1)) \cdots \beta_{\Omega_n}(\sigma^{n-1}(h_1)^{-1}\sigma^n(h_1)) \cdot \beta(\sigma^n(h_1)^{-1} g h_1).
    \end{align*}
    Let $h_1 \in G_r$ be any representative of a coset in $Y_g'$. Then we have $h_1^{-1}\sigma(h_1) = u t \sigma(u')$ for some $u,u' \in U_r$ and $t \in T_r$. By Lang's theorem, we may find an $s \in T_r$ (unique up to $T_r^\sigma$-translate) such that $s^{-1}\sigma(s) = t^{-1}$. Then $h_1 s B_r = h_1B_r$ and $(h_1 s)^{-1} \sigma(h_1s) = s^{-1} ut \sigma(u') \sigma(s)$. Since $U_r$ is normalized by $T_r$, we see that $(h_1 s)^{-1} \sigma(h_1 s) \in U_r\sigma(U_r)$. Hence for $h = h_1 s$ (which represents the same coset as $h_1$), we have
    \begin{equation*}
      f_Y'(hB_r) = \beta_{\Omega_1}(h^{-1}\sigma(h)) \cdots \beta_{\Omega_n}(\sigma^{n-1}(h)^{-1} \sigma^n(h)) \cdot \beta(\sigma^n(h)^{-1} g h) = \beta(\sigma^n(h)^{-1} g h).
    \end{equation*}
    By Lemma \ref{lem:F and bar F}, for any $y \in (Y_g)^{\sigma_Y}$, we have $f_Y(y) = f_Y(\sigma_Y(y)) = \sigma(f_Y(y))$. Therefore, for any $hB_r \in Y_g'$, we have $f_Y'(hB_r) = \sigma(f_Y'(hB_r))$, which gives $\beta(\sigma^n(h)^{-1}gh) = \sigma(\beta(\sigma^n(h)^{-1} g h))$. This means $\sigma^n(h)^{-1} g h \in T_r^\sigma U_r$. Altogether, this now shows that $Y_g'$ is isomorphic to
    \begin{equation*}
      Y_g'' \colonequals \left\{
        hT_r^\sigma U_r \in G_r/T_r^\sigma U_r : 
        \begin{gathered}
          h^{-1}\sigma(h) \in U_r\sigma(U_r) \\
          \sigma^n(h)^{-1} g h \in T_r^\sigma U_r
        \end{gathered}
        \right\}
    \end{equation*}
    and that the morphism $f_Y' \from Y_g' \to T_r$ becomes
    \begin{equation*}
      f_Y''(hT_r^\sigma U_r) = \beta(\sigma^n(h)^{-1} g h) \in T_r^\sigma.
    \end{equation*}

    To finish the proof, we have left to show that $Y_g'' \cong Z_g$ \eqref{eq:Z}. Let $h \in G_r$ be a representative of a coset in $Y_g''$. By definition, $h^{-1}\sigma(h) \in U_r\sigma(U_r)$, and it is clear from here that there is a unique $(U_r \cap \sigma(U_r))$-coset of $u \in U_r$ such that $(hu)^{-1}\sigma(hu) \in \sigma(U_r)$. Choose a representative $u$ of this $(U_r \cap \sigma(U_r))$-coset; we have $(hu)^{-1} \sigma(hu) = \sigma(z)$ for some $z \in U_r$. To show $Y_g'' \cong Z_g$, we have left to show that $\sigma^n(hu)^{-1} g hu \in T_r^\sigma\sigma(U_r)$ (we already know that $\sigma^n(hu)^{-1} g hu \in T_r^\sigma U_r$ by definition of $Y_g''$). Since $\sigma(hu) = hu\sigma(z)$, we have
    \begin{equation*}
      \sigma^n(z)^{-1} \cdot \sigma^{n-1}(hu)^{-1} g hu 
      = \sigma^n(hu)^{-1} \sigma^{n-1}(hu) \cdot \sigma^{n-1}(hu)^{-1} g hu 
      = \sigma^n(hu)^{-1} g hu \in T_r^\sigma U_r.
    \end{equation*}
    Since $\sigma^n(z) \in U_r$ and $T_r$ normalizes $U_r$, this implies that $\sigma^{n-1}(hu)^{-1} g hu \in T_r^\sigma U_r$. On the other hand, we also have
    \begin{equation*}
      \sigma^n(hu)^{-1} g hu = \sigma^n(hu)^{-1} g \sigma(hu) \sigma(z)^{-1} = \sigma^n(hu)^{-1} \sigma(g) \sigma(hu) \sigma(z)^{-1} = \sigma(\sigma^{n-1}(hu)^{-1} g hu z^{-1}),
    \end{equation*}
    which we now know is in $\sigma(T_r^\sigma U_r z^{-1}) = \sigma(T_r^\sigma U_r) = T_r^\sigma \sigma(U_r)$.
  \end{proof}

We now specialize to the case that $M \in D_{T_r}^\psi(T_r)$ is simple. If $M$ is $\sigma$-equivariant, then we know from Theorem \ref{thm:Borel sequence} and Lemma \ref{lem:F and bar F} that $\pInd_{B_r}^{G_r}(M)$ is also $\sigma$-equivariant. By Theorem \ref{thm:equivalence}, we know that $\pInd_{B_r}^{G_r}(M)$ is also simple, which in particular implies that any two isomorphisms $\pInd_{B_r}^{G_r}(M) \cong \sigma^* \pInd_{B_r}^{G_r}(M)$ differ at most by a scalar factor. This line of reasoning allows us to combine Proposition \ref{prop:Ind Y function} with Theorem \ref{thm:Borel sequence}, and obtain a formula for the trace-of-Frobenius function associated to $\pInd_{B_r}^{G_r}(M)$ for any simple $M \in D_{T_r}^\psi(T_r)$.

\begin{corollary}\label{cor:Ind function}
  Assume $M \in D_{T_r}^\psi(T_r)$ is simple and $\sigma$-equivariant. Then there exists a constant $\mu$ such that for any $g \in G_r^\sigma$,
  \begin{equation*}
    \chi_{\pInd_{B_r}^{G_r}(M)}(g) = \mu \cdot \sum_{hT_r^\sigma(U_r \cap \sigma(U_r)) \in Z_g} \chi_M(\pr_{T_r}((\sigma^nh)^{-1} g h)),
  \end{equation*}
  where $Z_g$ is as in \eqref{eq:Z}.
\end{corollary}

\begin{proof}
  By Corollary \ref{cor:independence of B}(b), we have an isomorphism $\tau \from \sigma^* \pInd_{B_r}^{G_r}(M) \to \pInd_{B_r}^{G_r}(M)$. Recall from \eqref{eq:sigma pullback} that we have an isomorphism $\tau_Y \from \sigma^* (\pi_Y)_! f_Y^* M[2\dim U_r + \ell(\underline \Omega)] \to (\pi_Y)_! f_Y^* M[2\dim U_r+ \ell(\underline \Omega)]$. Since $\pInd_{B_r}^{G_r}(M) \cong (\pi_{Y})_! f_{Y}^*M[2 \dim U_r + \ell(\underline \Omega)]$ by Theorem \ref{thm:Borel sequence}, $\tau_Y$ induces another isomorphism $\tau' \from \sigma^* \pInd_{B_r}^{G_r}(M) \to \pInd_{B_r}^{G_r}(M)$. Since $\pInd_{B_r}^{G_r}(M)$ is simple (Theorem \ref{thm:equivalence}), the two morphisms $\tau$ and $\tau'$ can differ at most by a scalar multiple. It follows then that $\chi_{\pInd_{B_r}^{G_r}(M)}$ at most differs from $\chi_{(\pi_Y)_! f_Y^* M}$ by a scalar and the desired result follows from Proposition \ref{prop:Ind Y function}.
\end{proof}

\section{Comparison to parahoric Deligne--Lusztig induction}\label{sec:comparison}

We retain the set-up of Section \ref{sec:sheaf function}. If $\cL$ is a $(T,G)$-generic multiplicative local system on $T_r$ such that $\sigma^* \cL \cong \cL$, then the associated trace-of-Frobenius $\chi_\cL$ is a one-dimensional representation $\theta$ of $T_r^\sigma$. Our goal in this section is to compare $\pInd_{B_r}^{G_r}(\cL)$ to the parahoric Deligne--Lusztig induction $R_{T_r}^{G_r}(\theta)$ of the character $\theta$ in the sense of \cite{CI21-RT}. This will involve essentially every theorem proved in this paper so far:

In Section \ref{subsec:parahoric DL} (see Proposition \ref{prop:DL formula}), we prove a character formula for the parahoric Deligne--Lusztig induction of $\theta$ which has the same shape as the explicit formula for the trace of Frobenius of $(T,G)$-generic parabolic induction established in Proposition \ref{prop:Ind Y function}. Recall that this relied on the description of $\pInd_{B_r}^{G_r}(\cL)$ in terms of sequences of Borel subgroups (Theorem \ref{thm:Borel sequence}). From this, we can conclude that $\pInd_{B_r}^{G_r}(\cL)$ realizes the character of $R_{T_r}^{G_r}(\theta)$ up to a constant. To pin down this constant, we make use of the description of $\pInd_{B_r}^{G_r}(\cL)$ as the intermediate extension of a local system on the very regular locus (Theorem \ref{thm:IC vreg}).
In Section \ref{subsec:comparison}, we compare the trace of Frobenius of $\pInd_{B_r}^{G_r}(\cL)$ and the character of $R_{T_r}^{G_r}(\theta)$ at a very regular element of $G_r^\sigma$, thereby establishing that generic parabolic induction realizes parahoric Deligne--Lusztig induction (up to a sign) under the assumption that a very regular element in $T_r^\sigma$ exists (Theorem \ref{thm:comparison}). We recall arguments of \cite{CO21} proving that this assumption is satisfied under a largeness assumption on $q$ (Lemma \ref{lem:q}).

We will recall in Section \ref{subsec:parahoric DL} the definition of a parahoric Deligne--Lusztig variety $X_r$. By definition, it is stable under $\sigma^n$, where $n$ is taken as in Section \ref{sec:sheaf function} to be any positive integer such that $\sigma^n(B) = B$. We know from \cite[Theorem 1.1]{CI21-RT} that as a $G_r^\sigma$-representation, $R_{T_r}^{G_r}(\theta)$ is irreducible (up to a sign) when $\theta$ is $(T,G)$-generic (Definition \ref{def:generic character}). It is expected but not yet known that $R_{T_r}^{G_r}(\theta)$ is in fact concentrated in a single cohomological degree. This would automatically imply that $\sigma^n$ acts on the genuine representation $R_{T_r}^{G_r}(\theta)$ by a scalar, a statement essential to our proof of Proposition \ref{prop:DL formula}. \textit{A priori}, the $(G_r^\sigma \times \langle \sigma^n \rangle)$-representation $R_{T_r}^{G_r}(\theta)$ is a virtual representation; we'll say $\sigma^n$ acts on $R_{T_r}^{G_r}(\theta)$ by a scalar if \eqref{eq:scalar Frob} holds. We prove this scalar-action assertion in Theorem \ref{thm:Frob scalar}. It is worth noting that although this assertion about the action of $\sigma^n$ is purely about parahoric Deligne--Lusztig varieties and their cohomology, our proof depends on the theory of generic character sheaves on $G_r$ developed in this paper! This therefore upgrades the $G_r^\sigma$-irreducibility result of \cite{CI21-RT} to $(G_r^\sigma \times \langle \sigma^n \rangle)$-irreducibility.

\subsection{Parahoric Deligne--Lusztig induction}\label{subsec:parahoric DL}

We remind the reader of the definition of parahoric Deligne--Lusztig induction, following work of the second author and Ivanov \cite{CI21-RT}. Define
\begin{equation*}
  X_r \colonequals \{x \in G_r : x^{-1}\sigma(x) \in \sigma(U_r)\}/(U_r \cap \sigma(U_r)).
\end{equation*}
Observe that $X_r$ has an action of $G_r^\sigma \times T_r^\sigma$ given by
\begin{equation*}
  (g,t) \cdot x(U_r \cap \sigma(U_r)) = gxt(U_r \cap \sigma(U_r)).
\end{equation*}
For any $\theta \from T_r^\sigma \to \overline \QQ_\ell^\times$, we define
\begin{equation*}
  R_{T_r}^{G_r}(\theta) \colonequals \sum_{i \in \bbZ} (-1)^i H_c^i(X_r, \overline \QQ_\ell)_\theta,
\end{equation*}
where $H_c^i(X_r, \overline \QQ_\ell)_\theta$ is the subspace of $H_c^i(X_r, \overline \QQ_\ell)$ on which $\{1\} \times T_r^\sigma$ acts by $\theta$. Let $n$ a positive integer such that $\sigma^n(U_r) = U_r$. Then $X_r$ is stable under $\sigma^n$.

We recall the following genericity condition \`a la \cite{Yu01}.

\begin{definition}[$(T,G)$-generic character]\label{def:generic character}
  We say that $\theta \from T_r^\sigma \to \overline \QQ_\ell^\times$ is \textit{$(T,G)$-generic} if its restriction $\psi \colonequals \theta|_{\mft^\sigma}$ satisfies the following two conditions:
  \begin{enumerate}
    \item[ge1] for any $\alpha \in \Phi(G,T)$ and any $n \geq 1$ such that $\sigma^n(\mft_\alpha) = \mft_\alpha$, the restriction of $\psi \circ N_{\sigma}^{\sigma^n} \from \mft^{\sigma^n} \to \overline \QQ_\ell^\times$ to $(\mft_\alpha)^{\sigma^n}$ is non-trivial
    \item[ge2] the stabilizer of $\psi$ in the absolute Weyl group of $G$ is trivial
  \end{enumerate}
\end{definition}

\begin{remark}
  \begin{enumerate}
    \item The condition ge1 is equivalent to regularity in the sense of \cite{Lus04}.
    \item Let $\cL$ be a Frobenius-equivariant multiplicative local system on $T_r$ such that $\chi_\cL = \theta$. Then $\cL$ is $(T,G)$-generic if and only if $\theta$ is $(T,G)$-generic.
  \end{enumerate}
\end{remark}

\begin{theorem}[{Chan--Ivanov, \cite[Theorems 1.1 and 1.2]{CI21-RT}}]\label{thm:CI}
  If $\theta$ is a $(T,G)$-generic character of $T_r^\sigma$, then $R_{T_r}^{G_r}(\theta)$ is irreducible (up to a sign) and for any very regular element $g \in G_r^\sigma$,
  \begin{equation*}
    \Theta_{R_{T_r}^{G_r}(\theta)}(g) = \sum_{w \in W_{G_r}((T_\gamma)_r,T_r)^\sigma} \theta^w(g).
  \end{equation*}
\end{theorem}

\begin{proposition}\label{prop:DL formula sigma}
  Let $\theta$ be a character of $T_r^\sigma$. For any $m \in \bbZ_{\geq 1}$ and any $g \in G_r^\sigma$, 
  \begin{equation*}
    \tr(g \circ \sigma^{nm}; R_{T_r}^{G_r}(\theta)) = \sum_{hT_r^\sigma(U_r \cap \sigma(U_r)) \in Z_g} \theta(\pr_{T_r}((\sigma^{nm}h)^{-1} g h)),
  \end{equation*}
  where $Z_g$ is as in \eqref{eq:Z}.
\end{proposition}

\begin{proof}
  Since $U_r$ is $\sigma^{nm}$-stable for any $m \geq \bbZ_{\geq 1}$, the variety $X_r$ is then stable under $\sigma^{nm}$. Moreover, this commutes with the action of $G_r^\sigma \times T_r^\sigma$. For any $(g^{-1},t) \in G_r^\sigma \times T_r^\sigma$, the composition $(g^{-1}, t) \circ \sigma^{nm}$ is the Frobenius endomorphism for some $\FF_{q^{nm}}$-rational structure on $X_r$, and hence by the Grothendieck trace formula, we obtain:
  \begin{align*}
    \sum_i (-1)^i {}&{} \tr(g^{-1} \circ \sigma^{nm} ; H_c^i(X_r, \overline \QQ_\ell)_{\theta^{-1}}) \\
    &= \sum_{t \in T_r^\sigma} \theta(t) \sum_{i \in \bbZ} \tr((g^{-1}, t) \circ \sigma^{nm}; H_c^i(X_r, \overline \QQ_\ell)) \\
    &= \sum_{t \in T_r^\sigma} \theta(t) \cdot \#\{h(U_r \cap \sigma(U_r)) \in G_r/(U_r \cap \sigma(U_r)) : \\
    &\qquad\qquad\qquad\qquad \text{$h^{-1} \sigma(h) \in \sigma(U_r)$ and $h \in g^{-1} \sigma^{nm}(h) t(U_r \cap \sigma(U_r))$}\} \\
    &= \sum_{hT_r^\sigma(U_r \cap \sigma(U_r)) \in Z_{g,nm}} \theta(\pr_{T_r}((\sigma^{nm}h)^{-1} g h)),
  \end{align*}
  where the last equality holds since the condition $h \in g^{-1} \sigma^{nm}(h) t (U_r \cap \sigma(U_r))$ is equivalent to the condition $(\sigma^{nm}(h))^{-1} g h \in t(U_r \cap \sigma(U_r))$ and $t = \pr_{T_r}((\sigma^{nm}h)^{-1} g h).$ The proof is now complete since $R_{T_r}^{G_r}(\theta)$ is the dual of $R_{T_r}^{G_r}(\theta^{-1})$.
\end{proof}

\begin{corollary}\label{cor:Frob cm}
  Let $\theta$ be any character of $T_r^\sigma$. For any $m \in \bbZ_{\geq 1}$, there exists a constant $c_m \in \overline \QQ_\ell^\times$ such that for all $g \in G_r^\sigma$,
  \begin{equation*}
    \tr(g \circ \sigma^n; R_{T_r}^{G_r}(\theta)) = c_m \cdot \tr(g \circ \sigma^{nm}; R_{T_r}^{G_r}(\theta)).
  \end{equation*}
\end{corollary}

\begin{proof}
  Let $\cL$ be a multiplicative local system on $T_r$ such that $\chi_\cL = \theta$. Then by Proposition \ref{prop:Ind Y function}, we have that the above is equal to $\chi_{(\pi_{Y_m})_! f_{Y_m}^* \cL}$, where $Y_m$ is the variety corresponding to $\underline B = (B_r, \sigma(B_r), \ldots, \sigma^{nm}(B_r))$. By Theorem \ref{thm:Borel sequence}, we know $\pInd_{B_r}^{G_r} \cong (\pi_{Y_m})_! \circ f_{Y_m}^*$ up to a shift. On the other hand, the  genericity condition on $\theta$ implies that $\cL$ is $(T,G)$-generic, so by Theorem \ref{thm:equivalence}, $\pInd_{B_r}^{G_r}(\cL)$ is a simple perverse sheaf, and therefore the functions $\chi_{(\pi_{Y_m})_! f_Y^* \cL}$ for $m \in \bbZ_{\geq 1}$ are equal up to a constant scalar; that is, constant $c_m \in \overline \QQ_\ell^\times$ such that for all $g \in G_r^\sigma$,
  \begin{equation*}
    \sum_{hT_r^\sigma(U_r \cap \sigma(U_r)) \in Z_{g,nm}} \theta(\pr_{T_r}((\sigma^{nm}h)^{-1} g h)) = c_m \sum_{hT_r^\sigma(U_r \cap \sigma(U_r)) \in Z_{g,n}} \theta(\pr_{T_r}((\sigma^{n}h)^{-1} g h)).
  \end{equation*}
  Therefore, for any $m \in \bbZ_{\geq 1}$,
  \begin{equation}\label{eq:c_m}
    \tr(g^{-1} \circ \sigma^{nm}; R_{T_r}^{G_r}(\theta)) = c_m \cdot \tr(g^{-1} \circ \sigma^{n}; R_{T_r}^{G_r}(\theta)) \qquad \text{for all $g \in G_r^\sigma$.} \qedhere
  \end{equation}
\end{proof}

We remark that although the statement of Corollary \ref{cor:Frob cm} is simple and purely in terms of parahoric Deligne--Lusztig varieties, the proof relied on essentially all the main theorems proved thus far: that $\pInd_{B_r}^{G_r}$ sends simple perverse sheaves to simple perverse sheaves (Theorem \ref{thm:equivalence}), that $\pInd_{B_r}^{G_r}$ has an alternative description using certain sequences of Borel subgroups $\underline B$ (Theorem \ref{thm:Borel sequence}), and that when $\underline B = (B, \sigma(B), \ldots, \sigma^n(B))$, we have an explicit description of the associated trace-of-Frobenius function (Proposition \ref{prop:Ind Y function}). An elementary argument allows us to promote Corollary \ref{cor:Frob cm} to the following strengthening of Theorem \ref{thm:CI}.

\begin{theorem}\label{thm:Frob scalar}
  If $\theta$ is a $(T,G)$-generic character of $T_r^\sigma$, then there exists a scalar $c \in \overline \QQ_\ell^\times$ such that for any $g \in G_r^\sigma$,
  \begin{equation}\label{eq:scalar Frob}
    \tr(g \circ \sigma^n; R_{T_r}^{G_r}(\theta)) = c \cdot \tr(g; R_{T_r}^{G_r}(\theta)).
  \end{equation}
  Hence $R_{T_r}^{G_r}(\theta)$ is irreducible (up to a sign) as a representation of $G_r^\sigma \times \langle \sigma^n \rangle$.
\end{theorem}

\begin{proof}
  For $\lambda \in \overline \QQ_\ell^\times$, let $\rho_\lambda$ denote the virtual $G_r^\sigma$-representation on which $\sigma^n$ acts by multiplication by $\lambda$. We therefore have
    \begin{equation*}
      R_{T_r}^{G_r}(\theta) = \sum_i \rho_{\lambda_i}
    \end{equation*}
    for some pairwise distinct $\lambda_i \in \overline \QQ_\ell^\times$. Note that the above sum is finite. We then have
    \begin{align*}
      \tr(g \circ \sigma^{nm}; R_{T_r}^{G_r}(\theta)) = \sum_i \lambda_i^m \tr(g;\rho_{\lambda_i}), \qquad \text{for all $m \in \bbZ_{\geq 1}$.}
    \end{align*}
    By Equation \eqref{eq:c_m}, there exists a constant $c_m$ such that for all $g \in G_r^\sigma$,
    \begin{equation*}
      \sum_i \lambda_i^m \tr(g;\rho_{\lambda_i}) = c_m \cdot \sum_i \lambda_i \tr(g;\rho_{\lambda_i}).
    \end{equation*}

    Suppose that $\sum_i \lambda_i \tr(g;\rho_{\lambda_i}) = 0$. By the above, we then see that 
    \begin{equation*}
      \sum_i \lambda_i^m \tr(g;\rho_{\lambda_i}) = 0,
    \end{equation*}
    and therefore
    \begin{equation*}
      \sum_i \frac{\lambda_i \tr(g;\rho_{\lambda_i})}{1 - \lambda_i t} = 0,
    \end{equation*}
    which therefore implies $\tr(g;\rho_{\lambda_i}) = 0$ for all $i$. In particular, we see that $\sum_i \lambda_i \tr(e;\rho_{\lambda_i}) \neq 0$.

    Let $g \in G_r^\sigma \smallsetminus \{e\}$ be such that $\sum_i \lambda_i \tr(g;\rho_{\lambda_i}) \neq 0$. Then we have
    \begin{equation*}
      \frac{\sum_i \lambda_i^m \tr(g; \rho_{\lambda_i})}{\sum_i \lambda_i \tr(g;\rho_{\lambda_i})} = 
      \frac{\sum_i \lambda_i^m \tr(e;\rho_{\lambda_i})}{\sum_i \lambda_i \tr(e;\rho_{\lambda_i})},
    \end{equation*}
    which implies
    \begin{equation*}
      \frac{1}{\sum_j \lambda_j \tr(g;\rho_{\lambda_j})} \cdot \sum_i \frac{\lambda_i \tr(g;\rho_{\lambda_i})}{1 - \lambda_i t} = 
      \frac{1}{\sum_j \lambda_j \tr(e;\rho_{\lambda_j})} \cdot \sum_i \frac{\lambda_i \tr(e;\rho_{\lambda_i})}{1 - \lambda_i t}.
    \end{equation*}
    Since the $\lambda_i$ are all distinct, the above equality of rational functions gives
    \begin{equation*}
      \tr(g;\rho_{\lambda_i}) = 
      \frac{\tr(e;\rho_{\lambda_i}) \cdot \sum_j \lambda_j \tr(g;\rho_{\lambda_j})}{\sum_j \lambda_j \tr(e;\rho_{\lambda_j})} \qquad \text{for all $i$},
    \end{equation*}
    and therefore
    \begin{equation*}
      \tr(g;\rho_{\lambda_i}) = \frac{\tr(e;\rho_{\lambda_i})}{\tr(e;\rho_{\lambda_1})} \cdot \tr(g;\rho_{\lambda_1}) \qquad \text{for all $i$}.
    \end{equation*}
    
    Combining this with the vanishing statement, we've now shown
    \begin{equation*}
      \tr(g;\rho_{\lambda_i}) = d_i \cdot \tr(g;\rho_{\lambda_1})
    \end{equation*}
    for some $d_i \in \overline \QQ_\ell^\times$. Hence 
    \begin{equation*}
      \tr(g \circ \sigma^n; R_{T_r}^{G_r}(\theta)) = \sum_i \lambda_i \tr(g;\rho_{\lambda_i}) = \sum_i \lambda_i d_i \tr(g;\rho_{\lambda_1}) = \frac{\sum_i \lambda_i}{\sum_i \lambda_i d_i} \tr(g;R_{T_r}^{G_r}(\theta)),
    \end{equation*}
    which establishes \eqref{eq:scalar Frob}. The final assertion now follows from Theorem \ref{thm:CI}.
\end{proof}

From Theorem \ref{thm:Frob scalar} and Proposition \ref{prop:DL formula}, we may now establish that up to a scalar, the character of $R_{T_r}^{G_r}(\theta)$ can be expressed using the same formula as in Proposition \ref{prop:Ind Y function}:

\begin{proposition}\label{prop:DL formula}
  If $\theta$ is a $(T,G)$-generic character of $T_r^\sigma$, then there exists a constant $\lambda \in \overline \QQ_\ell^\times$ such that for any $g \in G_r^\sigma$, we have
  \begin{equation*}
    \Theta_{R_{T_r}^{G_r}(\theta)}(g) = \lambda \cdot \sum_{hT_r^\sigma(U_r \cap \sigma(U_r)) \in Z_g} \theta(\pr_{T_r}((\sigma^nh)^{-1} g h)),
  \end{equation*}
  where $Z_g$ is as in \eqref{eq:Z}.
\end{proposition}

\begin{remark}
  In the $r=0$ setting, Proposition \ref{prop:DL formula} follows from Proposition \ref{prop:DL formula sigma} together with the classical fact that the cohomology groups $H_c^i(X_0, \overline \QQ_\ell)_\theta$ vanish outside the middle degree. The method of proof presented here gives an alternate argument in the $r=0$ setting, relying ``only'' on the weaker statement \eqref{eq:scalar Frob} about the action of Frobenius (Theorem \ref{thm:Frob scalar}). 
\end{remark}

\subsection{Comparison}\label{subsec:comparison}

We are now ready to establish the compatibility between $(T,G)$-generic parabolic induction and parabolic Deligne--Lusztig induction. At this point, comparing the two formulae given in Corollary \ref{cor:Ind function} and Proposition \ref{prop:DL formula}, we see that we already know this compatibility up to a constant and that it remains only to determine this constant. Hence we need only to compare the two sides at a single conveniently chosen element of $G_r^\sigma$. 

The locus of very regular elements provides a natural choice for this comparison. On the representation theoretic side, from \cite[Theorem 1.2]{CI21-RT} (see Theorem \ref{thm:CI} of this paper for the statement), we know that $\Theta_{R_{T_r}^{G_r}}(\theta)$ takes a simple shape on such elements. On the geometric side, from Theorem \ref{thm:IC vreg}, we know that $(T,G)$-generic parabolic induction is given by the intermediate extension of a local system from the very regular locus, and as such $\chi_{\pInd_{B_r}^{G_r}(\cL[\dim T_r])}$ has an analogously simple shape on these elements. We follow this line of reasoning and establish our desired comparison theorem under the (mild) assumption of non-emptyness of $(T_r^\sigma)_{\vreg}$.

\begin{theorem}\label{thm:comparison}
  Let $r>0$ and assume that $(T_r^\sigma)_{\vreg} \neq \varnothing$. Let $\cL$ be any generic rank-1 local system such that $\sigma^* \cL \cong \cL$. Then for all $g \in G_r^\sigma$, we have
  \begin{equation*}
    \chi_{\pInd_{B_r}^{G_r}(\cL[\dim T_r])}(g) = (-1)^{\dim G_r} \cdot \Theta_{R_{T_r}^{G_r}(\chi_\cL)}(g).
  \end{equation*}
  In particular, the class function $\chi_{\pInd_{B_r}^{G_r}(\cL[\dim T_r])}$ is the character of an irreducible virtual $G_r^\sigma$-representation.
\end{theorem}

\begin{proof}
  Recall from Corollary \ref{cor:Ind function} that there exists a constant $\mu \in \overline \QQ_\ell^\times$ such that 
  \begin{equation*}
    \chi_{\pInd_{B_r}^{G_r}(\cL[\dim T_r])}
    = \mu \cdot \chi_{(\pi_Y)_! f_Y^* \cL}.
  \end{equation*}
  Let $\theta = \chi_\cL$. By Proposition \ref{prop:Ind Y function} and Proposition \ref{prop:DL formula}, there is a constant $\lambda \in \overline \QQ_\ell^\times$ such that 
  \begin{equation*}
    \chi_{(\pi_Y)_! f_Y^* \cL} = \lambda^{-1} \cdot \Theta_{R_{T_r}^{G_r}(\theta)},
  \end{equation*}
  and so therefore we have the following identity of functions on $G_r^\sigma$
  \begin{equation*}
    \chi_{\pInd_{B_r}^{G_r}(\cL)} = \frac{\mu}{\lambda} \cdot \Theta_{R_{T_r}^{G_r}(\theta)}.
  \end{equation*}
  To determine $\mu/\lambda$, we need only compute both sides at a(ny) convenient element $g \in G_r^\sigma$.

  Let $g \in G_r^\sigma$ be any very regular element. By Theorem \ref{thm:CI}, we know 
  \begin{equation*}
    \Theta_{R_{T_r}^{G_r}(\theta)}(g) = \sum_{w \in W_{G_r}(T_r)^\sigma}\theta^w(g).
  \end{equation*}
  On the other hand, by Theorem \ref{thm:IC vreg} and Lemma \ref{lem:tilde Gvreg}, we have
  \begin{align*}
    \chi_{\pInd_{B_r}^{G_r}(\cL[\dim T_r])}(g) 
    &= \chi_{(\pi_{\vreg})_!f_{\vreg}^*\cL_{\vreg}[\dim G_r]}(g) \\
    &= (-1)^{\dim G_r}  \cdot \sum_{w \in W_{G_r}(T_r)^\sigma}\theta^w(g).
  \end{align*}
  In order to conclude that $\mu/\lambda = (-1)^{\dim G_r}$, we need to make sure that $\sum_{w \in W_{G_r}(T_r)^\sigma}\theta^w(g) \neq 0$ for some very regular element $g$. This follows from \cite[Lemma 9.6]{CO21}. The proof is a simple trick, so we provide it for the convenience of the reader: for any $t_+ \in T_r^\sigma$, the element $gt_+$ is still very regular. By the genericity condition on $\theta$, the characters $\theta^w|_{T_{0+:r+}^\sigma}$ are all distinct. Hence we have 
  \begin{equation*}
    \langle \sum_{w \in W_{G_r}(T_r)^\sigma} \theta^w(g) \cdot \theta^w|_{T_{0+:r+}^\sigma}, \theta|_{T_{0+:r+}^\sigma} \rangle_{T_{0+:r+}^\sigma} = \theta(g) \neq 0,
  \end{equation*}
  which in particular means that $\sum_{w \in W_{G_r}(T_r)^\sigma} \theta^w(g) \cdot \theta^w|_{T_{0+:r+}^\sigma}$ is not identically zero, finishing the proof.
\end{proof}


\begin{remark}
  Note that the proof of Theorem \ref{thm:comparison} relies on the positivity of $r$ to obtain the nonvanishing of $\Theta_{R_{T_r}^{G_r}(\theta)}$ at any very regular element in $T_r^\sigma$. 
  
  In the case $r = 0$, we may run the same argument to obtain a comparison between character sheaves and Deligne--Lusztig induction for $\theta$ in general position, but it is no longer true that $\Theta_{R_{T_0}^{G_0}(\theta)}$ is nonzero at every regular element of $T_0^\sigma$. Hence we would need to assume a stronger condition: the existence of a regular element of $T_0^\sigma$ for which $\Theta_{R_{T_0}^{G_0}(\theta)}$ is nonzero. In \cite[Lemma 4.3]{CO23}, it is shown that this assumption is guaranteed by requiring
  \begin{equation*}
    \frac{|T_0^\sigma|}{|T_0^\sigma| - |(T_0^\sigma)_{\reg}|} > 2 \cdot |W_{G_0^\sigma}(T_0)|.
  \end{equation*}

  We remark that this is a different strategy to Lusztig's $r = 0$ comparison, where he pins down the constant by making a comparison on a regular unipotent element. This requires an assumption on $q$ and also takes some effort to calculate on the cohomology side (but is trivial on the sheaf side). 
\end{remark}

We make some remarks on the existence of very regular elements in $T_r^\sigma$. Following Kaletha \cite[Section 3.4]{Kal19} (see also \cite[Lemma 5.6]{CO21}), the unramified torus $\T$ of $\G$ transfers to an unramified torus $\T^*$ of the quasisplit inner form $\G^*$ such that the associated point $\bar \x^*$ of $\cB^{\red}(\G^*,F)$ corresponds to a Chevalley valuation of $\G^*$. Moreover, by \cite[Lemma 5.7]{CO21}, this induces an isomorphism $T_r^\sigma \cong T_r^{* \sigma}$ identifying the respective sets of very regular elements $(T_r^\sigma)_\vreg \cong (T_r^{* \sigma})_\vreg$. But now, as explained in the proof of \cite[Proposition 5.8]{CO21}, the locus $(T_r^{* \sigma})_{\vreg}$ is exactly equal to the preimage under $T_r^{* \sigma} \to T_0^{*\sigma}$ of the regular elements of $T_0^{*\sigma}$ in $G_0^{* \sigma}$. Hence the existence of certain very regular elements in $G_r^\sigma$ is equivalent to the existence of certain regular semisimple elements in $G_0^{* \sigma}$. In particular, we may conclude (see \cite[Proposition 5.8]{CO21}):

\begin{lemma}\label{lem:q}
  There exists a constant $C$ depending only on the absolute rank of $\G$ such that $(T_r^\sigma)_{\vreg} \neq \varnothing$ whenever $q > C$.
\end{lemma}

\newpage


\providecommand{\bysame}{\leavevmode\hbox to3em{\hrulefill}\thinspace}
\providecommand{\MR}{\relax\ifhmode\unskip\space\fi MR }
\providecommand{\MRhref}[2]{%
  \href{http://www.ams.org/mathscinet-getitem?mr=#1}{#2}
}
\providecommand{\href}[2]{#2}

\end{document}